\documentclass{amsart}
\usepackage[colorlinks=true, allcolors=blue]{hyperref}
\usepackage{amsmath}
\usepackage{paralist}
\usepackage{amssymb}
\usepackage{amsthm}
\usepackage{amscd}
\usepackage{graphicx,mathrsfs}
\usepackage{fontenc}
\usepackage{enumitem}

\usepackage[paperwidth=7in, paperheight=10in, margin=.75in]{geometry}
\numberwithin{equation}{section}

\newtheorem{Thm}{Theorem}[section]
\newtheorem{Lem}{Lemma}[section]
\newtheorem{Cor}{Corollary}[section]
\newtheorem{Prop}{Proposition}[section]

\newtheorem{Def}{Definition}[section]
\newtheorem{Rem}{Remark}[section]

\begin{document}
\sloppy
\allowdisplaybreaks[4]
\title[Linear parabolic equations on Carnot tori]{Regularity results for linear parabolic equations on Carnot tori via mollifier kernel construction}

\author{Yiming Jiang, Yawei Wei, Yiyun Yang}
\address{School of Mathematical Sciences and LPMC\\ Nankai University\\ Tianjin 300071 China}
\email{ymjiangnk@nankai.edu.cn}
\address{School of Mathematical Sciences and LPMC\\ Nankai University\\ Tianjin 300071 China}
\email{weiyawei@nankai.edu.cn}
\address{School of Mathematical Sciences\\ Nankai University\\ Tianjin 300071 China}
\email{1120210035@mail.nankai.edu.cn}
%\thanks{Acknowledgements: This work is supported by the NSFC under the grands 12271269 and supported by the Fundamental Research Funds for the Central Universities.}

\begin{abstract}
This paper first proves the existence, uniqueness and regularity of the solution to a class of linear backward parabolic equations on Carnot tori, namely the periodic linear parabolic equation on Carnot groups. Such groups are non-commutative and typical examples of sub-Riemannian manifolds. Moreover, we apply the results for this equation to its dual equation (i.e., the Fokker-Planck-Kolmogorov equation in the general form), and derive the existence, uniqueness and regularity of its weak solution. To obtain the regularity results for solutions to the linear parabolic equation and its dual equation, firstly, we construct several families of mollifiers adapted respectively to the H\"{o}rmander vector fields generating Carnot groups, Carnot tori and dual spaces of non-isotropic H\"{o}lder spaces; secondly, we use the theory of singular integral operators to establish stronger a priori regularity for the solutions.
\end{abstract}

\maketitle
\section{Introduction}
%\noindent
In this paper, we first investigate the existence and uniqueness of the solution to the following linear degenerate backward parabolic equation
\begin{equation}\label{LHJB}
\begin{cases}
-\partial_t z-\Delta_{\mathcal{X}} z+b(t,x) \cdot D_{\mathcal{X}} z=f(t, x),& \text {in }[0, T) \times \mathbb{T}_{\mathbb{G}}, \\
z(T, x)=z_T(x),& \text {in }\mathbb{T}_{\mathbb{G}},
\end{cases}
\end{equation}
and establish Schauder estimates as well as H\"{o}lder continuity estimates for the solution. In addition, we also consider the dual form of equation \eqref{LHJB}, namely the general forward Fokker-Planck-Kolmogorov (FPK in short) equation:
\begin{equation}\label{general FPK}
\begin{cases}
\partial_t\rho-\Delta_\mathcal{X}\rho-\operatorname{div}_{\mathcal{X}}(\rho b)=\upsilon, & \text{ in } [0,T] \times \mathbb{T}_{\mathbb{G}}, \\
\rho(0)=\rho_{0}, & \text{ in } \mathbb{T}_{\mathbb{G}}.
\end{cases}
\end{equation}
We then prove the existence, uniqueness and regularity estimates for a class of weak solutions in the sense of distributions to the equation \eqref{general FPK}. Here, $\mathbb{T}_{\mathbb{G}}$ denotes the torus in the homogeneous Carnot group $\mathbb{G}=\left(\mathbb{R}^n,\circ\right)$, i.e. $\mathbb{T}_{\mathbb{G}}:=\mathbb{G}/\mathbb{Z}^{n}$. Let $\mathcal{X}=\{X_1,X_2,\ldots,X_{n_1}\},n_1< n$ denote the Jacobian generators of $\mathbb{G}$, which is a typical class of vector fields with an anisotropic structure satisfying the H\"{o}rmander condition, i.e.
\begin{equation*}
\operatorname{rank}(\operatorname{Lie}\{X_1,\ldots,X_{n_1}\}(x))=n,\,\text{ for any }x\in\mathbb{R}^n,
\end{equation*}
where $\operatorname{Lie}\{X_1,\ldots,X_{n_1}\}(x)$ denotes the Lie algebra induced by the given vector fields. See Subsection \ref{Subsec_2.1} for further details on H\"{o}rmander vector fields, Carnot group and Carnot torus. For any function $f:\mathbb{T}_{\mathbb{G}} \to \mathbb{R}$, the subgradient and the hypoelliptic operator associated with $\mathcal{X}$ are defined respectively as
$$
D_{\mathcal{X}}f:=(X_1f,\ldots,X_{n_1}f)^{\top},\quad\Delta_{\mathcal{X}}f:=\sum_{i=1}^{n_1}X_i^2f.
$$
While for any vector-valued function $g:\mathbb{T}_{\mathbb{G}} \to \mathbb{R}^{n_1}$, the corresponding divergence is defined as
$$
\operatorname{div}_{\mathcal{X}}g=\sum_{i=1}^{n_1}X_ig.
$$
Given any $T>0$, the coefficients $b(t,x)=(b_1(t,x),\ldots,b_{n_1}(t,x))^{\top}$ and $f(t,x)$ are continuous functions on $[0,T]\times\mathbb{T}_{\mathbb{G}}$, $\upsilon(t)$ and $\rho_{0}$ are generalized functions on $\mathbb{T}_{\mathbb{G}}$. Moreover, $b_i(t,\cdot)$, $f(t,\cdot)$ and $z_T$ belong to some non-isotropic H\"{o}lder spaces (see Subsection \ref{Subsec_2.1} for definitions).

H\"{o}rmander's groundbreaking work \cite{67Ho} laid the cornerstone for the study of subelliptic differential equations associated with H\"{o}rmander vector fields, a research area that has garnered escalating scholarly attention over the past few decades. Presently, numerous scholars have investigated the Schauder theory for second-order degenerate equations associated with H\"{o}rmander operators. Roughly speaking, Schauder estimates assert that if both the coefficients of the operator $H$ and the inhomogeneous term $f$ in the equation $Hu=f$ satisfy H\"{o}lder continuity, then all derivatives of $u$ possess H\"{o}lder continuity. In recent years, linear H\"{o}rmander operators with variable coefficients have also garnered attention, such as the operator:
\begin{equation}\label{Def_general Hormander elliptic operator}
\sum_{i,j=1}^{m}a_{i,j}(x)X_i X_j+\sum_{i=1}^{m}b_i(x)X_i+c(x)
\end{equation}
with $(X_1,\ldots,X_m),m<n$ being the family of H\"{o}rmander vector fields, or the operator with drift:
\begin{equation}\label{Def_general Hormander parabolic operator}
-\sum_{i,j=1}^{m}a_{i,j}(x)X_i X_j+a_0(x)X_0+\sum_{i=1}^{m}b_i(x)X_i+c(x)
\end{equation}
with $(X_0,X_1,\ldots,X_m),m+1<n$ being the family of H\"{o}rmander vector fields, where the coefficient matrix $(a_{i,j})_{i,j}$ is symmetric and uniformly positive definite. In 1992, Xu \cite{92Xu} established local a-priori Schauder estimates for operators of type \eqref{Def_general Hormander elliptic operator}, and the proof relied on an additional assumption regarding the structure of the Lie algebra generated by the family of H\"{o}rmander vector fields. Capogna and Han \cite{03CH} proved pointwise Schauder estimates for operators of type \eqref{Def_general Hormander elliptic operator} with $b_i,c\equiv0$ on Carnot groups. In \cite{07BB}, Bramanti and Brandolini demonstrated local a-priori Schauder estimates for the parabolic counterpart of the aforementioned operators, i.e., operators of type \eqref{Def_general Hormander parabolic operator} with $a_0X_0=\partial_t$. In this study, the analysis of the H\"{o}rmander vector fields did not presuppose any group structure. %For operators of type \eqref{Def_general Hormander parabolic operator} with more general drift $X_0$, Guti\'{e}rrez and Lanconelli \cite{09GL} derived the related local Schauder estimates under the framework of Carnot groups with $a_0\equiv1,b_i,c\equiv0$. Subsequently, Bramanti and Zhu \cite{11BZ} extended this result to the case of general H\"{o}rmander vector fields, where $a_0$ is a variable coefficient. 
Furthermore, the Schauder theory for H\"{o}rmander degenerate equations with less regular inhomogeneous terms and coefficients also constitutes an important research topic. In our previous work \cite{25JWY}, we investigated global a-priori Schauder estimates for solutions to the Cauchy problem of H\"{o}rmander operators of type \eqref{Def_general Hormander parabolic operator} with $a_0X_0=\partial_t$ and $(a_{i,j})=I$ on Carnot groups, focusing on cases where the inhomogeneous terms and coefficients are rough in time. This paper is a further development of \cite{25JWY}. We likewise consider the linear degenerate parabolic equations with time-rough coefficients in \eqref{LHJB}, and establish the well-posedness, Schauder estimates and H\"{o}lder continuity estimates of its solutions by using the regularity results proved in \cite{25JWY}.

FPK equations are a class of second-order PDEs that characterize the evolution of measures, and they have an inherent connection with Markov diffusion processes (cf. \cite{08St}). For example, for classical Markov diffusion models such as standard Brownian motion and $n$-dimensional Ornstein-Uhlenbeck process, their transition probabilities satisfy the corresponding parabolic FPK equations, while the invariant measures of the diffusion processes are exactly the solutions to the stationary FPK equations. The study of FPK equations can be reduced to the analysis of their dual PDEs in the sense of distributions, as described in \cite{10Ev,01GT,04RR}. The researches focus on key issues including the existence, uniqueness, regularity and density of solutions, which possess both probabilistic and analytical values. Another active research direction is investigating FPK equations on abstract spaces (e.g., Riemannian manifolds), where the geometric properties exert an influence on the analysis (cf. \cite{14BGL}). In this paper, we investigate the general form of the FPK equation on Carnot tori (see \eqref{general FPK}) by employing the idea of duality. In particular, when $\upsilon=0$ in \eqref{general FPK}, its solution can characterize the transition probability measure of the following diffusion process:
\begin{equation}\label{periodic Carnot diffusion}
\begin{cases}
d Z_{t}=-\sum_{k=1}^{n_1}b_{k}(t,Z_t)X_k(Z_t) dt+\sqrt{2}\sum_{k=1}^{n_1}X_k(Z_t)d B_{t}^{k}, \\
Z_{0}=x=(x^{1},\ldots,x^{n})^{\top}, \text{ in } \mathbb{T}_{\mathbb{G}},
\end{cases}
\end{equation}
where $X_k(Z_t)d B_{t}^{k}$ is the Stratonovich-type and $B_{t}=(B_t^{1},\ldots,B_t^{n_1})^{\top}$ is standard $n_1$-dimensional Brownian motion. This process is constrained to move periodically follow the horizontal curves with respect to the family of vector fields $\mathcal{X}$ generating the Carnot group $\mathbb{G}$. Notably, $\mathbb{G}$ is an example of sub-Riemannian manifolds (see \cite{22CCL}). By using the Schauder theory for the equation \eqref{LHJB}, this paper studies the well-posedness and regularity estimates of the weak solution in the framework of the dual spaces of non-isotropic H\"{o}lder spaces.

\noindent
\textbf{Notations.} Throughout this paper, we assume that $T > 0$ is a fixed finite time, and the positive constant, typically denoted by $C$, may change values from line to line. By default, we denote $c>0$ as a constant depending on $\mathbb{G}$. In addition, we shall use the notation $c_{f_1,\ldots,f_k}$ if $c$ also depends on $f_1,\ldots,f_k$.

For any $k\in\mathbb{N}$ and multi-index $I=\left(i_1,\ldots,i_k\right), i_j \in \{1,\ldots,n_1\},j\in\{1,\ldots,k\}$ with the length $|I|:=\sum_{j=1}^{n}\alpha_{i_j}$, we define
\begin{equation*}
X_I:=
\begin{cases}
X_{i_1} \cdots X_{i_k}, & \mbox{if } \dim(I)\geq1, \\
\mathrm{Id}, & \mbox{if } \dim(I)=0,
\end{cases}
\end{equation*}
where $\dim(I)$ is the dimension of vector $I$, i.e. $\dim(I)=k$. Here, $\alpha_{i_j}$ represents the homogeneous degree of $X_{i_j}$. It should be noted that $\alpha_{i_j}=1$ if $i_j\in \{1,\ldots,n_1\}$.

For any interval $D \subseteq\mathbb{R}$ and $n$-dimensional open domain $\Omega$, we define the space
\begin{equation*}
B\left(D;C_{\mathcal{X}}^{k+\alpha}\left(\Omega\right)\right):=\left\{\phi:D\to C_{\mathcal{X}}^{k+\alpha}\left(\Omega\right)~\bigg|~ \sup_{t\in D}\|\phi(t,\cdot)\|_{C_{\mathcal{X}}^{k+\alpha}\left(\Omega\right)}<+\infty\right\}
\end{equation*}  
for any $k\in\mathbb{N}$, $\alpha\in(0,1]$, where the H\"{o}lder space $C_{\mathcal{X}}^{k+\alpha}$ and the H\"{o}lder norm $\|\cdot\|_{C_{\mathcal{X}}^{k+\alpha}}$ are defined below in \eqref{non-isotropic Holder spaces_Def} and \eqref{Holder norm} respectively. We also define the function space
$$
C_{\mathcal{X}}^{1,2}(D\times\Omega):=\left\{u:D\times\Omega\to\mathbb{R}~|~\partial_{t}u, X_Iu\in C(D\times\Omega),\,\forall\,|I|\leq 2 \right\}.
$$

Let $d_{cc}$ denote the Carnot-Carath\'{e}odory distance induced by the vector fields $\mathcal{X}$ (see Definition \ref{Def_CC distance} below). For $x \in \Omega$, we introduce the $d_{cc}$-ball as
$$
B_\delta(x)=\{y \in \Omega:d_{cc}(x,y)<\delta\}.
$$

In the following, we present the main results of this paper. 

To prove the regularity results for the solutions to equations \eqref{LHJB} and \eqref{general FPK}, we need to construct several types of suitable mollifiers. Specifically, first we use the fundamental solution $\Gamma_0\left(t-s,y^{-1}\circ x\right)$ for the heat operator $\mathcal{H}=\partial_t-\Delta_{\mathcal{X}}$ (see \eqref{heat operator} below) to build a family of mollifiers adapted to the family of H\"{o}rmander vector fields $\mathcal{X}$.
\begin{Prop}[Mollifiers adapted to $\mathcal{X}$]\label{Prop_Mollifiers_Holder}
For each $\varepsilon>0$, set
\begin{equation}\label{mollifiers_phi_varepsilon}
\phi_{\varepsilon}(t,x):=\frac{1}{\varepsilon}\Gamma_0\left(\varepsilon,x\right)\varphi\left(\frac{t}{\varepsilon}\right),\,(t,x)\in\mathbb{R}\times\mathbb{R}^n,
\end{equation}
where $\varphi \in C_0^{\infty}(\mathbb{R})$ is a standard mollified function such that $\int_{\mathbb{R}}\varphi(t)dt=1$.
For any function $f:\mathbb{R}\times\mathbb{R}^{n}\to\mathbb{R}$ satisfying
\begin{equation}\label{f_growth cond.}
|f(t,x)| \leq M \exp\left(\mu \left\|x\right\|^2\right),\,(t,x)\in\mathbb{R}\times\mathbb{R}^{n}
\end{equation}
for some constants $M>0$ and $\mu \in [0,(2c)^{-1})$ with $c$ given in \eqref{Gamma_0 Lie derivatives GE}, define the function
\begin{equation*}
f_{\varepsilon}(t,x):=\int_{\mathbb{R}\times\mathbb{R}^{n}}\phi_{\varepsilon}(t-s,y^{-1}\circ x)f(s,y)d y ds,\,(t,x)\in\mathbb{R}\times\mathbb{R}^n,
\end{equation*}
where $\phi_{\varepsilon}$ is the function defined in \eqref{mollifiers_phi_varepsilon}. Then for any $\varepsilon\in(0,1]$, $f_{\varepsilon}(t,x) \in C^{\infty}\left(\mathbb{R}\times\mathbb{R}^{n}\right)$ and the following conclusions hold.
\begin{enumerate}[label=(\arabic*), leftmargin=*, itemindent=1.7em]
\item If $f(t,x) \in C\left(\mathbb{R}\times\mathbb{R}^{n}\right)$ and satisfies
\begin{equation}\label{f_ctn. cond.}
\left|f(t,x)-f(s,y)\right| \leq \omega_{f}(|t-s|)+c_f d_{cc}(x,y)^\alpha
\end{equation}
for some $\alpha\in(0,1]$, where the continuous modulus function $\omega_{f}(\cdot)$ and the constant $c_f>0$ depend only on $f$. Then
\begin{equation}\label{f_varepsilon->f}
\lim\limits_{\varepsilon\to0}\left\|f_{\varepsilon}-f\right\|_{L^{\infty}\left(\mathbb{R}\times\mathbb{R}^{n}\right)}=0.
\end{equation}
\label{mollifiers_f_varepsilon->f}
\item For any $\alpha \in (0,1)$ and $k \in \mathbb{N}$, if $f(t,x) \in C_{\mathcal{X}}^{\frac{\alpha}{2},k+\alpha}\left(\mathbb{R}\times\mathbb{R}^{n}\right)$, then
    \begin{equation*}
    \left\|f_{\varepsilon}\right\| _{C_{\mathcal{X}}^{\frac{\alpha}{2},k+\alpha}\left(\mathbb{R}\times\mathbb{R}^{n}\right)} \leq C \left\|f\right\| _{C_{\mathcal{X}}^{\frac{\alpha}{2},k+\alpha}\left(\mathbb{R}\times\mathbb{R}^{n}\right)},
    \end{equation*}
    where the constant $C>0$ depends on $\mathbb{G}$ and $k$ only. Here, for any $k\in\mathbb{N}$, the parabolic H\"{o}lder space $C_{\mathcal{X}}^{\frac{\alpha}{2},k+\alpha}$ is defined in \eqref{time-space-Holder space_Def} below.
    \label{mollifiers_C^alpha/2,alpha}
\item For any $\alpha \in (0,1)$ and $k \in \mathbb{N}$, if $f(t,x) \in B\left(\mathbb{R};C_{\mathcal{X}}^{k+\alpha}\left(\mathbb{R}^{n}\right)\right)$, then
    \begin{equation*}
    \sup_{t\in\mathbb{R}}\left\|f_{\varepsilon}(t,\cdot)\right\| _{C_{\mathcal{X}}^{k+\alpha}\left(\mathbb{R}^{n}\right)} \leq C \sup_{t\in\mathbb{R}}\left\|f(t,\cdot)\right\| _{C_{\mathcal{X}}^{k+\alpha}\left(\mathbb{R}^{n}\right)},
    \end{equation*}
    where the constant $C>0$ depends on $\mathbb{G}$ and $k$ only.
    \label{mollifiers_C^k+alpha}
\item If $f(t,\cdot)$ is $1_{\mathbb{G}}$-periodic on $\mathbb{R}^n$ for any $t\in\mathbb{R}$, i.e., $f(t,k \circ x)=f(t,x)$ for all $x \in \mathbb{R}^n$, then $f_{\varepsilon}(t,\cdot),\varepsilon\in(0,1]$ are also $1_{\mathbb{G}}$-periodic on $\mathbb{R}^n$ for any $t\in\mathbb{R}$.\label{mollifiers_periodicity}
\end{enumerate}
\end{Prop}
\begin{Rem}\label{Rem_Mollifiers_Holder}
\begin{enumerate}[label=(\arabic*), leftmargin=*, itemindent=1.7em]
\item Proposition \ref{Prop_Mollifiers_Holder} still applies to the function $f$ defined on $[0,T]\times\mathbb{R}^n$, as $f$ can be continuously extended to the entire interval in the following manner:
    \begin{equation*}
    f(t,x)=
    \begin{cases}
    f(t,x), & \mbox{if } t\in[0,T], \\
    f(0,x), & \mbox{if } t<0, \\
    f(T,x), & \mbox{if } t>T
    \end{cases}
    \end{equation*}
    for any $x\in\mathbb{R}^n$.
\item In fact, based on the properties of the Carnot torus (see Section \ref{Subsec_2.1} below for more details about Carnot tori), conclusion \ref{mollifiers_periodicity} in Proposition \ref{Prop_Mollifiers_Holder} implies that Proposition \ref{Prop_Mollifiers_Holder} also applies to the space $\mathbb{R}\times\mathbb{T}_{\mathbb{G}}$.
\end{enumerate}
\end{Rem}
To obtain the smooth approximations of Lipschitz functions on the Carnot torus $\mathbb{T}_{\mathbb{G}}$, analogous to the standard mollifiers with compact support on Euclidean space, we introduce the mollifiers on the Carnot group $\mathbb{G}$ as follows (see \cite[Proposition 3.6]{24MMM}):
\begin{equation*}
\psi(x):=
\begin{cases}
\exp\left(\frac{1}{\left\|x\right\|_{\mathbb{G}}^{2r!}-1}\right), & \mbox{if } \left\|x\right\|_{\mathbb{G}}\leq 1 \\
0, & \mbox{otherwise},
\end{cases}
\end{equation*}
and for each $\varepsilon>0$, set
\begin{equation}\label{mollifiers_psi_varepsilon}
\psi_{\varepsilon}(x):=\frac{C}{\varepsilon^Q}\psi\left(D_{\frac{1}{\varepsilon}}(x)\right),\,x\in\mathbb{R}^n,
\end{equation}
where the constant $C>0$ is independent of $\varepsilon$ such that $\int_{\mathbb{R}^n}\psi_{\varepsilon}(x)dx=1$. Here, $\left\|\cdot\right\|_{\mathbb{G}}$ is an homogeneous norm on $\mathbb{G}$ defined as
\begin{equation*}
\left\|x\right\|_{\mathbb{G}}=\left(\sum_{j=1}^{r}|x^{(j)}|^{\frac{2r!}{j}}\right)^{\frac{1}{2r!}}, \,x=\left(x^{(1)},\ldots,x^{(r)}\right)\in\mathbb{G},
\end{equation*}
where $\left|x^{(j)}\right|$ denotes the Euclidean norm on $\mathbb{R}^{n_j}$ and $r>1$ is the step of $\mathbb{G}$, $Q$ is the homogeneous dimension and $D_\lambda$ with $\lambda>0$ is the dilation of $\mathbb{G}$ (see Definition \ref{Def_Carnot group} below). The periodicization of $\psi_{\varepsilon}$ can yield the mollifiers on $\mathbb{T}_{\mathbb{G}}$.
\begin{Prop}[Mollifiers adapted to Carnot tori]\label{Prop_Mollifiers_Lip.}
For any $\varepsilon>0$ and integrable function $g:\mathbb{T}_{\mathbb{G}}\to\mathbb{R}$, define the function
\begin{equation*}
g_{\varepsilon}(x):=\int_{[0,1)^n}\sum_{k\in\mathbb{Z}^n}\psi_{\varepsilon}\left(k\circ x\circ y^{-1}\right)g(y)d y,\,x\in\mathbb{T}_{\mathbb{G}},
\end{equation*}
where $\psi_{\varepsilon}$ is the function defined in \eqref{mollifiers_psi_varepsilon}. Then $g_{\varepsilon}(x) \in C^{\infty}\left(\mathbb{T}_{\mathbb{G}}\right)$ and the following conclusions hold.
\begin{enumerate}[label=(\arabic*), leftmargin=*, itemindent=1.7em]
\item If $g(x) \in C\left(\mathbb{T}_{\mathbb{G}}\right)$, then
\begin{equation*}
\lim\limits_{\varepsilon\to0}\left\|g_{\varepsilon}-g\right\|_{L^{\infty}\left(\mathbb{T}_{\mathbb{G}}\right)}=0.
\end{equation*}\label{g_varepsilon->g}
\item If $g(x) \in C_{\mathcal{X}}^{0+1}\left(\mathbb{T}_{\mathbb{G}}\right)$, then
\begin{equation*}
\left[g_{\varepsilon}\right]_{C_{\mathcal{X}}^{0+1}\left(\mathbb{T}_{\mathbb{G}}\right)} \leq \left[g\right]_{C_{\mathcal{X}}^{0+1}\left(\mathbb{T}_{\mathbb{G}}\right)}.
\end{equation*}\label{mollifiers_C^0,1}
\end{enumerate}
\end{Prop}
Denote $\mathcal{D}'(\Omega)$ by the dual space of $C^{\infty}\left(\Omega\right)$ for any subset $\Omega\subset\mathbb{R}^n$. We are going to provide the mollifiers on the dual spaces $C_{\mathcal{X}}^{-(k+\alpha)}\left(\mathbb{T}_{\mathbb{G}}\right)$ and $L^1\left([0,T];C_{\mathcal{X}}^{-(k+\alpha)}\left(\mathbb{T}_{\mathbb{G}}\right)\right)$, where the definition of $C_{\mathcal{X}}^{-(k+\alpha)}$ is given in \eqref{dual Holder space_Def} below.
\begin{Prop}[Mollifiers adapted to dual spaces]\label{Prop_Mollifiers_dual space}
%Let the inverse of $x$ in $\mathbb{G}$ coincide with $-x$. 
The following conclusions hold.
\begin{enumerate}[label=(\arabic*), leftmargin=*, itemindent=1.7em]
\item For any $\varepsilon>0$ and $\mu\in\mathcal{D}'([0,1)^n)$, define the function
\begin{equation}\label{mu_varepsion}
\mu_{\varepsilon}(x):=\left\langle\mu,\sum_{k\in\mathbb{Z}^n}\psi_{\varepsilon}\left(k\circ x \circ (\cdot)^{-1}\right)\right\rangle,\,x\in\mathbb{T}_{\mathbb{G}},
\end{equation}
where $\psi_{\varepsilon}$ is the function defined in \eqref{mollifiers_psi_varepsilon}. Then $\mu_{\varepsilon}(x)\in C^{\infty}\left(\mathbb{T}_{\mathbb{G}}\right)$.

Moreover, for any $k \in \mathbb{N}$, if $\mu \in C_{\mathcal{X}}^{-k}\left([0,1)^n\right)$, then
    \begin{equation*}
    \left\|\mu_{\varepsilon}\right\| _{C_{\mathcal{X}}^{-k}\left(\mathbb{T}_{\mathbb{G}}\right)} \leq C \left\|\mu\right\| _{C_{\mathcal{X}}^{-k}\left([0,1)^n\right)},
    \end{equation*}
    where the constant $C>0$ depends on $\mathbb{G}$ and $k$ only; additionally, for any $\alpha\in(0,1]$, if $\mu \in C_{\mathcal{X}}^{-(k+\alpha)}\left(\mathbb{T}_{\mathbb{G}}\right)$, then
    $$
    \lim_{\varepsilon\to0}\left\|\mu_{\varepsilon}-\mu\right\| _{C_{\mathcal{X}}^{-(k+\alpha)}\left(\mathbb{T}_{\mathbb{G}}\right)}=0.
    $$\label{mollifiers_C^-k+alpha}

\item For any $\varepsilon>0$ and $\mu\in L_{\text{loc}}^1\left(\mathbb{R};\mathcal{D}'([0,1)^n)\right)$, define the function
$$
\mu_{\varepsilon}(t,x):=\int_{\mathbb{R}}\left\langle\mu(s),\sum_{k\in\mathbb{Z}^n}\psi_{\varepsilon}\left(k\circ x \circ (\cdot)^{-1}\right)\right\rangle\varphi_{\varepsilon}(t-s)ds,\,(t,x)\in\mathbb{R}\times\mathbb{T}_{\mathbb{G}},
$$
where $\varphi_{\varepsilon}$ is the one-dimensional standard mollifier and $\psi_{\varepsilon}$ is the function defined in \eqref{mollifiers_psi_varepsilon}. Then $\mu_{\varepsilon}(t,x)\in C^{\infty}\left(\mathbb{R}\times\mathbb{T}_{\mathbb{G}}\right)$.

Moreover, for any $k \in \mathbb{N}$ and $T>0$, if $\mu \in L^1\left([0,T];C_{\mathcal{X}}^{-k}\left([0,1)^n\right)\right)$, then
    \begin{equation*}
    \left\|\mu_{\varepsilon}\right\| _{L^1\left([0,T];C_{\mathcal{X}}^{-k}\left(\mathbb{T}_{\mathbb{G}}\right)\right)} \leq C \left\|\mu\right\| _{L^1\left([0,T];C_{\mathcal{X}}^{-k}\left([0,1)^n\right)\right)},
    \end{equation*}
    where the constant $C>0$ depends on $\mathbb{G}$ and $k$ only; additionally, for any $\alpha\in(0,1]$, if $\mu \in L^1\left([0,T];C_{\mathcal{X}}^{-(k+\alpha)}\left(\mathbb{T}_{\mathbb{G}}\right)\right)$, then 
    $$
    \lim_{\varepsilon\to0}\left\|\mu_{\varepsilon}-\mu\right\| _{L^1\left([0,T];C_{\mathcal{X}}^{-(k+\alpha)}\left(\mathbb{T}_{\mathbb{G}}\right)\right)}=0.
    $$\label{mollifiers_L^1(C^-k+alpha)}
\end{enumerate}
\end{Prop}
Denote $\mathcal{P}\left(\mathbb{T}_{\mathbb{G}}\right)$ by the set of Borel probability measures on $\mathbb{T}_{\mathbb{G}}$, endowed with the Kantorovich-Rubinstein distance
$$
d_1\left(m, m'\right):=\sup_{[\phi]_{C_{\mathcal{X}}^{0+1}\left(\mathbb{T}_{\mathbb{G}}\right)}\leq 1} \int_{\mathbb{T}_{\mathbb{G}}} \phi(y) d\left(m-m'\right)(y),
$$
where the supremum is taken over all $d_{cc}^{\mathbb{T}_{\mathbb{G}}}$-Lipschitz continuous maps $\phi: \mathbb{T}_{\mathbb{G}} \to \mathbb{R}$ with the Lipschitz constant bounded by $1$. From Section 5.1 of \cite{18CD}, we see that $d_1$ is well-defined and metricizes the weak convergence of measures. The following is a corollary of Proposition \ref{Prop_Mollifiers_dual space}, which provides the mollifiers adapted to $\mathcal{P}\left(\mathbb{T}_{\mathbb{G}}\right)$.
\begin{Cor}\label{Cor_Mollifiers_P(T_G)}
Let $\mu_{\varepsilon}$ be the smooth function defined in \eqref{mu_varepsion}, if $\mu \in \mathcal{P}(\mathbb{T}_{\mathbb{G}})$, then $\mu_{\varepsilon}\to\mu$ in $\mathcal{P}(\mathbb{T}_{\mathbb{G}})$ as $\varepsilon\to0$.
\end{Cor}
Now, we establish the well-posedness of the solution to the equation \eqref{LHJB}, and prove the Schauder estimates in the scale of non-isotropic H\"{o}lder spaces.
\begin{Thm}\label{Thm_LHJB wellposed regularity}
Let $k\in\mathbb{Z}_+$ and $\alpha\in(0,1)$. Assume $b(t,x)$ and $f(t,x)$ are continuous on $[0,T]\times\mathbb{T}_{\mathbb{G}}$, satisfying $b\in B\left([0,T];C_{\mathcal{X}}^{k-1+\alpha}\left(\mathbb{T}_{\mathbb{G}};\mathbb{R}^{n_1}\right)\right)$, $f\in B\left([0,T];C_{\mathcal{X}}^{k-1+\alpha}\left(\mathbb{T}_{\mathbb{G}}\right)\right)$ and $z_T(x)\in C_{\mathcal{X}}^{k+\alpha}\left(\mathbb{T}_{\mathbb{G}}\right)$. Then equation \eqref{LHJB} has a unique solution $z$ which belongs to $C_{\mathcal{X}}^{1,2} \left([0,T) \times \mathbb{T}_{\mathbb{G}} \right) \cap C([0, T] \times \mathbb{T}_{\mathbb{G}})$ and satisfies
\begin{equation}\label{LHJB_USE}
\sup_{t \in[0, T]}\|z(t, \cdot)\|_{C_{\mathcal{X}}^{k+\alpha}\left(\mathbb{T}_{\mathbb{G}}\right)} \leq C\left(\left\|z_T\right\|_{C_{\mathcal{X}}^{k+\alpha}\left(\mathbb{T}_{\mathbb{G}}\right)}+\sup_{t\in(0,T)}\|f(t, \cdot)\|_{C_{\mathcal{X}}^{k-1+\alpha}\left(\mathbb{T}_{\mathbb{G}}\right)}\right).
\end{equation}
In addition, for any constant $\epsilon \in (0,T)$, $z$ satisfies
\begin{equation}\label{LHJB_HSE 1}
\sup_{\substack{t \neq t'\\t,t' \in [0,T-\epsilon]}} \frac{\left\|z\left(t^{\prime}, \cdot\right)-z(t, \cdot)\right\|_{C_{\mathcal{X}}^{k+\alpha}\left(\mathbb{T}_{\mathbb{G}}\right)}}{\left|t^{\prime}-t\right|^{\frac{1}{2}}} \leq C\left(\epsilon^{-\frac{1}{2}}\left\|z_T\right\|_{C_{\mathcal{X}}^{k+\alpha}\left(\mathbb{T}_{\mathbb{G}}\right)}+\sup_{t \in(0, T)}\|f(t, \cdot)\|_{C_{\mathcal{X}}^{k-1+\alpha}\left(\mathbb{T}_{\mathbb{G}}\right)}\right).
\end{equation}
Moreover, if $z_T(x)\in C_{\mathcal{X}}^{k+1+\alpha}\left(\mathbb{T}_{\mathbb{G}}\right)$, then $z$ satisfies
\begin{equation}\label{LHJB_HSE 2}
\sup _{\substack{t \neq t'\\t,t' \in [0,T]}} \frac{\left\|z\left(t^{\prime}, \cdot\right)-z(t, \cdot)\right\|_{C_{\mathcal{X}}^{k+\alpha}\left(\mathbb{T}_{\mathbb{G}}\right)}}{\left|t^{\prime}-t\right|^{\frac{1}{2}}} \leq C\left(\left\|z_T\right\|_{C_{\mathcal{X}}^{k+1+\alpha}\left(\mathbb{T}_{\mathbb{G}}\right)}+\sup_{t \in(0, T)}\|f(t, \cdot)\|_{C_{\mathcal{X}}^{k-1+\alpha}\left(\mathbb{T}_{\mathbb{G}}\right)}\right).
\end{equation}
Here, the constants $C>0$ depend on $\mathbb{G}$, $\alpha$, $k$, $T$ and $\sup_{t\in(0,T)}\|b(t,\cdot)\|_{C_{\mathcal{X}}^{k-1+\alpha}\left(\mathbb{T}_{\mathbb{G}};\mathbb{R}^{n_1}\right)}$ only.
\end{Thm}
The next theorem states that when $z_T$ is only a $d_{cc}$-Lipschitz function in equation \eqref{LHJB}, the solution $z$ satisfies the H\"{o}lder continuity in $t$ and $d_{cc}$-Lipschitz continuity in $x$.
\begin{Thm}\label{Thm_LHJB Lipschitz}
Assume $b(t,x)$ and $f(t,x)$ are continuous on $[0,T]\times\mathbb{T}_{\mathbb{G}}$, satisfying $b\in B\left([0,T];C_{\mathcal{X}}^{\alpha}\left(\mathbb{T}_{\mathbb{G}};\mathbb{R}^{n_1}\right)\right)$, $f\in B\left([0,T];C_{\mathcal{X}}^{\alpha}\left(\mathbb{T}_{\mathbb{G}}\right)\right)$ and $z_T(x)\in C_{\mathcal{X}}^{0+1}\left(\mathbb{T}_{\mathbb{G}}\right)$. Then there exists a unique solution $z\in C_{\mathcal{X}}^{1,2} \left([0, T) \times \mathbb{T}_{\mathbb{G}} \right) \cap C([0, T] \times \mathbb{T}_{\mathbb{G}})$ to the equation \eqref{LHJB}, satisfying
\begin{align}\label{LHJB_Holder&Lipschitz}
& \sup_{\substack{t \neq t'\\t,t' \in [0,T]}} \frac{\left\|z\left(t^{\prime}, \cdot\right)-z(t, \cdot)\right\|_{L^{\infty}\left(\mathbb{T}_{\mathbb{G}}\right)}}{\left|t^{\prime}-t\right|^{\frac{1}{2}}} +\sup_{t \in [0,T]}\left[z(t,\cdot)\right]_{C_{\mathcal{X}}^{0+1}\left(\mathbb{T}_{\mathbb{G}}\right)}\notag \\
\leq & C\left(\left\|z_T\right\|_{C_{\mathcal{X}}^{0+1}\left(\mathbb{T}_{\mathbb{G}}\right)} +\|f\|_{L^{\infty}\left((0,T)\times\mathbb{T}_{\mathbb{G}}\right)}\right),
\end{align}
where the constant $C>0$ depends on $\mathbb{G}$, $T$ and $\|b\|_{L^{\infty}\left((0,T)\times\mathbb{T}_{\mathbb{G}}\right)}$ only.
\end{Thm}

Finally, we provide the results of existence, uniqueness and regularity of the weak solution to the degenerate FPK equation \eqref{general FPK}. Before this, we refer to the idea of duality in \cite{21Ri} to state a suitable definition of the distributional solution.
\begin{Def}\label{Def_general FPK weak sol.}
Let $k\in\mathbb{Z}_+$ and $\alpha\in(0,1)$. Assume $b \in C_{\mathcal{X}}^{\frac{\alpha}{2},k-1+\alpha}\left([0,T]\times\mathbb{T}_{\mathbb{G}};\mathbb{R}^{n_1}\right)$, $\upsilon \in L^1\left([0,T];C_{\mathcal{X}}^{-k}\left([0,1)^n\right)\cap C_{\mathcal{X}}^{-(k+\alpha)}(\mathbb{T}_{\mathbb{G}})\right)$ and $\rho_0 \in C_{\mathcal{X}}^{-k}\left([0,1)^n\right)\cap C_{\mathcal{X}}^{-(k+\alpha)}(\mathbb{T}_{\mathbb{G}})$. For a given function $\rho \in C\left([0,T];C_\mathcal{X}^{-(k+\alpha)}(\mathbb{T}_{\mathbb{G}})\right)$, if for all $f \in C\left([0,t]\times\mathbb{T}_{\mathbb{G}}\right)\cap B\left([0,t];C_\mathcal{X}^{k+\alpha}(\mathbb{T}_{\mathbb{G}})\right)$, $\xi \in C_\mathcal{X}^{k+\alpha}(\mathbb{T}_{\mathbb{G}})$ and the solution $z \in C_{\mathcal{X}}^{1,2}\left([0,t)\times\mathbb{T}_{\mathbb{G}}\right) \cap C\left([0,t]\times\mathbb{T}_{\mathbb{G}}\right)$ to the linear equation as follows
\begin{equation}\label{dual FPK}
\begin{cases}
-\partial_tz-\Delta_\mathcal{X}z+b\cdot D_{\mathcal{X}}z=f, & \text{ in } [0,t) \times \mathbb{T}_{\mathbb{G}}, \\
z(t)=\xi, & \text{ in } \mathbb{T}_{\mathbb{G}},
\end{cases}
\end{equation}
the following weak formulation holds true:
\begin{equation}\label{general FPK_weak formulation}
\left\langle\rho(t),\xi\right\rangle +\int_{0}^{t} \left\langle \rho(s),f(s,\cdot) \right\rangle ds =\left\langle\rho_0,z(0,\cdot)\right\rangle +\int_{0}^{t}\left\langle \upsilon(s),z(s,\cdot)\right\rangle ds,
\end{equation}
then we say that $\rho$ is a weak solution to equation \eqref{general FPK}. Here, $C_{\mathcal{X}}^{-k}\left([0,1)^n\right)$ and $C_{\mathcal{X}}^{-(k+\alpha)}(\mathbb{T}_{\mathbb{G}})$ are dual spaces of $C_{\mathcal{X}}^{k}\left([0,1)^n\right)$ and $C_{\mathcal{X}}^{k+\alpha}(\mathbb{T}_{\mathbb{G}})$ respectively, and $\left\langle \cdot,\cdot \right\rangle$ denotes the duality between $C_\mathcal{X}^{-(k+\alpha)}(\mathbb{T}_{\mathbb{G}})$ and $C_\mathcal{X}^{k+\alpha}(\mathbb{T}_{\mathbb{G}})$. See \eqref{dual Holder space_Def} below for the definition of dual spaces.
\end{Def}
\begin{Thm}\label{Thm_general FPK wellposed regularity}
Let $k\in\mathbb{Z}_+$ and $\alpha\in(0,1)$. Assume $b \in C_{\mathcal{X}}^{\frac{\alpha}{2},k-1+\alpha}\left([0,T]\times\mathbb{T}_{\mathbb{G}};\mathbb{R}^{n_1}\right)$, $\upsilon \in L^1\left([0,T];C_{\mathcal{X}}^{-k}\left([0,1)^n\right)\cap C_{\mathcal{X}}^{-(k+\alpha)}(\mathbb{T}_{\mathbb{G}})\right)$ and $\rho_0 \in C_{\mathcal{X}}^{-k}\left([0,1)^n\right)\cap C_{\mathcal{X}}^{-(k+\alpha)}(\mathbb{T}_{\mathbb{G}})$. Then there exists a unique weak solution $\rho$ in the sense of Definition \ref{Def_general FPK weak sol.} to the equation \eqref{general FPK}, satisfying
\begin{equation}\label{general FPK sol. regularity}
\sup_{t \in [0,T]}\left\|\rho(t)\right\|_{C_{\mathcal{X}}^{-(k+\alpha)}(\mathbb{T}_{\mathbb{G}})} \leq C\left(\left\|\rho_0\right\|_{C_{\mathcal{X}}^{-(k+\alpha)}(\mathbb{T}_{\mathbb{G}})} + \left\|\upsilon\right\|_{L^1\left([0,T];C_{\mathcal{X}}^{-(k+\alpha)}(\mathbb{T}_{\mathbb{G}})\right)}\right),
\end{equation}
where the constant $C>0$ depends on $\mathbb{G}$, $\alpha$, $k$, $T$ and $\sup_{t\in(0,T)}\|b(t,\cdot)\|_{C_{\mathcal{X}}^{k-1+\alpha}\left(\mathbb{T}_{\mathbb{G}};\mathbb{R}^{n_1}\right)}$ only.

Moreover, the solution is stable: if $b^{i} \to b$ in $C_{\mathcal{X}}^{\frac{\alpha}{2},k-1+\alpha}\left([0,T] \times \mathbb{T}_{\mathbb{G}};\mathbb{R}^{n_1}\right)$, $\upsilon^{i} \to \upsilon$ in $L^1\left([0,T];C_{\mathcal{X}}^{-(k+\alpha)}(\mathbb{T}_{\mathbb{G}})\right)$ and $\rho_{0}^{i} \to \rho_{0}$ in $C_{\mathcal{X}}^{-(k+\alpha)}(\mathbb{T}_{\mathbb{G}})$ as $i\to +\infty$, with $\upsilon^{i}\in L^1\left([0,T];C_{\mathcal{X}}^{-k}\left([0,1)^n\right)\right)$ and $\rho_{0}^{i}\in C_{\mathcal{X}}^{-k}\left([0,1)^n\right)$, then, calling $\rho^{i}$ and $\rho$ the solutions related to $\left(\rho_{0}^{i},b^{i},\upsilon^{i}\right)$ and $\left(\rho_{0},b,\upsilon\right)$ respectively, we have $\rho^{i} \to \rho$ in $C\left([0,T];C_{\mathcal{X}}^{-(k+\alpha)}(\mathbb{T}_{\mathbb{G}})\right)$ as $i\to +\infty$.
\end{Thm}

The main contributions of this paper are the following. First, Carnot groups are typical examples of sub-Riemannian manifolds, whose non-commutative nature brings an inherent geometric difference from classical Euclidean spaces. This paper investigates the periodic linear degenerate parabolic equation on Carnot groups, i.e., equation \eqref{LHJB}. Building upon the a priori regularity results in the previous work \cite{25JWY}, we develop the existence, uniqueness, and regularity of the solution to the equation \eqref{LHJB}. Moreover, we apply these results to the dual equation of \eqref{LHJB}, namely the general-form degenerate FPK equation in \eqref{general FPK}. This FPK equation corresponds to the periodic Carnot diffusion process in \eqref{periodic Carnot diffusion}, and its solution can characterize the transition probability measure of this diffusion process. Second, in terms of application, the results established in this paper concerning the well-posedness and regularity of solutions to equations \eqref{LHJB} and \eqref{general FPK} can be used to address the problem of well-posedness for a class of mean field game master equations on Carnot tori. For details of the master equations for mean field games, see \cite{19CDLL,22GMMZ}. Third, when proving the relevant regularity from the a priori Schauder regularity for equation \eqref{LHJB}, new difficulties must be overcome. On one hand, we require stronger a priori regularity, specifically by enhancing the a priori regularity of the solution up to time zero. We employ the theory of singular integral operators to overcome this difficulty. On the other hand, we also need the technique of smoothing. Due to the non-isotropy of the H\"{o}lder spaces and the periodicity of the associated functions, traditional mollifiers are no longer applicable. To this end, we construct several families of mollifiers which are respectively adapted to the family of H\"{o}rmander vector fields $\mathcal{X}$ and Carnot tori. Fourth, to prove the well-posedness and regularity of the weak solution to the FPK equation in the dual spaces of non-isotropic H\"{o}lder spaces, we also construct mollifiers on these dual spaces.

This paper is structured as follows. In Section \ref{Sec_2}, we first introduce Carnot groups, Carnot tori, and the definitions of non-isotropic H\"{o}lder spaces and their dual spaces. Next, we present key concepts of singular integral operator theory. Finally, we review a priori Schauder estimates for the Cauchy problem on Carnot groups with rough coefficients. In Section \ref{Sec_3}, we construct several mollifiers, which are respectively adapted to the family of H\"{o}rmander vector fields $\mathcal{X}$, Carnot tori and the dual spaces discussed in this paper. In Section \ref{Sec_4}, we study the linear degenerate parabolic equation \eqref{LHJB}, and then prove Theorem \ref{Thm_LHJB wellposed regularity} and Theorem \ref{Thm_LHJB Lipschitz}. Lastly, in Section \ref{Sec_5}, we obtain the existence, uniqueness and regularity of the weak solution to the degenerate FPK equation \eqref{general FPK}, i.e., Theorem \ref{Thm_general FPK wellposed regularity}.

\section{Preliminaries}\label{Sec_2}

In this section, we first introduce the concepts of the homogeneous Carnot group and Carnot torus, and provide the definitions of non-isotropic H\"{o}lder spaces as well as their dual spaces. Next, we present the relevant concepts of the theory of singular integral operators, which are crucial for the present work. Finally, we review some a priori Schauder estimates for the Cauchy problem on Carnot groups with rough coefficients.

\subsection{Brief overview of Carnot groups and Carnot tori}\label{Subsec_2.1}

We now introduce some notations and preliminaries about Carnot groups. For a comprehensive overview, we refer the reader to the monographs \cite{07BLU,22BB}. For additional related research on Carnot groups, see \cite{03BLU,13CCM,24MMM}.

\begin{Def}[Homogeneous group, see \cite{07BLU,22BB}]\label{Def_Homogeneous group}
Let $\mathbb{G}=\left(\mathbb{R}^n,\circ\right)$ be a Lie group on $\mathbb{R}^n$ with $\circ$ being a given Lie group law on $\mathbb{R}^n$, called ``translation''. We say that $\mathbb{G}$ is a homogeneous (Lie) group on $\mathbb{R}^n$ if there exists an $n$-tuple of real numbers $(\alpha_1,\ldots,\alpha_n)$, with $1=\alpha_1 \leq \alpha_2 \leq \ldots \leq \alpha_n$, such that the ``dilation'' $D_{\lambda}:\mathbb{R}^n \to \mathbb{R}^n$, defined as
$$
D_{\lambda}(x):=\left(\lambda^{\alpha_1} x_1,\ldots,\lambda^{\alpha_n} x_n\right)
$$
is an automorphism of the group $\mathbb{G}$ for every $\lambda>0$.

We denote by $\mathbb{G}=(\mathbb{R}^n,\circ,D_{\lambda})$ the datum of a homogeneous group on $\mathbb{R}^n$ with group law $\circ$ and dilation group $\{D_{\lambda}\}_{\lambda>0}$. Moreover, the number
$$
Q:=\sum_{i=1}^{n}\alpha_i
$$
is called the homogeneous dimension of the homogeneous group. 
\end{Def}
A differential operator $P$ on $\mathbb{G}$ is said to be \textit{left-invariant} if
$$
P^x \left(f(y \circ x)\right)=\left(P f\right)(y \circ x)
$$
for every test function $f$ and $x,y \in \mathbb{R}^n$. Similarly, a differential operator $P$ on $\mathbb{G}$ is said to be \textit{right-invariant} if
$$
P^y \left(f(y \circ x)\right)=\left(P f\right)(y \circ x)
$$
for every test function $f$ and $x,y \in \mathbb{R}^n$.

For $\delta\in\mathbb{R}$, $P$ is said to be \textit{$D_{\lambda}$-homogeneous of degree $\delta$} (or simply ``\textit{$\delta$-homogeneous}'') if
$$
P^x\left(f(D_{\lambda}(x))\right)=\lambda^{\delta} \left(P f\right) \left(D_{\lambda}(x)\right)
$$
for every test function $f$, $\lambda>0$ and $x \in \mathbb{R}^n\setminus\{0\}$. A real function $f$ defined on $\mathbb{R}^n$ is called \textit{$D_{\lambda}$-homogeneous of degree $\delta$} (or simply ``\textit{$\delta$-homogeneous}'') if $f\not\equiv 0$ and $f$ satisfies
$$
f(D_{\lambda}(x))=\lambda^{\delta}f(x)
$$
for any $\lambda>0$ and $x \in \mathbb{R}^n$.

It is a verifiable result that the Haar measure, denoted by $dx$, of the homogeneous group $\mathbb{G}=(\mathbb{R}^n,\circ, D_{\lambda})$ is left-right-invariant and coincides with the Lebesgue measure on $\mathbb{R}^n$.
%We provide some useful rules for changing variables inside an integral:
%\begin{gather*}
%x'=x \circ y \Rightarrow dx'=dx; \\
%x'=y \circ x \Rightarrow dx'=dx; \\
%x'=x^{-1} \Rightarrow dx'=dx; \\
%x'=D_{\lambda}(x) \Rightarrow dx'=\lambda^Q dx.
%\end{gather*}
%The \textit{convolution} in the group $\mathbb{G}$ is defined as
%\begin{equation}\label{convolution in group}
%(f \ast g)(x):=\int_{\mathbb{R}^n}f\left(x \circ y^{-1}\right)g(y)dy=\int_{\mathbb{R}^n}g\left(y^{-1} \circ x\right)f(y)dy,
%\end{equation}
%for any couple of functions such that the above integrals make sense.
\begin{Lem}[see {\cite[Theorem 3.29, Remark 3.30]{22BB}}]\label{Lem_LI&RI vector field}
Let $X_i$ be the left-invariant vector field which coincides with $\partial_{x_i}$ at the origin, i.e. $X_{i}(0)=\partial_{x_{i}}|_{0}$, for $i=1,\ldots,n$. Then $X_i$ is $\alpha_i$-homogeneous and has the following structure:
\begin{equation*}
X_i=\partial_{x_i}+\sum_{i<k}q_{ik}(x)\partial_{x_k},\quad i=1,\ldots,n
\end{equation*}
where $q_{ik}(\cdot)$ is a $(\alpha_k-\alpha_i)$-homogeneous polynomial. In particular, $q_{ik}$ can only depend on the variable $x_1,\ldots,x_{k-1}$. Moreover, the transpose of the vector field is just its opposite:
\begin{equation*}
X_i^*=-X_i.
\end{equation*}

Analogous properties hold for the right-invariant vector field $X_i^R$ which coincides with $\partial_{x_i}$ at the origin.
\end{Lem}

Meanwhile, we proceed to introduce the H\"{o}rmander's condition.
\begin{Def}[H\"{o}rmander's condition, see \cite{67Ho}]
Let $\Omega$ be an open subset of $\mathbb{R}^n$, and let $Y_{1},Y_{2},\ldots, Y_{m}$ be real smooth vector fields defined on $\Omega$. We say $Y_{1},Y_{2},\ldots, Y_{m}$ satisfy the H\"{o}rmander's condition on $\Omega$ if there exists a smallest integer $r\geq 1$ such that $Y_{1},Y_{2},\ldots, Y_{m}$ together with their commutators of length at most $r$ span the tangent space $T_{x}(\Omega)$ at each point $x\in\Omega$. The integer $r$ is called the H\"{o}rmander's index of $\Omega$.
\end{Def}
Given a family of vector fields $\mathcal{Y}=(Y_1,Y_2,\ldots,Y_m)$ satisfying H\"{o}rmander's condition on a domain $\Omega\subseteq\mathbb{R}^n$, we define the corresponding Carnot-Carath\'{e}odory distance (also referred to as the subunit metric or control distance). For additional details concerning its definition, we direct the reader to \cite[Definition 1.29]{22BB}.
\begin{Def}[Carnot-Carath\'{e}odory distance]\label{Def_CC distance}
For any points $x,y\in \Omega$ and $\delta>0$, let $C_{x,y}(\delta)$ be the collection of absolutely continuous mapping $\varphi:[0,\delta]\to \Omega$, which satisfies $\varphi(0)=x,\varphi(\delta)=y$ and
\[ \varphi'(t)=\sum_{i=1}^{m}a_{i}(t)Y_{i}(\varphi(t)),~~\sum_{i=1}^{m}{a_{i}(t)}^2\leq 1,~~a.e.~~ t\in [0,\delta]. \]
The Carnot-Carath\'{e}odory distance $d_{cc}(x,y)$ is defined as
\begin{equation*}\label{2-14}
d_{cc}(x,y):=\inf\{\delta>0~|~ \exists~\varphi\in C_{x,y}(\delta)\}.
\end{equation*}
\end{Def}
The Chow-Rashevskii theorem, together with H\"{o}rmander's condition, guarantees the well-definedness of the Carnot-Carath\'{e}odory distance (see \cite[Theorem 57]{14Br}). It is also known that this distance is (locally) topologically equivalent to the Euclidean distance, that is, for any compact set $K \subset \Omega$, there exists a constant $C>0$ such that, for any $x,y \in K$, we have
\begin{equation}\label{d_cc VS Euclidean}
C^{-1} |x-y| \leq d_{cc}(x, y) \leq C |x-y|^{\frac{1}{r}},
\end{equation}
where $|\cdot|$ denotes the Euclidean norm on $\mathbb{R}^n$ and $r$ is the H\"{o}rmander's index of $\Omega$.

We are now in a position to present the definitions of the Carnot group and Carnot torus, which are stated below:
\begin{Def}[Homogeneous Carnot group, see \cite{07BLU,22BB}]\label{Def_Carnot group}
The Lie group $\mathbb{G}=(\mathbb{R}^{n},\circ)$ is called the (homogeneous) Carnot group (or a stratified Lie group) if the following two conditions are fulfilled:
\begin{enumerate}[label=(\arabic*), leftmargin=*, itemindent=1.7em]
\item the decomposition $\mathbb{R}^{n}=\mathbb{R}^{n_{1}}\times \mathbb{R}^{n_{2}}\times \cdots\times \mathbb{R}^{n_{r}}$ holds for some integers $n_{1},\ldots, n_{r}$ such that      $n_{1}+n_{2}+\cdots+n_{r}=n$, and for each $\lambda>0$ there exists a dilation
\begin{equation*}
D_{\lambda}(x)=D_{\lambda}\left(x^{(1)},x^{(2)},\ldots,x^{(r)}\right)=\left(\lambda x^{(1)},\lambda^{2}x^{(2)},\ldots,\lambda^{r}x^{(r)}\right),
\end{equation*}
which is an automorphism of the group $\mathbb{G}$. Here $x^{(i)}\in \mathbb{R}^{n_{i}}$ for $i=1,2,\ldots,r$;\label{Carnot group C1}
\item let $X_{1},\ldots,X_{n_{1}}$ be the left-invariant vector fields on $\mathbb{G}$ such that $X_{k}(0)=\partial_{x_{k}}|_{0}$ for $k=1,\ldots,n_{1}$. Then
\begin{equation*}
\operatorname{rank}(\operatorname {Lie}\{X_1,\ldots,X_{n_1}\}(x))=n
\end{equation*}
for every $x\in\mathbb{R}^n$, i.e., $X_1,\ldots,X_{n_1}$ satisfy the H\"{o}rmander's condition of step $r$ on $\mathbb{R}^{n}$.\label{Carnot group C2}
\end{enumerate}

If \ref{Carnot group C1} and \ref{Carnot group C2} are satisfied, we shall say that the triple $\mathbb{G}=(\mathbb{R}^{n},\circ, D_{\lambda})$ is a (homogeneous) Carnot group with homogeneous dimension 
\begin{equation*}
Q=\sum_{j=1}^{r}j n_{j}.
\end{equation*}
We also say that $\mathbb{G}$ has step $r$ and $n_1$ generators. The vector fields $X_1,\ldots,X_{n_{1}}$ are called the Jacobian generators of $\mathbb{G}$.  
\end{Def}
\begin{Def}[Homogeneous norm on $\mathbb{G}$]
A continuous function $\|\cdot\|:\mathbb{G}\to [0,+\infty)$ is said to be a homogeneous norm on  $\mathbb{G}$ if it satisfies the following conditions:
\begin{enumerate}[label=(\arabic*), leftmargin=*, itemindent=1.7em]
\item $\|D_{\lambda}(x)\|=\lambda\|x\|$ for all $x\in \mathbb{G}$ and $\lambda>0$;
\item $\|x\|=0$ if and only if $x=0$.
\end{enumerate}
Moreover,  $\|\cdot\|$ is symmetric if $\|x^{-1}\|=\|x\|$ for all $x\in \mathbb{G}$.
\end{Def}
It follows from \cite[Proposition 5.1.4]{07BLU} that all homogeneous norms on $\mathbb{G}$ are equivalent. Throughout this paper, we shall take the homogeneous norm (see \cite[Theorem 5.2.8]{07BLU})
\begin{equation*}\label{homogeneous norm:=d_cc}
\|x\|:=d_{cc}(x,0),\quad x\in\mathbb{G}.
\end{equation*}
And we note that for any $x,y\in\mathbb{G}$, $\|y^{-1}\circ x\|:=d_{cc}(x,y)$.
\begin{Def}[Carnot torus]
The torus in the Carnot group $\left(\mathbb{G},\circ\right)$, denoted by $\mathbb{T}_{\mathbb{G}}$, is defined as the quotient space $\mathbb{G}/\mathbb{Z}^{n}$, which is determined by the following equivalence relation:
$$
x \sim y \text{ if there exists } k \in \mathbb{Z}^{n} \text{ such that } k \circ x=y.
$$
\end{Def}
Similar to the Euclidean torus, functions on $\mathbb{T}_{\mathbb{G}}$ are functions $f$ on $\mathbb{G}$ that satisfy $f(k \circ x)=f(x)$ for all $x \in \mathbb{G}$ and $k \in \mathbb{Z}^{n}$. Such functions are called $1_{\mathbb{G}}$-periodic functions. Moreover, the following lemma implies that the torus $\mathbb{T}_{\mathbb{G}}$ can be considered as the cube $[0,1)^n$.
\begin{Lem}\label{Lem_T_G=[0,1)^n}
For every point $x\in\mathbb{G}$, there exists a unique point $x_0\in[0,1)^n$ and a finite number of group actions generated by elements of the form $(k,0)\in\mathbb{Z}^{n_1}\times\{0\}^{n-n_1}$ such that applying these actions to $x_0$ yields $x$.
\end{Lem}
\begin{proof}
For existence, we can directly refer to \cite[Lemma 1]{07St}. To prove uniqueness, if there exist $k_1,k_2\in\mathbb{Z}^n$ and $x_{0,1},x_{0,2}\in[0,1)^n$ such that $k_1\circ x_{0,1}=x=k_2\circ x_{0,2}$, then
$$
x_{0,2}\circ x_{0,1}^{-1}=k_2^{-1}\circ k_1\in\mathbb{Z}^n.
$$
Due to the fact that (see \cite[Corollary 1.3.16]{07BLU})
\begin{equation*}
(y^{-1})_{j}=-y_j+q_j(y),
\end{equation*}
where $q_j(y)$ is a polynomial function in $y$, $D_{\lambda}$-homogeneous of degree $\alpha_j$, only depending on the $y_k$'s with $\alpha_k<\alpha_j$, and (see \cite[Corollary 1.3.18]{07BLU})
\begin{equation*}
(y^{-1}\circ x)_{j}=x_j-y_j+\sum_{k:\alpha_k<\alpha_j}P_{j,k}(x,y)(x_k-y_k),\,x,y\in\mathbb{G},\,j\in\{1,\ldots,n\},
\end{equation*}
where $P_{j,k}(x,y)$ is a polynomial function in $x$ and $y$ only depending on the $x_k$'s and $y_k$'s with $\alpha_k<\alpha_j$. For any $x\in\mathbb{G}$, denote the decomposition $x=\left(x^{(1)},x^{(2)},\ldots,x^{(r)}\right)$, where $x^{(i)}\in\mathbb{R}^{n_i}$ and $n_{1}+n_{2}+\cdots+n_{r}=n$. Thus we have
\begin{equation*}
\left(x_{0,2}\circ x_{0,1}^{-1}\right)^{(1)}=\left(x_{0,2}^{-1}\right)^{(1)}-\left(x_{0,1}^{-1}\right)^{(1)} =-x_{0,2}^{(1)}+x_{0,1}^{(1)}\in\mathbb{Z}^{n_1}\cap[0,1)^{n_1},
\end{equation*}
namely $x_{0,1}^{(1)}-x_{0,2}^{(1)}=0$. Furthermore, since $\left(x_{0,2}^{-1}\right)^{(1)}=\left(x_{0,1}^{-1}\right)^{(1)}$, we have
\begin{equation*}
\left(x_{0,2}\circ x_{0,1}^{-1}\right)^{(2)}=\left(x_{0,2}^{-1}\right)^{(2)}-\left(x_{0,1}^{-1}\right)^{(2)} +0=-x_{0,2}^{(2)}+x_{0,1}^{(2)}\in\mathbb{Z}^{n_2}\cap[0,1)^{n_2}.
\end{equation*}
By analogy, we can conclude that $x_{0,1}^{(i)}=x_{0,2}^{(i)}$ for any $i\in\{1,\ldots,r\}$, which leads to the uniqueness.
\end{proof}
We refer to \cite{07St,21MMT} for more details about the periodicity on the Carnot group. However, it is worth to observe that the Carnot torus does not coincide with the Euclidean torus. For instance, $\mathbb{T}_{\mathbb{G}}$ is not obtained identifying the points of two opposite faces of $[0,1]^n$ with the same two coordinates. It is easy to find that $\mathbb{T}_{\mathbb{G}}$ is a bounded compact space, naturally endowed with the distance induced by any distance $d$ in $\mathbb{G}$ as
$$
d^{\mathbb{T}_{\mathbb{G}}}(x,y):=\inf_{\substack{x', y' \in \mathbb{G} \\ x' \sim x, y' \sim y}}d\left(x',y'\right),\,x,y \in \mathbb{T}_{\mathbb{G}}.
$$

\begin{Rem}\label{Rem_Holder norm R^n=T_mathbb G}
Since a $1_{\mathbb{G}}$-periodic function $f$ on $\mathbb{G}=\mathbb{R}^{n}$ can be regarded as a function on $\mathbb{T}_{\mathbb{G}}$ (still denoted as $f$), we note that the H\"{o}lder norm induced by $d$ on $\mathbb{G}$ is equal to the one induced by $d^{\mathbb{T}_{\mathbb{G}}}$ on $\mathbb{T}_{\mathbb{G}}$, namely $\left\|f\right\|_{C_{d}^{\alpha}\left(\mathbb{R}^{n}\right)} =\left\|f\right\|_{C_{d^{\mathbb{T}_{\mathbb{G}}}}^{\alpha}\left(\mathbb{T}_{\mathbb{G}}\right)},\alpha\in(0,1]$. This is because, on the one hand, for any $x,y \in \mathbb{G}$, $x \neq y$,
\begin{equation*}
\frac{|f(x)-f(y)|}{d(x,y)^{\alpha}} \leq\frac{|f(x)-f(y)|}{d^{\mathbb{T}_{\mathbb{G}}}(x,y)^{\alpha}}.
\end{equation*}
Hence $\left[f\right]_{C_{d}^{\alpha}\left(\mathbb{R}^{n}\right)} \leq\left[f\right]_{C_{d^{\mathbb{T}_{\mathbb{G}}}}^{\alpha}\left(\mathbb{T}_{\mathbb{G}}\right)}$. On the other hand, for any $x,y \in \mathbb{T}_{\mathbb{G}}$, $x \neq y$,
\begin{equation*}
\frac{|f(x)-f(y)|}{d^{\mathbb{T}_{\mathbb{G}}}(x,y)^{\alpha}} =\inf_{\substack{x', y' \in \mathbb{G} \\ x' \sim x, y' \sim y}}\frac{|f(x')-f(y')|}{d(x',y')^{\alpha}}.
\end{equation*}
Hence $\left[f\right]_{C_{d^{\mathbb{T}_{\mathbb{G}}}}^{\alpha}\left(\mathbb{T}_{\mathbb{G}}\right)} \leq\left[f\right]_{C_{d}^{\alpha}\left(\mathbb{R}^{n}\right)}$. Since it's easy to find that $\left\|f\right\|_{C\left(\mathbb{R}^{n}\right)} =\left\|f\right\|_{C\left(\mathbb{T}_{\mathbb{G}}\right)}$, finally we obtain $\left\|f\right\|_{C_{d}^{\alpha}\left(\mathbb{R}^{n}\right)} =\left\|f\right\|_{C_{d^{\mathbb{T}_{\mathbb{G}}}}^{\alpha}\left(\mathbb{T}_{\mathbb{G}}\right)}$. In this paper, we typically take $d=d_{cc}$.
\end{Rem}

Next we introduce an important class of non-isotropic H\"{o}lder spaces associated with the family of vector fields $\mathcal{X}=\{X_1,\ldots,X_{n_1}\}$ (see \cite{07BB,10BBLU}).

Let $\Omega\subseteq\mathbb{R}^n$ be any open subset and $X_I:=X_{i_1} \cdots X_{i_{|I|}}$, where $I$ is any multi-index $I=\left(i_1, \ldots, i_{|I|}\right)$ with $i_j \in \{1,\cdots,n_1\}, j=1,\cdots,|I|$. For any $k \in \mathbb{N}$, we define the non-isotropic space
\begin{equation*}
C_{\mathcal{X}}^{k}\left(\Omega\right):=\left\{\phi \in C\left(\Omega\right)~|~X_I \phi \in C\left(\Omega\right),\,\forall\,|I| \leq k\right\}.
\end{equation*}
For any $k \in \mathbb{N}$ and $\alpha \in(0,1]$ we define the non-isotropic H\"{o}lder spaces
\begin{equation*}
C_{\mathcal{X}}^{\alpha}\left(\Omega\right):=\left\{\phi \in L^{\infty}\left(\Omega\right)~\bigg|~ \sup _{\substack{x, y \in \Omega \\
x \neq y}} \frac{|\phi(x)-\phi(y)|}{d_{cc}(x, y)^\alpha}<+\infty\right\},
\end{equation*}
\begin{equation}\label{non-isotropic Holder spaces_Def}
C_{\mathcal{X}}^{k+\alpha}\left(\Omega\right):=\left\{\phi \in L^{\infty}\left(\Omega\right)~|~X_I \phi \in C_{\mathcal{X}}^{\alpha}\left(\Omega\right),\,\forall\,|I| \leq k\right\}.
\end{equation}
For any function $\phi \in C_{\mathcal{X}}^{ \alpha}\left(\Omega\right)$, the H\"{o}lder seminorm can be defined as
$$
[\phi]_{C_{\mathcal{X}}^{ \alpha}\left(\Omega\right)}:=\sup_{\substack{x, y \in \Omega \\ x \neq y}} \frac{|\phi(x)-\phi(y)|}{d_{cc}(x, y)^\alpha}.
$$
Furthermore, for any $\phi \in C_{\mathcal{X}}^{k+ \alpha}\left(\Omega\right)$, the H\"{o}lder norm is defined as
\begin{equation}\label{Holder norm}
\|\phi\|_{C_{\mathcal{X}}^{k+\alpha}\left(\Omega\right)}:=\|\phi\|_{C_{\mathcal{X}}^k\left(\Omega\right)}+\sum_{0 \leq|I| \leq k}\left[X_I \phi\right]_{C_{\mathcal{X}}^{ \alpha}\left(\Omega\right)},
\end{equation}
where $\|\phi\|_{C_{\mathcal{X}}^k\left(\Omega\right)}:=\sum\limits_{0 \leq |I|\leq k}\|X_I\phi\|_{L^{\infty}\left(\Omega\right)}$.

Endowed with the above norm, $C_{\mathcal{X}}^{k+ \alpha}\left(\Omega\right)$ is a Banach space and it follows from \eqref{d_cc VS Euclidean} that, for any compact set $K \subset \Omega$,
$$
C^{-1}\|\phi\|_{C^{ \frac{\alpha}{k}}\left(K\right)} \leq\|\phi\|_{C_{\mathcal{X}}^{\alpha}\left(K\right)} \leq C\|\phi\|_{C^{ \alpha}\left(K\right)},
$$
where $\|\phi\|_{C^{\alpha}\left(K\right)}$ is the standard H\"{o}lder norm, and $C>0$ is a constant depending only on the dimension $n$ and the family of vector fields $\mathcal{X}$.

%Moreover, we denote by $C_{\mathcal{X},c}^{k+\alpha}\left(\Omega\right)$ the space consisting of all $C_{\mathcal{X}}^{k+\alpha}\left(\Omega\right)$ functions that have compact support in $\Omega$.

%We can also define the parabolic Carnot-Carath\'{e}odory distance
%$$
%d_p\left((t,x),(s,y)\right):=\sqrt{d_{cc}(x,y)^2+|t-s|},
%$$
%which is a well-defined distance on $\mathbb{R} \times \Omega$. Replacing distance $d_{cc}$ with $d_p$,
For any $T>0$, $l,k \in \mathbb{N}$ and $\alpha,\beta \in(0,1]$, we define the parabolic non-isotropic H\"{o}lder spaces on $[0,T] \times \Omega$ as
\begin{align*}
& C_{\mathcal{X}}^{\beta,\alpha}\left([0,T]\times \Omega \right)\\
:=&\left\{\phi \in L^{\infty}\left([0,T] \times \Omega\right)~\bigg|~\sup _{\substack{(t,x), (s,y) \in [0,T]\times \Omega \\
(t,x) \neq (s,y)}} \frac{|\phi(t,x)-\phi(s,y)|}{|t-s|^\beta+d_{cc}(x,y)^\alpha}<+\infty\right\},
\end{align*}
\begin{align}\label{time-space-Holder space_Def}
& C_{\mathcal{X}}^{l+\beta,k+\alpha}\left([0,T] \times \Omega\right)\\
:=&\left\{\phi \in L^{\infty}\left([0,T]\times \Omega\right)~\big|~\partial_t^i X_I \phi \in C_{\mathcal{X}}^{\beta,\alpha}\left([0,T]\times \Omega \right),\,\forall\, i,|I|\in\mathbb{N},i\leq l,|I|\leq k\right\}.\notag
\end{align}
with the seminorm
$$
[\phi]_{C_{\mathcal{X}}^{\beta,\alpha}\left([0,T]\times \Omega\right)}:=\sup_{\substack{(t,x), (s,y) \in [0,T]\times \Omega \\
(t,x) \neq (s,y)}} \frac{|\phi(t,x)-\phi(s,y)|}{|t-s|^\beta+d_{cc}(x,y)^\alpha},
$$
and the norm
\begin{equation*}\label{para. Holder norm}
\|\phi\|_{C_{\mathcal{X}}^{l+\beta,k+\alpha}\left([0,T]\times \Omega\right)}:=\sum_{i\leq l,|I|\leq k}\left(\\|\partial_t^i X_I\phi\|_{L^{\infty}\left([0,T]\times \Omega\right)} +\left[\partial_t^i X_I \phi\right]_{C_{\mathcal{X}}^{\beta,\alpha}\left([0,T]\times \Omega\right)}\right).
\end{equation*}
%We also define the space
%\begin{equation*}
%B\left([0,T];C_{\mathcal{X}}^{k+\alpha}\left(\Omega\right)\right):=\left\{\phi:[0,T]\to C_{\mathcal{X}}^{k+\alpha}\left(\Omega\right)~\bigg|~ \sup_{t\in[0,T]}\|\phi(t,\cdot)\|_{C_{\mathcal{X}}^{k+\alpha}\left(\Omega\right)}<+\infty\right\}.
%\end{equation*}

%To simplify the notations, we can abbreviate the H\"{o}lder norms in \eqref{Holder norm} and \eqref{para. Holder norm} respectively as $\|\cdot\|_{k+\alpha}$ and $\|\cdot\|_{l+\beta,k+\alpha}$.

The dual space of $C_{\mathcal{X}}^{k+\alpha}\left(\Omega\right)$ is denoted by
\begin{equation}\label{dual Holder space_Def}
C_{\mathcal{X}}^{-(k+\alpha)}\left(\Omega\right):=\left\{\psi~\big|~ \left\langle\psi, \phi\right\rangle_{C_{\mathcal{X}}^{-(k+\alpha)}\left(\Omega\right), C_{\mathcal{X}}^{k+\alpha}\left(\Omega\right)}<+\infty,\,\forall\, \phi\in C_{\mathcal{X}}^{k+\alpha}\left(\Omega\right)\right\}
\end{equation}
with the norm
\begin{equation*}\label{dual Holder space norm_Def}
\|\psi\|_{C_{\mathcal{X}}^{-(k+\alpha)}\left(\Omega\right)}:=\sup_{\|\phi\|_{C_{\mathcal{X}}^{k+\alpha}\left(\Omega\right)} \leq 1}\left\langle\psi, \phi\right\rangle_{C_{\mathcal{X}}^{-(k+\alpha)}\left(\Omega\right), C_{\mathcal{X}}^{k+\alpha}\left(\Omega\right)}\text{ for any }\psi \in C_{\mathcal{X}}^{-(k+\alpha)}\left(\Omega\right).
\end{equation*}
In the following derivation, we denote $\left\langle\psi, \phi\right\rangle_{C_{\mathcal{X}}^{-(k+\alpha)}\left(\Omega\right), C_{\mathcal{X}}^{k+\alpha}\left(\Omega\right)}$ simply as $\left\langle\psi, \phi\right\rangle$. %Using the same manner, we can also define the dual space of $C_{\mathcal{X},c}^{k+\alpha}\left(\Omega\right)$, denoted as $C_{\mathcal{X},c}^{-(k+\alpha)}\left(\Omega\right)$.
\begin{Rem}
Following the same way as above, we can also define the non-isotropic H\"{o}lder spaces on $\mathbb{T}_{\mathbb{G}}$ and their dual spaces, where it suffices to replace the distance $d_{cc}$ with $d_{cc}^{\mathbb{T}_{\mathbb{G}}}$.
\end{Rem}

\subsection{Theory of singular integral operators and continuity on H\"{o}lder spaces}\label{Subsec_2.2}

To start with, we introduce the following definitions.
\begin{Def}[Quasidistance]
Let $\Omega \subseteq \mathbb{R}^{n}$ be a set. A function $q: \Omega \times \Omega \to \mathbb{R}$ is called a quasidistance on $\Omega$ if there exists a constant $c_q \geqslant 1$ such that for any $x, y, z \in \Omega$:
$$
\begin{gathered}
q(x, y) \geqslant 0 \text { and } q(x, y)=0 \text{ if and only if } x=y ; \\
q(x, y)=q(y, x) ; \\
q(x, y) \leq c_q(q(x, z)+q(z, y)) .
\end{gathered}
$$
\end{Def}
\begin{Def}[Spaces of homogeneous type]
Let $(\Omega, q)$ be a space endowed with a quasidistance $q$ such that the $q$-balls are open with respect to the topology induced by $q$, and let $\mu$ be a positive Borel measure on $\Omega$ satisfying the doubling condition: there exists a constant $c_\mu>0$ such that
$$
\mu\left(\mathbf{B}_{2r}(x)\right) \leq c_\mu \cdot \mu\left(\mathbf{B}_r(x)\right) \text { for any } x \in \Omega, r>0,
$$
where the ball
$$
\mathbf{B}_r(x):=\{y \in \Omega~|~ q(x, y)<r\}\text{ for any }x\in\Omega,\,r>0.
$$
Then $(\Omega,q,\mu)$ is called a space of homogeneous type.
\end{Def}
\begin{Def}[H\"{o}lder spaces]
For any $\alpha>0$ and function $u: \Omega \to \mathbb{R}$, define
$$
\begin{gathered}\relax
[u]_{\mathbf{C}^\alpha(\Omega)}:=\sup \left\{\frac{|u(x)-u(y)|}{q(x, y)^\alpha}~\bigg|~ x, y \in \Omega, x \neq y\right\}, \\
\|u\|_{\mathbf{C}^\alpha(\Omega)}:=[u]_{\mathbf{C}^\alpha(\Omega)}+\|u\|_{L^{\infty}(\Omega)}, \\
\mathbf{C}^\alpha(\Omega):=\left\{u: \Omega \to \mathbb{R}~|~\|u\|_{\mathbf{C}^\alpha(\Omega)}<+\infty\right\}.
\end{gathered}
$$
\end{Def}
\begin{Def}
Let $(\Omega,q,\mu)$ be a space of homogeneous type. We say that a measurable function $\mathcal{K}(x, y): \Omega \times \Omega \to \mathbb{R}$ is a standard kernel on $\Omega$ if $\mathcal{K}$ satisfies
\begin{enumerate}[label=(\arabic*), leftmargin=*, itemindent=1.7em]
\item (``growth condition") for any $x, y \in \Omega$, there exists a constant $c_1>0$ such that
$$
|\mathcal{K}(x, y)| \leq \frac{c_1}{|\mathbf{B}_{q(x, y)}(x)|};
$$
\item (``mean value inequality") for any $x_0, x, y \in \Omega$ with $q\left(x_0, y\right) \geqslant M q\left(x_0, x\right)$, $M>1$, there exist constants $c_2,\beta>0$ such that
\begin{equation}\label{MVI_S}
\left|\mathcal{K}(x, y)-\mathcal{K}\left(x_0, y\right)\right| \leq \frac{c_2}{\left|\mathbf{B}_{q(x_0, y)}(x_0)\right|}\left(\frac{q\left(x_0, x\right)}{q\left(x_0, y\right)}\right)^\beta.
\end{equation}
\end{enumerate}
\end{Def}
In the following, we present some of the main tools that would be used in this paper.
\begin{Lem}[$\mathbf{C}^{\alpha}$ continuity of singular integral operators, see \mbox{\cite[Theorem 2.7]{07BB}}]\label{Lem_HCSIO}
Let $(\Omega,q,\mu)$ be a bounded space of homogeneous type, and let $\mathcal{K}(x, y)$ be a standard kernel. Define
\begin{equation}\label{K_varepsilon}
K_{\varepsilon} f(x):=\int_{q^{\prime}(x, y)>\varepsilon} \mathcal{K}(x, y) f(y) d\mu(y)
\end{equation}
where $q'$ is any quasidistance on $\Omega$, equivalent to $q$. Assume that
\begin{enumerate}[label=(\arabic*), leftmargin=*, itemindent=1.7em]
\item for every $f \in \mathbf{C}^\alpha(\Omega)$ and $x \in \Omega$ the following limit exists:
$$
K f(x)=P V \int_\Omega \mathcal{K}(x, y) f(y) d\mu(y)=\lim _{\varepsilon \to 0} K_{\varepsilon} f(x);
$$
\item (``cancellation properties") there exist constants $c_3,c_4>0$ such that
\begin{equation}\label{IP_S}
\left|\int_{q^{\prime}(x, y)>r} \mathcal{K}(x, y) d\mu(y)\right| \leq c_3
\end{equation}
for any $r>0$, where $c_3$ is independent of $r$, and
\begin{equation}\label{CP_S}
\lim_{\varepsilon \to 0} \left|\int_{q^{\prime}(x, y)>\varepsilon} \mathcal{K}(x, y) d\mu(y)-\int_{q^{\prime}(x_0, y)>\varepsilon} \mathcal{K}\left(x_0, y\right) d\mu(y)\right|\leq c_4 q(x, x_0)^\gamma
\end{equation}
for some $\gamma \in(0,1]$, where $q^{\prime}$ is the same quasidistance appearing in \eqref{K_varepsilon}.
\end{enumerate}
Then the integral operator $K$ is continuous on $\mathbf{C}^\alpha(\Omega)$. More precisely,
$$
[K f]_{\mathbf{C}^\alpha(\Omega)} \leq c\sum_{i=1}^{4}c_i\|f\|_{\mathbf{C}^\alpha(\Omega)} \text { for every } \alpha \leq \gamma, \alpha<\beta,
$$
where $\gamma$ is the constant in \eqref{CP_S} and $\beta$ is the constant in \eqref{MVI_S}, and
\begin{equation*}\label{L^inf E}
\|K f\|_{L^\infty(\Omega)} \leq c_{R,\alpha}(c_1+c_3)\|f\|_{\mathbf{C}^\alpha(\Omega)},\text { where } R=\operatorname{diam} \Omega.
\end{equation*}
\end{Lem}
\begin{Lem}[$\mathbf{C}^{\alpha}$ continuity of fractional integral operators, see \mbox{\cite[Theorem 2.11]{07BB}}]\label{Lem_HCFIO}
Let $(\Omega,q,\mu)$ be a bounded space of homogeneous type, and assume that $\Omega$ does not contain atoms (that is, points of positive measure). Let $\mathcal{K}_\delta(x, y)$ be a ``fractional integral kernel", that is,
\begin{enumerate}[label=(\arabic*), leftmargin=*, itemindent=1.7em]
\item (``growth condition")
\begin{equation}\label{GC_F}
|\mathcal{K}_\delta(x, y)| \leq \frac{c_1 q(x, y)^\delta}{|\mathbf{B}_{q(x, y)}(x)|}
\end{equation}
for any $x, y \in \Omega$ and some $c_1, \delta>0$;
\item (``mean value inequality")
\begin{equation}\label{MVI_F}
\left|\mathcal{K}_\delta(x, y)-\mathcal{K}_\delta\left(x_0, y\right)\right| \leq \frac{c_2 q\left(x_0, y\right)^\delta}{\left|\mathbf{B}_{q(x_0, y)}(x_0)\right|}\left(\frac{q\left(x_0, x\right)}{q\left(x_0, y\right)}\right)^\beta
\end{equation}
for any $x_0, x, y \in \Omega$ with $q\left(x_0, y\right) \geqslant M q\left(x_0, x\right)$, some $c_2, \beta>0$ and $M>1$.
\end{enumerate}
Then the integral operator
$$
I_\delta f(x)=\int_\Omega \mathcal{K}_\delta(x, y) f(y) d\mu(y)
$$
is continuous on $\mathbf{C}^\alpha(\Omega)$. More precisely,
$$
[I_\delta f]_{\mathbf{C}^\alpha(\Omega)} \leq c_{R,\delta}\sum_{i=1}^{2}c_i\|f\|_{\mathbf{C}^\alpha(\Omega)} \text { for every } \alpha<\min (\beta, \delta),
$$
and
\begin{equation*}
\|I_\delta f\|_{L^\infty(\Omega)} \leq c_{R,\delta}c_1\|f\|_{L^\infty(\Omega)},\text { where } R=\operatorname{diam} \Omega.
\end{equation*}
\end{Lem}
\begin{Rem}
In Lemma \ref{Lem_HCFIO}, we do not need to fulfill \eqref{CP_S} and \eqref{IP_S}, this is because properties \eqref{GC_F} and \eqref{MVI_F} can imply these cancellation properties.
\end{Rem}
Now, we consider the fundamental solution $\Gamma_0(t,x;s,y):\mathbb{R}\times\mathbb{G}\times\mathbb{R}\times\mathbb{G}\to\mathbb{R}$ for the heat operator
\begin{equation}\label{heat operator}
\mathcal{H}=\partial_t-\Delta_{\mathcal{X}}
\end{equation}
on $\mathbb{R}\times\mathbb{G}$ to be an integral kernel, which is nonnegative and vanishes when $t\leq s$, having the form $\Gamma_0(t,x;s,y)=\Gamma_0(t-s,y^{-1}\circ x)$. For more properties of $\Gamma_0$ such as existence as well as Gaussian estimates, one can refer to \cite{02BLU}. Here, we regard $\mathbb{R}\times\mathbb{G}$ as a homogeneous group with translation
\begin{equation*}
(t,x)*(s,y):=(t+s,x\circ y),
\end{equation*}
dilation
\begin{equation*}
D_{\lambda}^p(t,x):=\left(\lambda^2t,D_{\lambda}(x)\right)
\end{equation*}
and homogeneous dimension $Q'=Q+2$, where $Q$ is the homogeneous dimension of $\mathbb{G}$. We introduce homogeneous functions with respect to the dilation $D_{\lambda}^p$.
\begin{Def}[$D_{\lambda}^p$-homogeneous function]\label{Def_homogeneous function}
A real function $f$ defined on $\mathbb{R}^{n}$ is called the $D_{\lambda}^p$-homogeneous of degree $\delta\in \mathbb{R}$ (or simply ``$\delta$-homogeneous'') if $f\not\equiv 0$ and $f$ satisfies
$$
f(D_{\lambda}^p(x))=\lambda^{\delta}f(x)
$$
for any $\lambda>0$ and $x \in \mathbb{R}^{n}$.
\end{Def}
\begin{Rem}
If $f$ is a $\gamma$-homogeneous function, smooth outside the origin, and $P$ is a $\beta$-homogeneous differential operator, then $Pf$ is a $(\gamma-\beta)$-homogeneous function (see \cite[p. 107]{22BB}).
\end{Rem}
Denote $d_p$ by the parabolic Carnot-Carath\'{e}odory distance, defined as
\begin{equation*}
d_p((t,x),(s,y)):=\sqrt{|t-s|+d_{cc}(x,y)^2},\quad(t,x),(s,y)\in\mathbb{R}\times\mathbb{R}^{n}.
\end{equation*}
Additionally, for any $(t,x),(s,y)\in\mathbb{R}\times\mathbb{R}^{n}$, we denote the ball
$$
B^p((t,x);(s,y)):=B_{d_{p}((t,x),(s,y))}(t,x).
$$
Then the following lemma holds.
\begin{Lem}\label{Lem_mathcal K_I,q,l_singular integral theory}
For any multi-index $I=\left(i_1, i_2, \cdots, i_p\right),i_j \in\{1,\ldots,n_1\}$ with the length $|I|=p\in\mathbb{N}$ and $q,l\in\mathbb{N}$, we define a kernel as
$$
\mathcal{K}_{I,q,l}(t,x;s,y):=r_{l}(y^{-1}\circ x)X_I\partial_t^q\Gamma_0\left(t-s,y^{-1}\circ x\right),\quad(t,x),(s,y)\in\mathbb{R}\times\mathbb{R}^{n},
$$
where $r_{l}(\cdot)$ is a $l$-homogeneous polynomial. Let $\Omega$ be a bounded subset of $\mathbb{R}\times\mathbb{R}^{n}$. Then there exists a constant $c>0$, depending only on $|I|$, $q$, $\max_{\|x\|=1}\left|r_l(x)\right|$ and $\mathbb{G}$, such that the integral kernel $\mathcal{K}_{I,q,l}:\Omega\times\Omega\to\mathbb{R}$ satisfies the conditions in Lemma \ref{Lem_HCSIO} and Lemma \ref{Lem_HCFIO} as follows:
\begin{enumerate}[label=(\arabic*), leftmargin=*, itemindent=1.7em]
\item (growth condition)
\begin{equation*}
|\mathcal{K}_{I,q,l}(t,x;s,y)| \leq \frac{c}{d_p((t,x),(s,y))^{Q+|I|+2q-l}} \leq c \frac{d_p((t,x),(s,y))^{2-|I|-2q+l}}{|B^p((t,x);(s,y))|};
\end{equation*}\label{mathcal K_I,q,l_GC}
\item (mean value inequality)
\begin{align*}
|\mathcal{K}_{I,q,l}(t,x;s,y)-\mathcal{K}_{I,q,l}(t_1,x_1;s,y)| & \leq c \frac{d_p((t_1,x_1),(t,x))}{d_p((t_1,x_1),(s,y))^{Q+|I|+2q-l+1}}\\
& \leq c \frac{d_p((t_1,x_1),(t,x))^{2-|I|-2q+l}}{|B^p((t_1,x_1);(s,y))|} \cdot\left(\frac{d_p((t_1,x_1),(t,x))}{d_p((t_1,x_1),(s,y))}\right)
\end{align*}
when $d_p((t_1,x_1),(s,y)) > 4 d_p((t_1,x_1),(t,x))$.\label{mathcal K_I,q,l_MVI}
\item If $|I|+2q-l=2$, then $\mathcal{K}_{I,q,l}$ also satisfies (cancellation properties)
\begin{equation*}
\left|\int_{r<d_p^{\prime}((t,x),(s,y))<R} \mathcal{K}_{I,q,l}(t,x;s,y)dyds\right| \leq c
\end{equation*}
with $c$ independent of $r, R>0$, and
\begin{align*}
\lim_{\varepsilon \to 0} \bigg| & \int_{d_p^{\prime}((t,x),(s,y))>\varepsilon} \mathcal{K}_{I,q,l}(t,x;s,y)dyds-\int_{d_p^{\prime}((t_1,x_1),(s,y))>\varepsilon} \mathcal{K}_{I,q,l}(t_1,x_1;s,y)dyds\bigg| \\
\leq & c d_p((t_1,x_1),(t,x))^\gamma
\end{align*}
for any constant $\gamma$, where $d_p^{\prime}$ is any quasidistance on $\Omega$, equivalent to $d_p$.\label{mathcal K_I,q,l_CP}
\end{enumerate}
\end{Lem}
\begin{proof}
It can be known from the Gaussian estimates of $\Gamma_0$ (see \cite[Theorem 2.5]{02BLU}) that
\begin{equation*}
|X_I\partial_t^q\Gamma_0\left(t,x\right)| \leq c_{|I|,q}t^{-\frac{Q+|I|+2q}{2}}\exp\left(-\frac{\left\|x\right\|^2}{ct}\right),\,(t,x)\in\mathbb{R}_+\times\mathbb{R}^n,
\end{equation*}
where $c>1$ depends only on $\mathbb{G}$ and $c_{|I|,q}>0$ depends only on $|I|$, $q$ and $\mathbb{G}$. Hence, we have
\begin{align*}
|\mathcal{K}_{I,q,l}(t,x;s,y)| \leq & c_{|I|,q}\max_{\|x\|=1}\left|r_l(x)\right| \left\|y^{-1}\circ x\right\|^{l}(t-s)^{-\frac{Q+|I|+2q}{2}}\exp\left(-\frac{\left\|y^{-1}\circ x\right\|^2}{c(t-s)}\right) \\
\leq & c_{|I|,q,r_l}\frac{\left(\frac{\left\|y^{-1}\circ x\right\|^2+t-s}{t-s}\right)^{\frac{Q+|I|+2q}{2}}\exp\left(-\frac{\left\|y^{-1}\circ x\right\|^2}{c(t-s)}\right)}{\left(\left\|y^{-1}\circ x\right\|^2+t-s\right)^{\frac{Q+|I|+2q-l}{2}}} \\
\leq & \frac{c_{|I|,q,r_l,Q}}{d_p((t,x),(s,y))^{Q+|I|+2q-l}}\text{ for }t>s,
\end{align*}
by means of the fact that
\begin{equation*}
\sup_{x\in[0,+\infty)}(1+x)^{j}\exp(-x)=c_j<+\infty,\,j \geq 0.
\end{equation*}
Thus condition \ref{mathcal K_I,q,l_GC} is proved.

Next let us prove condition \ref{mathcal K_I,q,l_MVI}. Fix any $(t_1,x_1),(s,y)\in\Omega$. Let $d_{p}((t_1,x_1),(s,y))=2R$, then $d_{p}((t_1,x_1),(s,y)) > 4 d_{p}((t_1,x_1),(t,x))$ if and only if $(t,x) \in B_{\frac{R}{2}}^p(t_1,x_1)$. Let $\varphi$ be a smooth cutoff function such that $0 \leq \varphi \leq 1$, $\varphi \equiv 1$ on $B_{\frac{R}{2}}^p(t_1,x_1)$, $\operatorname{supp}\left(\varphi\right) \subset B_{R}^p(t_1,x_1)$, and for any multi-index $I=(i_1,i_2,\ldots,i_{k_1}), i_j\in\{1,\ldots,n_1\}$ with the length $|I|=k_1\in\mathbb{N}$ and $k_2\in\mathbb{N}$,
\begin{equation*}
\left|\partial_t^{k_2} X_I \varphi\right| \leq c_{k_1,k_2}R^{-(k_1+2k_2)}
\end{equation*}
(see \cite[Lemma 6.2]{07BB}). We define
\begin{equation*}
h(t,x):=\mathcal{K}_{I,q,l}(t-s,y^{-1}\circ x)\varphi(t,x).
\end{equation*}
Clearly, $h \in C_{\mathcal{X}}^{1,1}\left(B_{R}^p(t_1,x_1)\right)$. For any $(t,x) \in B_{\frac{R}{2}}^p(t_1,x_1)$, using \cite[Proposition 4.2(ii)]{07BB}, we have
\begin{equation}\label{h(t,x)-h(t_1,x_1)}
|h(t,x)-h(t_1,x_1)| \leq\sup_{(t,x) \in B_{R}^p(t_1,x_1)}\left(\left|D_{\mathcal{X}}h(t,x)\right|+R\left|\partial_t h(t,x)\right|\right)\cdot d_{p}((t,x),(t_1,x_1)).
\end{equation}
For any $i \in \{1,\ldots,n_1\}$,
\begin{align*}
X_i h(t,x)= & r_{l}(y^{-1}\circ x) X_iX_I\partial_t^q\Gamma_0(t-s,y^{-1}\circ x)\varphi(t,x)+\mathcal{K}_{I,q,l}(t-s,y^{-1}\circ x)X_i\varphi(t,x) \\
& + X_ir_{l}(y^{-1}\circ x) X_I\partial_t^q\Gamma_0(t-s,y^{-1}\circ x)\varphi(t,x) \\
=: & I+II+III.
\end{align*}
It is easy to find that for $d_{p}((t,x),(t_1,x_1))<R$,
\begin{align*}
|I| \leq & c_{|I|,q,r_l}\frac{\left(\frac{\left\|y^{-1}\circ x\right\|^2+t-s}{t-s}\right)^{\frac{Q+|I|+2q+1}{2}}\exp\left(-\frac{\left\|y^{-1}\circ x\right\|^2}{c(t-s)}\right)}{\left(\left\|y^{-1}\circ x\right\|^2+t-s\right)^{\frac{Q+|I|+2q+1-l}{2}}} \\
\leq & \frac{c_{|I|,q,r_l,Q}}{d_p((t,x),(s,y))^{Q+|I|+2q+1-l}} \\
\leq & \frac{c_{|I|,q,r_l,Q}}{d_p((t_1,x_1),(s,y))^{Q+|I|+2q+1-l}},
\end{align*}
and
\begin{align*}
|II| \leq & \frac{c}{R}\frac{c_{|I|,q,r_l,Q}}{d_p((t,x),(s,y))^{Q+|I|+2q-l}} \\
\leq & \frac{c_{|I|,q,r_l,Q}}{d_p((t_1,x_1),(s,y))^{Q+|I|+2q-l+1}},
\end{align*}
where we have used the fact that
$$
d_{p}((t,x),(s,y)) \geq d_{p}((t_1,x_1),(s,y))-d_{p}((t,x),(t_1,x_1)) \geq R
$$
and $d_{p}((t_1,x_1),(s,y))=2R$. Note that $X_ir_{l}(\cdot)$ is a $(l-1)$-homogeneous polynomial when $l\geq 1$, and $X_ir_{l}(\cdot)\equiv 0$ when $l=0$. Then, we likewise have
\begin{equation*}
|III| \leq \frac{c_{|I|,q,r_l,Q}}{d_p((t_1,x_1),(s,y))^{Q+|I|+2q-(l-1)}}.
\end{equation*}
Therefore, for $d_{p}((t,x),(t_1,x_1))<R$, we have
\begin{equation*}
\left|X_i h(t,x)\right| \leq \frac{c_{|I|,q,r_l,Q}}{d_p((t_1,x_1),(s,y))^{Q+|I|+2q-l+1}}.
\end{equation*}
Since
\begin{equation*}
\partial_t h(t,x)=r_{l}(y^{-1}\circ x) X_I\partial_t^{q+1}\Gamma_0(t-s,y^{-1}\circ x)\varphi(t,x)+\mathcal{K}_{I,q,l}(t-s,y^{-1}\circ x)\partial_t\varphi(t,x),
\end{equation*}
In the same way, we can obtain that for $d_{p}((t,x),(t_1,x_1))<R$,
\begin{equation*}
\left|\partial_t h(t,x)\right| \leq \frac{c_{|I|,q,r_l,Q}}{R d_p((t_1,x_1),(s,y))^{Q+|I|+2q-l+1}}
\end{equation*}
Combining the above estimates of $\left|X_i h(t,x)\right|$ and $\left|\partial_t h(t,x)\right|$ with \eqref{h(t,x)-h(t_1,x_1)}, we finally get that condition \ref{mathcal K_I,q,l_MVI} holds.

To prove condition \ref{mathcal K_I,q,l_CP} when $|I|+2q-l=2$, we note that $\Gamma_0(t,x)$ is $-Q$-homogeneous with respect to the dilations $D_{\lambda}^p$ (see \cite[Theorem 3.7]{02BLU}). Hence $\mathcal{K}_{I,q,l}(t,x)=\mathcal{K}_{I,q,l}(t,x;0,0)$ is a $-Q'$-homogeneous function, where $Q'=Q-l+|I|+2q=Q+2$. According to a known property of homogeneous distributions of degree $-Q'$ in homogeneous groups (see \cite[Proposition 1.8]{75Fo}), we have
\begin{equation*}
\int_{r<d_p^{\prime}((t,x),(s,y))<R} \mathcal{K}_{I,q,l}(t-s,y^{-1}\circ x)dyds=\int_{r'<d_p((t,x),(0,0))<R'} \mathcal{K}_{I,q,l}(t,x)dxdt=0
\end{equation*}
for any $R>r>0$ and quasidistance $d_p^{\prime}$ on $\Omega$, equivalent to $d_p$. Hence, the cancellation properties in condition \ref{mathcal K_I,q,l_CP} hold since $\Omega$ is bounded.
\end{proof}

\subsection{A priori Schauder estimates for the Cauchy problem on Carnot groups with
rough coefficients}

We come to present some a priori regularity estimates of the solution to the following linear degenerate Cauchy problem
\begin{equation}\label{LHJB in R^n}
\begin{cases}
\partial_t z(t,x)-\Delta_{\mathcal{X}} z(t,x)+b(t,x)\cdot D_{\mathcal{X}} z(t,x)+c(t,x)z(t,x)=f(t,x),& \text {in }(0, T] \times \mathbb{R}^{n}, \\
z(0,x)=g(x),& \text {in }\mathbb{R}^{n},
\end{cases}
\end{equation}
as detailed in the following propositions.
\begin{Prop}[see {\cite[Theorem 1.3]{25JWY}}]\label{Prop_LHJB in R^n UE&HE}
Let $\alpha\in(0,1)$. Assume $b(t,x)$, $c(t,x)$, $f(t,x)$ are bounded and continuous on $[0,T]\times\mathbb{R}^{n}$, $b\in B\left((0,T);C_{\mathcal{X},loc}^{\alpha}\left(\mathbb{R}^{n};\mathbb{R}^{n_1}\right)\right)$, $c$, $f\in B\left((0,T);C_{\mathcal{X},loc}^{\alpha}\left(\mathbb{R}^{n}\right)\right)$ and $g(x)\in C_{\mathcal{X}}^{0+1}\left(\mathbb{R}^{n}\right)$. Suppose $z(t,x)\in C_{{\mathcal{X}}}^{1,2} \left((0,T]\times\mathbb{R}^{n}\right)\cap C\left([0,T]\times\mathbb{R}^{n}\right)$ is a solution to the Cauchy problem \eqref{LHJB in R^n}, satisfying $z\in B\left((0,T);C_{\mathcal{X},loc}^{1+\alpha}\left(\mathbb{R}^{n}\right)\right)$ and $\sup_{t\in(0,T]}\|z(t,\cdot)\|_{C_{\mathcal{X}}^{1}\left(\mathbb{R}^{n}\right)}<+\infty$. The following conclusions hold:
\begin{enumerate}[label=(\arabic*), leftmargin=*, itemindent=1.7em]
\item $z$ satisfies
\begin{equation}\label{LHJB in R^n_prior C^1 estimate}
\sup_{t \in(0, T]}\|z(t, \cdot)\|_{C_{\mathcal{X}}^{1}\left(\mathbb{R}^{n}\right)} \leq C\left(\left\|g\right\|_{C_{\mathcal{X}}^{0+1}\left(\mathbb{R}^{n}\right)} +\|f\|_{L^{\infty}\left((0,T)\times\mathbb{R}^{n}\right)}\right).
\end{equation}
\item For any constant $\epsilon \in (0,T)$, $z$ satisfies
\begin{equation*}
\sup _{\substack{t \neq t'\\t,t' \in [\epsilon,T]}} \frac{\left\|z\left(t^{\prime}, \cdot\right)-z(t, \cdot)\right\|_{L^{\infty}\left(\mathbb{R}^{n}\right)}}{\left|t^{\prime}-t\right|} \leq C\left(\epsilon^{-1}\|g\|_{L^{\infty}\left(\mathbb{R}^{n}\right)} +\left\|g\right\|_{C_{\mathcal{X}}^{0+1}\left(\mathbb{R}^{n}\right)} +\|f\|_{L^{\infty}\left((0,T)\times\mathbb{R}^{n}\right)}\right).
\end{equation*}
Moreover, if $g(x)\in C_{\mathcal{X}}^{2+\alpha}\left(\mathbb{R}^{n}\right)$, then $z$ satisfies
\begin{equation}\label{LHJB in R^n_HSE_L^infty}
\sup _{\substack{t \neq t'\\t,t' \in [0,T]}} \frac{\left\|z\left(t^{\prime}, \cdot\right)-z(t, \cdot)\right\|_{L^{\infty}\left(\mathbb{R}^{n}\right)}}{\left|t^{\prime}-t\right|^{\frac{1}{2}}} \leq C\left(\left\|g\right\|_{C_{\mathcal{X}}^{1}\left(\mathbb{R}^{n}\right)} +\|f\|_{L^{\infty}\left((0,T)\times\mathbb{R}^{n}\right)}\right).
\end{equation}
\end{enumerate}
Here, the constants $C>0$ depend on $\mathbb{G}$,  $\|b\|_{L^{\infty}\left((0,T)\times\mathbb{R}^{n}\right)}$, $\|c\|_{L^{\infty}\left((0,T)\times\mathbb{R}^{n}\right)}$ and $T$ only.
\end{Prop}
\begin{Cor}[Existence and uniqueness of the $C_{\mathcal{X}}^{1}$-bounded solution]\label{Cor_LHJB existence and uniqueness}
Let $\alpha\in(0,1)$. Assume $b(t,x)$, $c(t,x)$ and $f(t,x)$ are bounded and continuous on $[0,T]\times\mathbb{R}^{n}$, $b\in B\left([0,T];C_{\mathcal{X}}^{\alpha}\left(\mathbb{R}^{n};\mathbb{R}^{n_1}\right)\right)$, $c\in B\left([0,T];C_{\mathcal{X}}^{\alpha}\left(\mathbb{R}^{n}\right)\right)$, $f\in B\left((0,T);C_{\mathcal{X}}^{\alpha}\left(\mathbb{R}^{n}\right)\right)$ and $g(x)\in C_{\mathcal{X}}^{2+\alpha}\left(\mathbb{R}^{n}\right)$. Then there exists a unique solution $z(t,x) \in C_{{\mathcal{X}}}^{1,2} \left([0,T]\times\mathbb{R}^{n}\right)$ to the Cauchy problem \eqref{LHJB in R^n}, satisfying $\sup_{t\in(0,T]}\|z(t,\cdot)\|_{C_{\mathcal{X}}^{1}\left(\mathbb{R}^{n}\right)}<+\infty$.
\end{Cor}
\begin{proof}
See \cite[Corollary 1.1]{25JWY} for a detailed proof, where the uniqueness is obtained directly from \eqref{LHJB in R^n_prior C^1 estimate}.
\end{proof}
\begin{Prop}[see {\cite[Theorem 1.4]{25JWY}}]\label{Prop_LHJB in R^n USE&HSE}
Let $k\in\mathbb{Z}_+$, $\alpha\in(0,1)$, and $b(t,x)$, $c(t,x)$, $f(t,x)$ be bounded and continuous on $[0,T]\times\mathbb{R}^{n}$, $b\in B\left((0,T);C_{\mathcal{X}}^{k-1+\alpha}\left(\mathbb{R}^{n};\mathbb{R}^{n_1}\right)\right)$, $c$, $f\in B\left((0,T);C_{\mathcal{X}}^{k-1+\alpha}\left(\mathbb{R}^{n}\right)\right)$ and $g(x)\in C_{\mathcal{X}}^{k+\alpha}\left(\mathbb{R}^{n}\right)$. Suppose $z(t,x)\in C_{{\mathcal{X}}}^{1,2} \left((0,T]\times\mathbb{R}^{n}\right)\cap C\left([0,T]\times\mathbb{R}^{n}\right)$ is a solution to the Cauchy problem \eqref{LHJB in R^n}, satisfying $z\in
B\left([0,T];C_{\mathcal{X}}^{k+\alpha}\left(\mathbb{R}^{n}\right)\right)$. Then the following conclusions hold:
\begin{enumerate}[label=(\arabic*), leftmargin=*, itemindent=1.7em]
\item $z$ satisfies
\begin{equation}\label{LHJB in R^n_prior USE}
\sup _{t \in[0, T]}\|z(t, \cdot)\|_{C_{\mathcal{X}}^{k+\alpha}\left(\mathbb{R}^{n}\right)} \leq C\left(\left\|g\right\|_{C_{\mathcal{X}}^{k+\alpha}\left(\mathbb{R}^{n}\right)}+\sup_{t\in(0,T)}\|f(t, \cdot)\|_{C_{\mathcal{X}}^{k-1+\alpha}\left(\mathbb{R}^{n}\right)}\right)
\end{equation}
\item For any constant $\epsilon \in (0,T)$, $z$ satisfies
\begin{equation*}\label{LHJB in R^n_prior HSE}
\sup _{\substack{t \neq t'\\t,t' \in [\epsilon,T]}} \frac{\left\|z\left(t^{\prime}, \cdot\right)-z(t, \cdot)\right\|_{C_{\mathcal{X}}^{k+\alpha}\left(\mathbb{R}^{n}\right)}}{\left|t^{\prime}-t\right|^{\frac{1}{2}}} \leq C\left(\epsilon^{-\frac{1}{2}}\left\|g\right\|_{C_{\mathcal{X}}^{k+\alpha}\left(\mathbb{R}^{n}\right)}+\sup_{t \in(0, T)}\|f(t, \cdot)\|_{C_{\mathcal{X}}^{k-1+\alpha}\left(\mathbb{R}^{n}\right)}\right).
\end{equation*}
Moreover, if $g(x)\in C_{\mathcal{X}}^{k+1+\alpha}\left(\mathbb{R}^{n}\right)$, then $z$ satisfies
\begin{equation*}
\sup _{\substack{t \neq t'\\t,t' \in [0,T]}} \frac{\left\|z\left(t^{\prime}, \cdot\right)-z(t, \cdot)\right\|_{C_{\mathcal{X}}^{k+\alpha}\left(\mathbb{R}^{n}\right)}}{\left|t^{\prime}-t\right|^{\frac{1}{2}}} \leq C\left(\left\|g\right\|_{C_{\mathcal{X}}^{k+1+\alpha}\left(\mathbb{R}^{n}\right)}+\sup_{t \in(0, T)}\|f(t, \cdot)\|_{C_{\mathcal{X}}^{k-1+\alpha}\left(\mathbb{R}^{n}\right)}\right).
\end{equation*}
\end{enumerate}
Here, the constants $C>0$ depend on $\mathbb{G}$,  $\alpha$, $k$, $T$, $\sup_{t\in(0,T)}\|b(t,\cdot)\|_{C_{\mathcal{X}}^{k-1+\alpha}\left(\mathbb{R}^{n};\mathbb{R}^{n_1}\right)}$ and $\sup_{t\in(0,T)}\|c(t, \cdot)\|_{C_{\mathcal{X}}^{k-1+\alpha}\left(\mathbb{R}^{n}\right)}$ only.
\end{Prop}
It is worth mentioning that several tricks are needed in the process of proving Proposition \ref{Prop_LHJB in R^n USE&HSE}, which are also useful for this paper.
\begin{Lem}\label{Lem_HE existence and uniqueness}
Let $\alpha\in(0,1)$. Assume $f:[0,T]\times\mathbb{R}^{n}\to\mathbb{R}$ and $g:\mathbb{R}^{n}\to\mathbb{R}$ are bounded and vanish as $|x|\to+\infty$, and $f\in C\left([0,T]\times\mathbb{R}^{n}\right)\cap B\left((0,T);C_{\mathcal{X},loc}^{\alpha}\left(\mathbb{R}^{n}\right)\right)$, $g\in C\left(\mathbb{R}^{n}\right)$. Then the function
$$
z(t,x)=\int_{\mathbb{R}^{n}}\Gamma_0(t,y^{-1}\circ x)g(y)d y + \int_{0}^{t}\int_{\mathbb{R}^{n}}\Gamma_0(t-s,y^{-1}\circ x)f(s,y)d y d s
$$
belongs to the class
$$
C_{{\mathcal{X}}}^{1,2} \left((0, T] \times \mathbb{R}^{n} \right) \cap C\left([0, T] \times \mathbb{R}^{n}\right)
$$
and is the unique solution to the heat equation
\begin{equation}\label{heat equ.}
\begin{cases}
\partial_tz(t,x)-\Delta_{\mathcal{X}}z(t,x)=f(t,x),& \text {in }(0, T] \times \mathbb{R}^{n}, \\
z(0,x)=g(x),& \text {in }\mathbb{R}^{n}.
\end{cases}
\end{equation}
Here, $\Gamma_0$ is the fundamental solution for the heat operator $\mathcal{H}=\partial_t-\Delta_{\mathcal{X}}$.
\end{Lem}
\begin{proof}
It is directly obtained from \cite[Proposition 4.2(1)]{25JWY} that
$$
z(t,x)\in C_{{\mathcal{X}}}^{1,2} \left((0, T] \times \mathbb{R}^{n} \right) \cap C\left([0, T] \times \mathbb{R}^{n}\right)
$$
is a solution to the equation \eqref{heat equ.}. To prove the uniqueness, we shall verify that for any $t\in[0,T]$,
\begin{equation}\label{z vanishes at infty}
z(t,x)\to 0 \text{ as }|x|\to+\infty.
\end{equation}
Indeed, by the dominated convergence theorem and the fact $\|x\circ y^{-1}\|\geq\|x\|-\|y\|$, we have
\begin{equation*}
\int_{0}^{t}\int_{\mathbb{R}^{n}}\Gamma_0\left(t-s,y^{-1}\circ x\right)f(s,y)dyds=\int_{0}^{t}\int_{\mathbb{R}^{n}}\Gamma_0\left(t-s,y\right)f(s,x\circ y^{-1})dyds\to 0
\end{equation*}
as $|x|\to+\infty$. In the same way, we also have
\begin{equation*}
\int_{\mathbb{R}^{n}}\Gamma_0\left(t,y^{-1}\circ x\right)g(y)dyds=\int_{\mathbb{R}^{n}}\Gamma_0\left(t,y\right)g(x\circ y^{-1})dyds\to 0
\end{equation*}
as $|x|\to+\infty$. Hence \eqref{z vanishes at infty} holds.

Finally, we apply the weak maximum principle in infinite strip (see \cite[Corollary 13.2]{10BBLU}) to obtain that the solution to the equation \eqref{heat equ.}, which vanishes as $|x|\to+\infty$, is unique.
\end{proof}
\begin{Lem}\label{Lem_X_i X_J T f}
For any multi-index $I=(i_1,i_2,\ldots,i_p), i_j\in\{1,\ldots,n_1\}$ with the length $|I|=p\in\mathbb{N}$ and $q\in\mathbb{N}$, let $\mathcal{K}_{I,q}$ be the kernel defined as
\begin{equation*}
\mathcal{K}_{I,q}(t,x;s,y):=X_I \partial_t^q \Gamma_0(t-s,y^{-1}\circ x), \quad(t,x),(s,y)\in\mathbb{R}\times\mathbb{R}^{n},
\end{equation*}
where $\Gamma_0$ is the fundamental solution for the operator $\mathcal{H}=\partial_t-\Delta_{\mathcal{X}}$. Then for any $k\in\mathbb{N}$ and multi-index $J=(j_1,j_2,\ldots,j_k),j_l\in\{1,\ldots,n_1\}$, there exist kernels $\mathcal{K}_{I,q}^P(t,x;s,y)$, $P=(p_1,p_2,\ldots,p_k),p_i\in\{1,\ldots,n_1\}$ having the form
\begin{equation}\label{mathcal K_I,q^P}
\mathcal{K}_{I,q}^{P}(t,x;s,y)=
\begin{cases}
\sum_{i=1}^{M}\mathcal{R}_{i}^{J,P}(y^{-1}\circ x)X_{l_1^{J,P}}\cdots X_{l_i^{J,P}} \mathcal{K}_{I,q}(t-s,y^{-1}\circ x), & \mbox{if } \dim(P)\geq1, \\
\mathcal{K}_{I,q}(t-s,y^{-1}\circ x), & \mbox{if } \dim(P)=0
\end{cases}
\end{equation}
for a certain family of finite number of $i$-homogeneous polynomials $\mathcal{R}_i^{J,P},i\in\{1,\ldots,M\}$ and left-invariant vector fields $\{X_{l_i^{J,P}}\}_{i=1}^{M},l_i^{J,P}\in\{1,\ldots,n_1\}$ with $M=M(J,P)\in\mathbb{Z}_+$, such that the following holds:
\begin{equation}\label{int |mathcal K_I,q^P|}
\left|\mathcal{K}_{I,q}^{P}(t,x;s,y)\right| \leq c_{|I|,q,J,P}(t-s)^{-\frac{Q+|I|+2q}{2}}\exp\bigg(-\frac{\left\|y^{-1}\circ x\right\|^2}{c(t-s)}\bigg),\,t>s,\,x,y\in\mathbb{R}^n;
\end{equation}
if $|I|\in\{0,1\}$ and $q=0$, then for any $f \in B\left((0,t);C_{\mathcal{X}}^{k+\alpha}\left(\mathbb{R}^{n}\right)\right)$,
\begin{equation}\label{X_J int_0^t T f}
X_J \int_{0}^{t}\int_{\mathbb{R}^{n}} \mathcal{K}_{I,0}(t,x;s,y)f(s,y)dyds =\sum_{p_1,\ldots,p_k=1}^{n_1}\int_{0}^{t}\int_{\mathbb{R}^{n}} \mathcal{K}_{I,0}^P(t,x;s,y)X_Pf(s,y)dyds
\end{equation}
for any $t>0$ and $x\in\mathbb{R}^{n}$. Moreover,
\begin{equation}\label{X_i X_J int_0^t T f}
X_i X_J \int_{0}^{t}\int_{\mathbb{R}^{n}} \mathcal{K}_{I,0}(t,x;s,y)f(s,y)dyds =\sum_{p_1,\ldots,p_k=1}^{n_1}\int_{0}^{t}\int_{\mathbb{R}^{n}} X_i \mathcal{K}_{I,0}^P(t,\cdot;s,y)(x) X_Pf(s,y)dyds
\end{equation}
for any $i\in\{1,\ldots,n_1\}$, $t>0$ and $x\in\mathbb{R}^{n}$.
\end{Lem}
\begin{proof}
Clearly, \eqref{int |mathcal K_I,q^P|} holds true due to the Gaussian estimates of $\Gamma_0$. It is straightforward from \cite[Theorem 1.2(2)]{25JWY} that \eqref{X_J int_0^t T f} holds. In order to prove \eqref{X_i X_J int_0^t T f}, we simply refer to that for (3.28) in the proof of \cite[Theorem 1.2]{25JWY}.
\end{proof}
Specifically, following the method of proof for \eqref{LHJB in R^n_prior USE} with $g=0$, $\alpha=0$ and using \eqref{int |mathcal K_I,q^P|}, we can generalize \eqref{LHJB in R^n_prior C^1 estimate} to obtain the higher order regularity estimates as follows.
\begin{Lem}
Let $k\in\mathbb{Z}_+$. Assume $b(t,x)$, $c(t,x)$, $f(t,x)$ are bounded and continuous on $[0,T]\times\mathbb{R}^{n}$, $b\in B\left((0,T);C_{\mathcal{X},loc}^{\alpha}\left(\mathbb{R}^{n};\mathbb{R}^{n_1}\right)\right)\cap B\left((0,T);C_{\mathcal{X}}^{k-1}\left(\mathbb{R}^{n};\mathbb{R}^{n_1}\right)\right)$, and $c$, $f\in B\left((0,T);C_{\mathcal{X},loc}^{\alpha}\left(\mathbb{R}^{n}\right)\right)\cap B\left((0,T);C_{\mathcal{X}}^{k-1}\left(\mathbb{R}^{n}\right)\right)$. Suppose $z(t,x)\in C_{{\mathcal{X}}}^{1,2} \left((0,T]\times\mathbb{R}^{n}\right)\cap C\left([0,T]\times\mathbb{R}^{n}\right)$ is a solution to the Cauchy problem \eqref{LHJB in R^n}, satisfying $z\in B\left((0,T);C_{\mathcal{X},loc}^{1+\alpha}\left(\mathbb{R}^{n}\right)\right)\cap B\left([0,T];C_{\mathcal{X}}^{k}\left(\mathbb{R}^{n}\right)\right)$. Then $z$ satisfies
\begin{equation}\label{LHJB in R^n_prior C^k estimate}
\sup_{t \in[0, T]}\|z(t, \cdot)\|_{C_{\mathcal{X}}^{k}\left(\mathbb{R}^{n}\right)} \leq C\sup_{t \in(0, T)}\|f(t, \cdot)\|_{C_{\mathcal{X}}^{k-1}\left(\mathbb{R}^{n}\right)},
\end{equation}
where the constant $C>0$ depends on $\mathbb{G}$, $k$, $T$, $\sup_{t\in(0,T)}\|b(t,\cdot)\|_{C_{\mathcal{X}}^{k-1}\left(\mathbb{R}^{n};\mathbb{R}^{n_1}\right)}$ and $\sup_{t\in(0,T)}\|c(t, \cdot)\|_{C_{\mathcal{X}}^{k-1}\left(\mathbb{R}^{n}\right)}$ only.
\end{Lem}

\section{Construction of several types of Mollifiers}\label{Sec_3}

In this section, we study the several types of mollifiers constructed in Propositions \ref{Prop_Mollifiers_Holder}-\ref{Prop_Mollifiers_dual space}, and prove the validity of these propositions. Furthermore, we derive the corollary of Proposition \ref{Prop_Mollifiers_dual space}, namely Corollary \ref{Cor_Mollifiers_P(T_G)}.

Recalling \cite{02BLU}, the fundamental solution $\Gamma_0\left(t-s,y^{-1}\circ x\right)$ for the heat operator $\mathcal{H}=\partial_t-\Delta_{\mathcal{X}}$ is smooth on $\left\{(t,x;s,y)\in (\mathbb{R}\times \mathbb{R}^n)\times (\mathbb{R}\times \mathbb{R}^n)| t>s\right\}$, $\int_{\mathbb{R}^n}\Gamma_0\left(t,x\right)dx=1$ for any $t>0$, and for any multi-index $I=\left(i_1, i_2, \cdots, i_{|I|}\right),i_j \in\{1,\ldots,n_1\}$ with the length $|I|\in\mathbb{N}$, the Lie derivative $X_I\Gamma_0$ satisfies the following Gaussian estimate:
\begin{equation}\label{Gamma_0 Lie derivatives GE}
|X_I\Gamma_0\left(t,x\right)| \leq c_{|I|}t^{-\frac{Q+|I|}{2}}\exp\left(-\frac{\left\|x\right\|^2}{ct}\right),\,(t,x)\in\mathbb{R}_+\times\mathbb{R}^n,
\end{equation}
where $c>1$ depends only on $\mathbb{G}$ and $c_{|I|}>0$ depends only on $|I|$ and $\mathbb{G}$. We now prove the rationality of the mollifiers adapted to the family of H\"{o}rmander vector fields $\mathcal{X}$.
\begin{proof}[\bf{Proof of Proposition \ref{Prop_Mollifiers_Holder}}]
It is easy to find that $\phi_{\varepsilon}(t,x) \in C^{\infty}\left(\mathbb{R}\times\mathbb{R}^{n}\right)$, and using \eqref{Gamma_0 Lie derivatives GE} yields that for any multi-index $I=\left(i_1, i_2, \cdots, i_p\right),i_j \in\{1,\ldots,n_1\}$ with the length $|I|=p\in\mathbb{N}$ and $q\in\mathbb{N}$,
\begin{equation}\label{phi_varepsilon derivatives estimates}
|X_I\partial_t^q\phi_{\varepsilon}\left(t,x\right)| \leq C \varepsilon^{-\frac{Q+|I|+2q+2}{2}}\mathbf{1}_{[-\varepsilon,\varepsilon]}(t) \exp\left(-\frac{\left\|x\right\|^2}{c\varepsilon}\right),
\end{equation}
where the constant $C>0$ depends on $|I|$, $q$ and $\mathbb{G}$ only.

By referring to \cite[P. 4]{22BB}, it is known that for any $X_i\in\mathcal{X}$, the Lie derivative of any function $g:\mathbb{R}^{n}\to\mathbb{R}$ can be expressed as
\begin{equation*}
X_{i}g(x)=\frac{d}{d\tau}\bigg|_{\tau=0}g(\exp(\tau X_{i})(x))=\lim_{\tau\to 0}\frac{g(\exp(\tau X_{i})(x))-g(x)}{\tau}.
\end{equation*}
Here, $\exp(\tau X_{i})(x)$ denotes the exponential map of the vector field $X_{i}$, which is the solution to the ordinary differential equation $\gamma'(\tau)=X_{i}(\gamma(\tau))$ with initial condition $\gamma(0)=x$.

Applying the mean value theorem and noting that $X_ig(y^{-1}\circ \cdot)(x)=X_ig(y^{-1}\circ x)$, we obtain that for any $\tau\in \mathbb{R}$ and $(t,x)\in\mathbb{R}\times\mathbb{R}^{n}$,
\begin{align}\label{f_varepsilon differential}
& \frac{1}{\tau}\left(f_{\varepsilon}(t,\exp(\tau X_{i})(x))-f_{\varepsilon}(t,x)\right)\notag \\
= & \frac{1}{\tau}\int_{\mathbb{R}\times\mathbb{R}^{n}}\left(\phi_{\varepsilon}(t-s,y^{-1}\circ \exp(\tau X_{i})(x))-\phi_{\varepsilon}(t-s,y^{-1}\circ x)\right)f(s,y)d y ds\notag \\
= & \int_{\mathbb{R}\times\mathbb{R}^{n}} X_i\phi_{\varepsilon}(t-s,y^{-1}\circ \exp(\theta_{\tau} X_{i})(x))f(s,y)d y ds,
\end{align}
where $\theta_{\tau}$ is a value between $0$ and $\tau$. To consider the case when $\tau\to 0$, we simply take $|\tau|\leq\delta$ and choose $\delta>0$ sufficiently small. From the continuity of the mapping $\tau\mapsto \|\exp(\tau X_i)(x)\|$, there exists a constant $L=L(x,X_i,\delta)>0$ such that for any $|\tau|\leq \delta$, we have $\|\exp(\tau X_i)(x)\|\leq L$. Therefore, by the triangle inequality,
\begin{equation*}
\|y^{-1} \circ \exp(\theta_\tau X_i)(x)\|^2\geq\frac{1}{2}\|y^{-1}\|^2-\|\exp(\theta_\tau X_i)(x)\|^2\geq \frac{1}{2}\|y\|^2-L^2.
\end{equation*}
Then, it follows from \eqref{phi_varepsilon derivatives estimates} and \eqref{f_growth cond.} that for any $t,s\in\mathbb{R}$ and $x,y\in\mathbb{R}^n$,
\begin{align*}
& \left|X_i\phi_{\varepsilon}(t-s,y^{-1}\circ \exp(\theta_{\tau} X_{i})(x))f(s,y)\right| \\
\leq & C \varepsilon^{-\frac{Q+3}{2}}\mathbf{1}_{[-\varepsilon,\varepsilon]}(t-s)\exp\left(-\frac{\left\|y^{-1}\circ \exp(\theta_{\tau} X_{i})(x)\right\|^2}{c\varepsilon}\right)\exp\left(\mu \left\|y\right\|^2\right) \\
\leq & C \varepsilon^{-\frac{Q+3}{2}}\mathbf{1}_{[-\varepsilon,\varepsilon]}(t-s)\exp\left(\frac{L^2}{c\varepsilon}\right)\exp\left(-\left(\frac{1}{2c\varepsilon}-\mu\right) \left\|y\right\|^2\right),
\end{align*}
where the constant $C>0$ depends on $M$ and $\mathbb{G}$ only. Since
\begin{align*}
& \int_{\mathbb{R}\times\mathbb{R}^{n}}\varepsilon^{-\frac{Q+3}{2}}\mathbf{1}_{[-\varepsilon,\varepsilon]}(t-s)\exp\left(\frac{L^2}{c\varepsilon}\right)\exp\left(-\left(\frac{1}{2c\varepsilon}-\mu\right) \left\|y\right\|^2\right)dy ds \\
\leq & 2\varepsilon^{-\frac{1}{2}}\left(\frac{1}{2c}-\mu\varepsilon\right)^{-\frac{Q}{2}} \exp\left(\frac{L^2}{c\varepsilon}\right)<+\infty,
\end{align*}
letting $\tau\to 0$ in \eqref{f_varepsilon differential} and using the dominated convergence theorem, we get that $X_if_{\varepsilon}(t,x)$ exists, and
\begin{equation*}
X_if_{\varepsilon}(t,x)=\int_{\mathbb{R}\times\mathbb{R}^{n}} X_i\phi_{\varepsilon}(t-s,y^{-1}\circ x)f(s,y)d y ds,\,(t,x)\in\mathbb{R}\times\mathbb{R}^n.
\end{equation*}
Based on \eqref{phi_varepsilon derivatives estimates}, repeating a similar way can yield that $X_I\partial_t^q f_{\varepsilon}(t,x)$ exists, and
\begin{equation*}
X_I\partial_t^q f_{\varepsilon}(t,x)=\int_{\mathbb{R}\times\mathbb{R}^{n}} X_I\partial_t^q \phi_{\varepsilon}(t-s,y^{-1}\circ x)f(s,y)d y ds,\,(t,x)\in\mathbb{R}\times\mathbb{R}^n
\end{equation*}
for each multi-index $I=\left(i_1, i_2, \cdots, i_p\right),i_j \in\{1,\ldots,n_1\}$ with the length $|I|=p\in\mathbb{N}$ and $q\in\mathbb{N}$.

According to the Sobolev's embedding theorem (see \cite[Proposition 2.7]{22BB}), we have
\begin{equation*}
\bigcap_{k=1}^{\infty}W_{\mathcal{X}}^{k,p}(\mathbb{G})\subset C^{\infty}(\mathbb{G}),\,1\leq p\leq\infty,
\end{equation*}
where $W_{\mathcal{X}}^{k,p}$ denotes the Sobolev space defined by the family of H\"{o}rmander's vector fields $\mathcal{X}$. Therefore, we have proven that $f_{\varepsilon}(t,x) \in C^{\infty}\left(\mathbb{R}\times\mathbb{R}^{n}\right)$.

To prove conclusion \ref{mollifiers_f_varepsilon->f}, we notice that
$$
f_{\varepsilon}(t,x)=\int_{|h| \leq 1}\int_{\mathbb{R}^{n}}\Gamma_0\left(\varepsilon,y^{-1}\circ x\right)\varphi(h)f(t-\varepsilon h,y)dy dh
$$
and
$$
\int_{|h| \leq 1}\varphi(h)\int_{\mathbb{R}^{n}}\Gamma_0\left(\varepsilon,y^{-1}\circ x\right)dydh=\int_{|h| \leq 1}\varphi(h)\int_{\mathbb{R}^{n}}\Gamma_0\left(\varepsilon,z\right)dzdh=1.
$$
Then we have that for any $(t,x) \in \mathbb{R}\times\mathbb{R}^{n}$,
$$
\begin{aligned}
\left|f_{\varepsilon}(t,x)-f(t,x)\right|= & \left|\int_{-1}^{1}\int_{\mathbb{R}^{n}}\Gamma_0\left(\varepsilon,y^{-1}\circ x\right)\varphi(h) \left(f(t-\varepsilon h,y)-f(t,x)\right)dy dh\right| \\
\leq & \int_{-1}^{1}\int_{\mathbb{R}^{n}}\Gamma_0\left(\varepsilon,y^{-1}\circ x\right)\varphi(h) \left|f(t-\varepsilon h,y)-f(t,y)\right|dy dh \\
& + \int_{-1}^{1}\int_{\mathbb{R}^{n}}\Gamma_0\left(\varepsilon,y^{-1}\circ x\right)\varphi(h) \left|f(t,y)-f(t,x)\right|dy dh \\
=: & I_1+I_2.
\end{aligned}
$$
According to \eqref{f_ctn. cond.}, we have $I_1 \leq \omega_f(\varepsilon)$, where $\omega_f$ is the modulus of continuity of $f(\cdot,x)$ for any $x\in\mathbb{R}^n$.

As for $I_2$, we refer to the method in \cite[Theorem 11.2]{07BB}. Specifically, using \eqref{Gamma_0 Lie derivatives GE} and \eqref{f_ctn. cond.}, since the Lebesgue measure of the $d_{cc}$-ball satisfies (see \cite[(5.41)]{07BLU})
\begin{equation*}
\left|B_{r}\left(x\right)\right|=r^{Q}\left|B_{1}\left(x\right)\right|,\,r>0,
\end{equation*}
where $Q$ is the homogeneous dimension of $\mathbb{G}$ (see Definition \ref{Def_Carnot group}), then we have
\begin{align*}
I_2 \leq & \frac{C}{\left|B_{\sqrt{\varepsilon}}\left(x\right)\right|} \int_{\mathbb{R}^{n}}\exp\left(-\frac{d_{cc}(x,y)^2}{c\varepsilon}\right) d_{cc}(x,y)^{\alpha}dy \int_{-1}^{1}\varphi(h)dh \\
= & \frac{C}{\left|B_{\sqrt{\varepsilon}}\left(x\right)\right|} \int_{B_{\sqrt{\varepsilon}}\left(x\right)}\exp\left(-\frac{d_{cc}(x,y)^2}{c\varepsilon}\right) d_{cc}(x,y)^{\alpha}dy \\
& + \sum_{k=0}^{+\infty}\frac{C}{\left|B_{\sqrt{\varepsilon}}\left(x\right)\right|} \int_{B_{2^{k+1}\sqrt{\varepsilon}}\left(x\right)\backslash
B_{2^{k}\sqrt{\varepsilon}}\left(x\right)} \exp\left(-\frac{d_{cc}(x,y)^2}{c\varepsilon}\right) d_{cc}(x,y)^{\alpha}dy \\
\leq & C\left(\varepsilon^{\frac{\alpha}{2}} +\sum_{k=0}^{+\infty}\frac{\left|B_{2^{k+1}\sqrt{\varepsilon}}\left(x\right)\right|}{\left|B_{\sqrt{\varepsilon}}\left(x\right)\right|} \exp\left(-\frac{4^k}{c}\right)\left(2^{k+1}\sqrt{\varepsilon}\right)^{\alpha}\right) \\
= & C\left(\varepsilon^{\frac{\alpha}{2}} +\sum_{k=0}^{+\infty}2^{Q(k+1)} \exp\left(-\frac{4^k}{c}\right)\left(2^{k+1}\sqrt{\varepsilon}\right)^{\alpha}\right) \\
\leq & C \varepsilon^{\frac{\alpha}{2}},
\end{align*}
where the constant $C>0$ depends on $\mathbb{G}$ and $\alpha$ only.

Combining the above estimates of $I_1$ and $I_2$, we finally get $\lim\limits_{\varepsilon\to0}\left\|f_{\varepsilon}-f\right\|_{L^{\infty}\left(\mathbb{R}\times\mathbb{R}^{n}\right)}=0$.

To prove conclusion \ref{mollifiers_C^alpha/2,alpha}, we refer to the proof of \cite[Theorem 1.2(2)]{25JWY} to get that for any $k\in\mathbb{N}$ and multi-index $J=(j_1,j_2,\ldots,j_k),j_l\in\{1,\ldots,n_1\}$, there exist kernels $\phi_{\varepsilon}^P(t,x;s,y)$, $P=(p_1,p_2,\ldots,p_k),p_i\in\{1,\ldots,n_1\}$ having the form
\begin{equation*}
\phi_{\varepsilon}^{P}(t,x;s,y)=
\begin{cases}
\sum_{i=1}^{M}\mathcal{R}_{i}^{J,P}(y^{-1}\circ x)X_{l_1^{J,P}}\cdots X_{l_i^{J,P}} \phi_{\varepsilon}(t-s,y^{-1}\circ x), & \mbox{if } \dim(P)\geq1, \\
\phi_{\varepsilon}(t-s,y^{-1}\circ x), & \mbox{if } \dim(P)=0
\end{cases}
\end{equation*}
for a certain family of finite number of $i$-homogeneous polynomials $\mathcal{R}_i^{J,P},i\in\{1,\ldots,M\}$ and left-invariant vector fields $\{X_{l_i^{J,P}}\}_{i=1}^{M},l_i^{J,P}\in\{1,\ldots,n_1\}$ with $M=M(J,P)\in\mathbb{Z}_+$, such that
\begin{equation}\label{X_J int_0^t phi_varepsilon f}
X_J \int_{\mathbb{R}\times\mathbb{R}^{n}} \phi_{\varepsilon}(t-s,y^{-1}\circ x)f(s,y)dyds =\sum_{p_1,\ldots,p_k=1}^{n_1}\int_{\mathbb{R}\times\mathbb{R}^{n}} \phi_{\varepsilon}^{P}(t,x;s,y)X_Pf(s,y)dyds.
\end{equation}
It is easy to find that \eqref{phi_varepsilon derivatives estimates} still holds for $\phi_{\varepsilon}^{P}$, and can be further written as
\begin{align*}
|X_I\partial_t^q\phi_{\varepsilon}^{P}(t,\cdot;s,y)(x)| \leq & C\mathbf{1}_{[-\varepsilon,\varepsilon]}(t-s) \frac{\left(\frac{\left\|y^{-1}\circ x\right\|^2+\varepsilon}{\varepsilon}\right)^{\frac{Q+|I|+2q+2}{2}} \exp\left(-\frac{\left\|y^{-1}\circ x\right\|^2}{c\varepsilon}\right)}{\left(\left\|y^{-1}\circ x\right\|^2+t-s\right)^{\frac{Q+|I|+2q+2}{2}}} \\
\leq & \frac{C}{\left(\left\|y^{-1}\circ x\right\|^2+t-s\right)^{\frac{Q+|I|+2q+2}{2}}},
\end{align*}
where the constant $C>0$ depends on $|I|$, $q$, $J$, $P$ and $\mathbb{G}$ only. Then, using the same method as in Lemma \ref{Lem_mathcal K_I,q,l_singular integral theory}, we can verify that $\phi_{\varepsilon}^{P}(t,x;s,y)$ satisfies the growth condition:
\begin{equation*}
|\phi_{\varepsilon}^{P}(t,x;s,y)| \leq \frac{c}{d_p((t,x),(s,y))^{Q+2}} \leq \frac{c}{|B^p((t,x);(s,y))|};
\end{equation*}
the mean value inequality:
\begin{align*}
|\phi_{\varepsilon}^{P}(t,x;s,y)-\phi_{\varepsilon}^{P}(t_1,x_1;s,y)| & \leq c \frac{d_p((t_1,x_1),(t,x))}{d_p((t_1,x_1),(s,y))^{Q+3}}\\
& \leq \frac{c}{|B^p((t_1,x_1);(s,y))|} \cdot\left(\frac{d_p((t_1,x_1),(t,x))}{d_p((t_1,x_1),(s,y))}\right)
\end{align*}
when $d_p((t_1,x_1),(s,y)) > 4 d_p((t_1,x_1),(t,x))$, for any $(t,x),(s,y)\in\mathbb{R}\times\mathbb{R}^n$ and some constant $c>0$ depending only on $\mathbb{G}$. In addition, applying integration by parts, we have
\begin{equation*}
\int_{\mathbb{R}\times\mathbb{R}^{n}}\phi_{\varepsilon}^{P}(t,x;s,y)d y ds=c_{J,P}\int_{\mathbb{R}\times\mathbb{R}^{n}}\phi_{\varepsilon}(t-s,u)d u ds=c_{J,P}
\end{equation*}
for some real constant $c_{J,P}$ depending only on $J$ and $P$. Thus $\phi_{\varepsilon}^P$ satisfies the cancellation properties in Lemma \ref{Lem_mathcal K_I,q,l_singular integral theory}.

Therefore, by \eqref{X_J int_0^t phi_varepsilon f}, we can obtain
\begin{align*}
\left\|f_\varepsilon\right\|_{C_{\mathcal{X}}^{0,k}\left(\mathbb{R}\times\mathbb{R}^{n}\right)} \leq & C \left\|f\right\|_{C_{\mathcal{X}}^{0,k}\left(\mathbb{R}\times\mathbb{R}^{n}\right)}\int_{\mathbb{R}\times\mathbb{R}^{n}}\varepsilon^{-\frac{Q+2}{2}}\mathbf{1}_{[-\varepsilon,\varepsilon]}(t-s) \exp\left(-\frac{\left\|y^{-1}\circ x\right\|^2}{c\varepsilon}\right)dyds \\
= & C \left\|f\right\|_{C_{\mathcal{X}}^{0,k}\left(\mathbb{R}\times\mathbb{R}^{n}\right)},
\end{align*}
and it follows from Lemma \ref{Lem_HCSIO} that
\begin{equation*}
\left[f_\varepsilon\right]_{C_{\mathcal{X}}^{\frac{\alpha}{2},k+\alpha}\left(\mathbb{R}\times\mathbb{R}^{n}\right)} \leq C \left\|f\right\|_{C_{\mathcal{X}}^{\frac{\alpha}{2},k+\alpha}\left(\mathbb{R}\times\mathbb{R}^{n}\right)},
\end{equation*}
where the constants $C>0$ depend on $\mathbb{G}$ and $k$ only. Thus, conclusion \ref{mollifiers_C^alpha/2,alpha} holds true.

Let us now prove conclusion \ref{mollifiers_C^k+alpha}. Set the function
\begin{equation*}
F(t,y):=\int_{\mathbb{R}}\frac{1}{\varepsilon}\varphi\left(\frac{t-s}{\varepsilon}\right)f(s,y)ds, \,t\in\mathbb{R},\,y\in\mathbb{R}^n.
\end{equation*}
As before, by using the mean value theorem and the dominated convergence theorem, we simply obtain that for any multi-index $J=\left(j_1, j_2, \cdots, j_k\right),j_i \in\{1,\ldots,n_1\}$ with the length $|J|=k\in\mathbb{N}$, the Lie derivative $X_J F(t,y)$ exists, and
\begin{equation*}
X_J F(t,y)=\int_{\mathbb{R}}\frac{1}{\varepsilon}\varphi\left(\frac{t-s}{\varepsilon}\right)X_J f(s,y)ds,\,t\in\mathbb{R},\,y\in\mathbb{R}^n.
\end{equation*}
Furthermore, we apply \cite[Theorem 1.2(1)]{25JWY} to obtain
\begin{equation*}
X_J f_{\varepsilon}(t,x) =\sum_{p_1,\ldots,p_k=1}^{n_1}\int_{\mathbb{R}^n}\mathcal{K}^P\left(\varepsilon,y^{-1}\circ x\right)X_P F(t,y)dy,\,(t,x)\in\mathbb{R}\times\mathbb{R}^n,
\end{equation*}
where $\mathcal{K}^P(\varepsilon,y^{-1}\circ x),P=(p_1,p_2,\ldots,p_k),p_i\in\{1,\ldots,n_1\}$ are kernels in the form of \eqref{mathcal K_I,q^P} with $|I|=q=0$. And $\mathcal{K}^P$ satisfies the conditions in Lemma \ref{Lem_HCSIO} with respect to the distance $d_{cc}$ with the constants $c$ depending only on $\mathbb{G}$, $k$, $J$ and $P$.

In fact, by referring to the proof of \cite[Theorem 1.2(1)]{25JWY}, we can see that
\begin{equation*}
\left|\mathcal{K}^{P}\left(\varepsilon,y^{-1}\circ x\right)\right| \leq c_{J,P}\varepsilon^{-\frac{Q}{2}}\exp\left(-\frac{\left\|y^{-1}\circ x\right\|^2}{c\varepsilon}\right),\,x,y\in\mathbb{R}^n.
\end{equation*}
Hence,
\begin{equation*}
\left\|X_J f_{\varepsilon}\right\|_{L^{\infty}\left(\mathbb{R}\times\mathbb{R}^{n}\right)} \leq \sum_{p_1,\ldots,p_k=1}^{n_1}c_{J,P} \left\|X_P F\right\|_{L^{\infty}\left(\mathbb{R}\times\mathbb{R}^{n}\right)} \leq C \sup_{t\in\mathbb{R}}\left\|f(t,\cdot)\right\|_{C_{\mathcal{X}}^{k}\left(\mathbb{R}^{n}\right)}.
\end{equation*}
where the constant $C>0$ depends on $\mathbb{G}$ and $k$ only. On the other hand, by Lemma \ref{Lem_HCSIO}, we obtain
\begin{equation*}
\sup_{t\in\mathbb{R}}\left[f_{\varepsilon}(t,\cdot)\right]_{C_{\mathcal{X}}^{k+\alpha}\left(\mathbb{R}^{n}\right)} \leq C\sup_{t\in\mathbb{R}}\left\|F(t,\cdot)\right\|_{C_{\mathcal{X}}^{k+\alpha}\left(\mathbb{R}^{n}\right)} \leq C\sup_{t\in\mathbb{R}}\left\|f(t,\cdot)\right\|_{C_{\mathcal{X}}^{k+\alpha}\left(\mathbb{R}^{n}\right)},
\end{equation*}
where the constant $C>0$ depends on $\mathbb{G}$ and $k$ only. Hence, conclusion \ref{mollifiers_C^k+alpha} can be derived.

Finally, conclusion \ref{mollifiers_periodicity} follows directly from the fact that for any $k\in\mathbb{Z}^n$,
\begin{align*}
f_{\varepsilon}(t,k\circ x)= & \int_{\mathbb{R}\times\mathbb{R}^{n}}\phi_{\varepsilon}(t-s,y^{-1}\circ k\circ x)f(s,y)d y ds \\
= & \int_{\mathbb{R}\times\mathbb{R}^{n}}\phi_{\varepsilon}(t-s,z^{-1}\circ x)f(s,k\circ z)d z ds \\
= & f_{\varepsilon}(t,x),\,(t,x)\in\mathbb{R}\times\mathbb{R}^n.
\end{align*}

Here completes the proof.
\end{proof}
The following proof is for the mollifiers adapted to Carnot tori.
\begin{proof}[\bf{Proof of Proposition \ref{Prop_Mollifiers_Lip.}}]
Since $\left\|\cdot\right\|_{\mathbb{G}}$ is smooth out of the origin and all homogeneous norms on $\mathbb{G}$ are equivalent, we have that $\psi_{\varepsilon}$ is a smooth function with support in $B_{\varepsilon}(0)$. Then, it can be found that
\begin{align*}
\int_{[0,1)^n}\sum_{k\in\mathbb{Z}^n}\psi_{\varepsilon}\left(k\circ x\circ y^{-1}\right)g(y)d y= &\int_{[0,1)^n}\sum_{k\in\mathbb{Z}^n:\left\|k\circ x\circ y^{-1}\right\|_{\mathbb{G}}\leq \varepsilon}\psi_{\varepsilon}\left(k\circ x\circ y^{-1}\right)g(y)d y \\
= &\int_{x\circ\left([0,1)^n\right)^{-1}}\sum_{k\in\mathbb{Z}^n:\left\|k\circ z\right\|_{\mathbb{G}}\leq \varepsilon}\psi_{\varepsilon}(k\circ z)g(z^{-1}\circ x)d z,
\end{align*}
and there exist a finite number of $k$'s in $\mathbb{Z}^n$, depending on $\varepsilon$, such that $\left\|k\circ z\right\|_{\mathbb{G}}\leq \varepsilon$ for any $z\in\mathbb{R}^n$ with $\|z\|_{\mathbb{G}} \leq \|x\|_{\mathbb{G}}+r^{\frac{1}{2r!}}\sqrt{n}<+\infty$. Because $\psi_{\varepsilon}\left(k\circ x\circ y^{-1}\right)$ is smooth in $x$, by using standard calculation, it is easy to show that $g_{\varepsilon}(x) \in C^{\infty}\left(\mathbb{T}_{\mathbb{G}}\right)$.

To prove conclusion \ref{g_varepsilon->g}, from Lemma \ref{Lem_T_G=[0,1)^n}, since there exist finite sets in the family $\left\{k\circ [0,1)^n\right\}_{k\in\mathbb{Z}^n}$ that cover $x\circ\left([0,1)^n\right)^{-1}$ and $\left|x\circ\left([0,1)^n\right)^{-1}\right|=\left|[0,1)^n\right|$, then we have
$$
\int_{x\circ\left([0,1)^n\right)^{-1}}\sum_{k\in\mathbb{Z}^n}\psi_{\varepsilon}(k\circ z)dz=\int_{[0,1)^n}\sum_{k\in\mathbb{Z}^n}\psi_{\varepsilon}(k\circ k'\circ z')dz'=\int_{\mathbb{R}^n}\psi_{\varepsilon}(x)dx=1.
$$
Hence, for any $x\in\mathbb{T}_{\mathbb{G}}$,
\begin{align*}
\left|g_{\varepsilon}(x)-g(x)\right|= & \left|\int_{x\circ\left([0,1)^n\right)^{-1}}\sum_{k\in\mathbb{Z}^n}\psi_{\varepsilon}(k\circ z)\left(g(z^{-1}\circ x)-g(x)\right)dz\right| \\
\leq & C\int_{D_{\frac{1}{\varepsilon}}\left(x\right)\circ \left(D_{\frac{1}{\varepsilon}}\left([0,1)^n\right)\right)^{-1}}\sum_{k\in\mathbb{Z}^n}\psi(D_{\frac{1}{\varepsilon}}(k)\circ z)\left|g(\left(D_{\varepsilon}(z)\right)^{-1} \circ x)-g(x)\right|d z \to 0
\end{align*}
as $\varepsilon\to0$, where we notice that $\left\{D_{\frac{1}{\varepsilon}}(k)\circ D_{\frac{1}{\varepsilon}}\left([0,1)^n\right)\right\}$ is still a tiling of $\mathbb{G}$. This leads to conclusion \ref{g_varepsilon->g}.

Next we prove conclusion \ref{mollifiers_C^0,1}. For any $x_1,x_2 \in \mathbb{T}_{\mathbb{G}}$,
\begin{align*}
& \left|g_{\varepsilon}(x_1)-g_{\varepsilon}(x_2)\right| \\
= & \left|\int_{x\circ\left([0,1)^n\right)^{-1}}\sum_{k\in\mathbb{Z}^n}\psi_{\varepsilon}(k\circ z)\left(g(z^{-1}\circ x_1)-g(z^{-1}\circ x_2)\right)dz\right| \\
\leq & [g]_{C_{\mathcal{X}}^{0+1}\left(\mathbb{T}_{\mathbb{G}}\right)} d_{cc}^{\mathbb{T}_{\mathbb{G}}}(x_1, x_2).
\end{align*}
Thus we obtain conclusion \ref{mollifiers_C^0,1}.
\end{proof}
Lastly, we present the proof for the mollifiers adapted to the dual spaces of non-isotropic H\"{o}lder spaces.
\begin{proof}[\bf{Proof of Proposition \ref{Prop_Mollifiers_dual space}}]
We only need to prove conclusion \ref{mollifiers_L^1(C^-k+alpha)}, since it implies conclusion \ref{mollifiers_C^-k+alpha}. Note that for any fixed $x$, there exists a finite number of $k$'s in $\mathbb{Z}^n$, depending on $\varepsilon$, such that $\left\|k\circ x \circ y^{-1}\right\|_{\mathbb{G}}\leq \varepsilon$ for any $y\in[0,1)^n$, then $\sum_{k\in\mathbb{Z}^n}\psi_{\varepsilon}\left(k\circ x \circ (\cdot)^{-1}\right) \in C^\infty\left([0,1)^n\right)$, hence $\left\langle\mu(s),\sum_{k\in\mathbb{Z}^n}\psi_{\varepsilon}\left(k\circ x \circ (\cdot)^{-1}\right)\right\rangle$ makes sense.

It can be proved that $\mu_{\varepsilon}(t,x) \in C^\infty\left(\mathbb{R}\times\mathbb{T}_{\mathbb{G}}\right)$. Indeed, firstly, since $k\circ x \circ y^{-1}$ is smooth in $(x,y)$ for any $k\in\mathbb{Z}^n$, then
\begin{equation*}
\left\|\sum_{k\in\mathbb{Z}^n}\psi_{\varepsilon}\left(k\circ x \circ (\cdot)^{-1}\right) -\sum_{k\in\mathbb{Z}^n}\psi_{\varepsilon}\left(k\circ x_0 \circ (\cdot)^{-1}\right)\right\|_{C^\infty\left([0,1)^n\right)}\to0
\end{equation*}
as $x \to x_0\in \mathbb{T}_{\mathbb{G}}$. This leads to the continuity of $\left\langle\mu(s),\sum_{k\in\mathbb{Z}^n}\psi_{\varepsilon}\left(k\circ x \circ (\cdot)^{-1}\right)\right\rangle$ on $\mathbb{T}_{\mathbb{G}}$ for any $s\in\mathbb{R}$. Hence, it is easy to obtain that
$$
\mu_{\varepsilon}(t,x)=\int_{\mathbb{R}}\left\langle\mu(s),\sum_{k\in\mathbb{Z}^n}\psi_{\varepsilon}\left(k\circ x \circ (\cdot)^{-1}\right)\right\rangle\varphi_{\varepsilon}(t-s)ds\in C(\mathbb{R}\times\mathbb{T}_{\mathbb{G}}).
$$
Nextly, we denote $e_l$ as the unit vector in the $x_l$'s direction and construct
\begin{align*}
& \frac{\partial_t^q\mu_{\varepsilon}(t,x+he_l)-\partial_t^q\mu_{\varepsilon}(t,x)}{h} \\
= & \int_{\mathbb{R}}\frac{\left\langle\mu(s),\sum\limits_{k\in\mathbb{Z}^n}\psi_{\varepsilon}\left(k\circ \left(x+he_l\right)\circ (\cdot)^{-1}\right)\right\rangle-\left\langle\mu(s),\sum\limits_{k\in\mathbb{Z}^n}\psi_{\varepsilon}\left(k\circ x \circ (\cdot)^{-1}\right)\right\rangle}{h} \partial_t^q\varphi_{\varepsilon}(t-s)ds \\
= & \int_{\mathbb{R}}\left\langle\mu(s),\frac{\sum\limits_{k\in\mathbb{Z}^n}\psi_{\varepsilon}\left(k\circ \left(x+he_l\right)\circ (\cdot)^{-1}\right)-\sum\limits_{k\in\mathbb{Z}^n}\psi_{\varepsilon}\left(k\circ x \circ (\cdot)^{-1}\right)}{h}\right\rangle \partial_t^q\varphi_{\varepsilon}(t-s)ds
\end{align*}
for any $(t,x)\in\mathbb{R}\times\mathbb{T}_{\mathbb{G}}$ and $q\in\mathbb{N}$.
Letting $h \to 0$, we can get that the partial derivative of $\partial_t^q\mu_{\varepsilon}$ exists and
$$
\partial_{x_l}\partial_t^q\mu_{\varepsilon}(t,x)=\int_{\mathbb{R}}\left\langle\mu(s),\sum_{k\in\mathbb{Z}^n}\partial_{x_l}\psi_{\varepsilon}\left(k\circ x \circ (\cdot)^{-1}\right)\right\rangle \partial_t^q\varphi_{\varepsilon}(t-s)ds \in C\left(\mathbb{R}\times\mathbb{T}_{\mathbb{G}}\right).
$$
Repeating the above steps, we finally prove that $\mu_{\varepsilon}(t,x) \in C^\infty\left(\mathbb{R}\times\mathbb{T}_{\mathbb{G}}\right)$.

Moreover, choose any $g \in C_{\mathcal{X}}^{k}\left(\mathbb{T}_{\mathbb{G}}\right)$ and set
\begin{align*}
g_{\varepsilon}(y):= & \int_{\mathbb{T}_{\mathbb{G}}}\sum_{k\in\mathbb{Z}^n}\psi_{\varepsilon}\left(k\circ x \circ y^{-1}\right)g(x)dx \\
= & \int_{\mathbb{T}_{\mathbb{G}}}\sum_{k\in\mathbb{Z}^n}\psi_{\varepsilon}\left(k\circ x \circ y^{-1}\right)g(k\circ x)dx \\
= & \int_{\mathbb{R}^n}\psi_{\varepsilon}\left(x \circ y^{-1}\right)g(x)dx,\,y\in\mathbb{R}^n.
\end{align*}
Since $\mu_{\varepsilon}(t,x) \in C^\infty\left(\mathbb{R}\times \mathbb{T}_{\mathbb{G}}\right)$, we can find that
\begin{align*}
\left\langle\mu_{\varepsilon}(t),g(\cdot)\right\rangle= &
\int_{\mathbb{T}_{\mathbb{G}}}\left(\int_{\mathbb{R}}\left\langle\mu(s),\sum_{k\in\mathbb{Z}^n}\psi_{\varepsilon}\left(k\circ x \circ (\cdot)^{-1}\right)\right\rangle \varphi_{\varepsilon}(t-s)ds\right) g(x)dx \\
= & \int_{\mathbb{R}}\left\langle\mu(s),g_{\varepsilon}(\cdot)\right\rangle \varphi_{\varepsilon}(t-s)ds,\,t\in[0,T],
\end{align*}
where the last equality can be obtained by the standard method, which relies on the linearity of $\mu(s)$ for any $s\in\mathbb{R}$.

We can prove $\|g_{\varepsilon}\|_{C_{\mathcal{X}}^{k}\left(\mathbb{R}^n\right)} \leq \|g\|_{C_{\mathcal{X}}^{k}\left(\mathbb{T}_{\mathbb{G}}\right)}$. Indeed, for any multi-index $I=(i_1,i_2,\ldots,i_k)$, $i_j\in\{1,\ldots,n_1\}$ with the length $|I|=k\in\mathbb{N}$ and $y\in\mathbb{R}^n$, we have
\begin{align*}
\left|X_I g_{\varepsilon}(y)\right| = & \left|\int_{\mathbb{R}^n} \psi_{\varepsilon}(z) X_I g\left(z \circ y\right)dz\right| \\
= & \left|\int_{\{\|z\|_{\mathbb{G}}\leq 1\}} \psi(z) X_I g\left(D_{\varepsilon}(z) \circ y\right)dz\right| \leq \sup_{x\in\mathbb{R}^n}\left|X_I g\left(x\right)\right|.
\end{align*}
Therefore, for any $g \in C_{\mathcal{X}}^{k}\left(\mathbb{T}_{\mathbb{G}}\right)$, we obtain that
\begin{align}\label{<mu_varepsilon,g>}
\int_{0}^{T}\left\langle\mu_{\varepsilon}(t),g(\cdot)\right\rangle dt= & \int_{0}^{T}\int_{\mathbb{R}}\left\langle\mu(s),g_{\varepsilon}(\cdot)\right\rangle \varphi_{\varepsilon}(t-s)ds dt \\
\leq & C\left\|\mu\right\|_{L^1\left([0,T];C_{\mathcal{X}}^{-k}\left([0,1)^n\right)\right)} \left\|g_{\varepsilon}\right\|_{C_{\mathcal{X}}^{k}\left(\mathbb{R}^n\right)} \notag\\
\leq & C \left\|\mu\right\|_{L^1\left([0,T];C_{\mathcal{X}}^{-k}\left([0,1)^n\right)\right)} \left\|g\right\|_{C_{\mathcal{X}}^{k}\left(\mathbb{T}_{\mathbb{G}}\right)},\notag
\end{align}
where the constant $C>0$ is independent of $\varepsilon$ and $g$. Thus we have
\begin{equation*}
\left\|\mu_{\varepsilon}\right\|_{L^1\left([0,T];C_{\mathcal{X}}^{-k}\left(\mathbb{T}_{\mathbb{G}}\right)\right)} \leq C \left\|\mu\right\|_{L^1\left([0,T];C_{\mathcal{X}}^{-k}\left([0,1)^n\right)\right)}.
\end{equation*}

Furthermore, we have
\begin{align*}
\left|X_I(g_{\varepsilon}-g)(y)\right| = & \left|\int_{\{\|z\|_{\mathbb{G}}\leq 1\}} \psi(z)\left(X_I g\left(D_{\varepsilon}(z) \circ y\right)-X_Ig(y)\right)dz\right|\notag \\
\leq &\sup_{\{\|z\|_{\mathbb{G}}\leq 1\}}\left|X_I g\left(D_{\varepsilon}(z) \circ y\right)-X_Ig(y)\right|.
\end{align*}
Hence, if $g \in C_{\mathcal{X}}^{k+\alpha}\left(\mathbb{T}_{\mathbb{G}}\right)$, there is
\begin{equation*}
\left\|g_{\varepsilon}-g\right\|_{C_{\mathcal{X}}^{k}\left(\mathbb{R}^n\right)} \leq C\left[g\right]_{C_{\mathcal{X}}^{k+\alpha}\left(\mathbb{T}_{\mathbb{G}}\right)} \sup_{\{\|z\|_{\mathbb{G}}\leq 1\}}\left\|y^{-1}\circ D_{\varepsilon}(z)\circ y\right\|^{\alpha},
\end{equation*}
where the constant $C>0$ depends on $k$ and $\mathbb{G}$ only. Then, we obtain that
\begin{align}\label{<mu_varepsilon-mu,g>}
& \int_{0}^{T}\left\langle\left(\mu_{\varepsilon}-\mu\right)(t),g\right\rangle dt \\
= & \int_{0}^{T}\int_{\mathbb{R}}\left\langle\mu(s),\left(g_{\varepsilon}-g\right)(\cdot)\right\rangle \varphi_{\varepsilon}(t-s)ds dt +\int_{0}^{T}\left\langle\int_{\mathbb{R}}\mu(s)\varphi_{\varepsilon}(t-s)ds-\mu(t),g(\cdot)\right\rangle dt \notag\\
\leq & \left\|g_{\varepsilon}-g\right\|_{C_{\mathcal{X}}^{k}\left(\mathbb{R}^n\right)} \left\|\mu\right\|_{L^1\left([0,T];C_{\mathcal{X}}^{-k}\left([0,1)^n\right)\right)} +\left\|g\right\|_{C_{\mathcal{X}}^{k}\left(\mathbb{R}^n\right)} \left\|\tilde{\mu}_{\varepsilon}-\mu\right\|_{L^1\left([0,T];C_{\mathcal{X}}^{-k}\left([0,1)^n\right)\right)} \notag\\
\leq & C(\varepsilon)\left(\left[g\right]_{C_{\mathcal{X}}^{k+\alpha}\left(\mathbb{T}_{\mathbb{G}}\right)} \left\|\mu\right\|_{L^1\left([0,T];C_{\mathcal{X}}^{-k}\left([0,1)^n\right)\right)}+\left\|g\right\|_{C_{\mathcal{X}}^{k}\left(\mathbb{T}_{\mathbb{G}}\right)}\right),\notag
\end{align}
where $\tilde{\mu}_{\varepsilon}(t):=\int_{\mathbb{R}}\mu(s)\varphi_{\varepsilon}(t-s)ds$, and $C(\varepsilon)>0$ is independent of $g$, satisfying $C(\varepsilon) \to 0$ as $\varepsilon\to0$. This yields that $\mu_{\varepsilon}\to\mu$ in $L^1\left([0,T];C_{\mathcal{X}}^{-(k+\alpha)}\left(\mathbb{T}_{\mathbb{G}}\right)\right)$ as $\varepsilon\to0$.

Thus we have completed the proof.
\end{proof}
The same method as above can be applied to prove Corollary \ref{Cor_Mollifiers_P(T_G)}.
\begin{proof}[\bf{Proof of Corollary \ref{Cor_Mollifiers_P(T_G)}}]
It suffices to repeat the proof steps for Proposition \ref{Prop_Mollifiers_dual space}, where taking $g=1_{A}$ for any Borel set $A \subset \mathbb{T}_{\mathbb{G}}$ in \eqref{<mu_varepsilon,g>} can yield $\mu_{\varepsilon}\in\mathcal{P}(\mathbb{T}_{\mathbb{G}})$; taking any $g\in C_\mathcal{X}^{0+1}(\mathbb{T}_{\mathbb{G}})$ with $\left[g\right]_{C_{\mathcal{X}}^{0+1}\left(\mathbb{T}_{\mathbb{G}}\right)}\leq1$ in \eqref{<mu_varepsilon-mu,g>} can yield $d_1\left(\mu_{\varepsilon},\mu\right)\to0$ as $\varepsilon\to0$.
\end{proof}

\section{Results for the linear degenerate parabolic equation}\label{Sec_4}

In this section, we focus on proving the well-posedness and the Schauder estimates of the solution to the linear degenerate parabolic equation \eqref{LHJB}, i.e., Theorem \ref{Thm_LHJB wellposed regularity}. Additionally, we shall demonstrate the H\"{o}lder continuity estimates of the solution, i.e., Theorem \ref{Thm_LHJB Lipschitz}.

First, we would like to provide the following a priori Schauder estimate of the solution to the Cauchy problem \eqref{LHJB in R^n}, which is useful for the proof of Theorem \ref{Thm_LHJB wellposed regularity}.
\begin{Prop}\label{Prop_LHJB in R^n SE}
Let $k\in\{2,3,\ldots\}$ and $\alpha\in(0,1)$. Assume $b(t,x)\in C^{\frac{\alpha}{2},1+\alpha}_{\mathcal{X}}\left([0,T]\times\mathbb{R}^{n};\mathbb{R}^{n_1}\right)\cap C^{\frac{\alpha}{2},k-2+\alpha}_{\mathcal{X}}\left([0,T]\times\mathbb{R}^{n};\mathbb{R}^{n_1}\right)$, $c(t,x)$, $f(t,x)\in C^{\frac{\alpha}{2},k-2+\alpha}_{\mathcal{X}}\left([0,T]\times\mathbb{R}^{n}\right)$, and $g(x) \in C^{k+\alpha}_{\mathcal{X}}\left(\mathbb{R}^{n}\right)$. Suppose $z(t,x)\in C_{{\mathcal{X}}}^{1,2} \left((0,T]\times\mathbb{R}^{n}\right)\cap C\left([0,T]\times\mathbb{R}^{n}\right)$ is a solution to the Cauchy problem \eqref{LHJB in R^n}, satisfying $z\in B\left([0,T];C_{\mathcal{X}}^{1+\alpha}\left(\mathbb{R}^{n}\right)\right)$.

Then $z$ satisfies
\begin{equation}\label{LHJB in R^n SE}
\left\|z\right\|_{C^{\frac{\alpha}{2},k+\alpha}_{\mathcal{X}}\left([0,T]\times\mathbb{R}^{n}\right)} +\left\|z\right\|_{C^{1+\frac{\alpha}{2},\alpha}_{\mathcal{X}}\left([0,T]\times\mathbb{R}^{n}\right)} \leq C\left(\left\|g\right\|_{C_{\mathcal{X}}^{k+\alpha}\left(\mathbb{R}^{n}\right)} +\left\|f\right\|_{C_{\mathcal{X}}^{\frac{\alpha}{2},k-2+\alpha}\left([0,T]\times \mathbb{R}^{n}\right)}\right),
\end{equation}
where $C>0$ depends on $\mathbb{G}$, $\left\|b\right\|_{C_{\mathcal{X}}^{\frac{\alpha}{2},k-2+\alpha}\left([0,T]\times \mathbb{R}^{n};\mathbb{R}^{n_1}\right)}$, $\left\|c\right\|_{C_{\mathcal{X}}^{\frac{\alpha}{2},k-2+\alpha}\left([0,T]\times \mathbb{R}^{n}\right)}$, $T$, $k$, $\alpha$ and $\left\|b\right\|_ {C_{\mathcal{X}}^{\frac{\alpha}{2},1+\alpha}\left([0,T]\times \mathbb{R}^{n};\mathbb{R}^{n_1}\right)}$ only. 
\end{Prop}
\begin{proof}
The idea to prove the Schauder estimate \eqref{LHJB in R^n SE} is similar to the one in Proposition \ref{Prop_LHJB in R^n USE&HSE}. While the difference is that we apply the abstract theory of singular integrals and fractional integrals (see Subsection \ref{Subsec_2.2}) to the space $\left([0,T]\times B_\delta\left(x_0\right), d_p, dt d x\right)$ instead of $\left(B_\delta\left(x_0\right),d_{cc},dx\right)$, where $d_p$ is the parabolic Carnot-Carath\'{e}odory distance, i.e.
\begin{equation*}
d_p((t,x),(s,y))=\sqrt{|t-s|+d_{cc}(x,y)^2},\quad(t,x),(s,y)\in\mathbb{R}\times\mathbb{G}.
\end{equation*}
It can be known from \cite[Lemma 3.3]{07BB} that, for any $\delta>0$, $x_0\in\mathbb{R}^{n}$, $([0,T]\times B_{\delta}(x_0),d_{p},dt dx)$ is a space of homogeneous type.

Fix any point $x_0 \in \mathbb{R}^{n}$. For any $\delta > 0$, set
$$
v_\delta(t,x):=(z(t,x)-g(x))\varphi_{\delta}(x),
$$
where $\varphi_{\delta}\in C_0^{\infty}\left(\mathbb{R}^{n}\right)$ is the cutoff function satisfying $0 \leq \varphi_{\delta} \leq 1$, $\varphi_{\delta} \equiv 1$ on $B_{\frac{\delta}{2}}(x)$, $\operatorname{supp}\left(\varphi_{\delta}\right) \subset B_{\delta}(x)$, and for any multi-index $I=(i_1,i_2,\ldots,i_k), i_j\in\{1,\ldots,n_1\}$ with the length $|I|=k\in\mathbb{N}$, $\alpha\in(0,1]$,
\begin{equation}\label{cut-off properties}
\left|X_I \varphi_{\delta}\right| \leq c_{k}\delta^{-k},\quad
\left[X_I \varphi_{\delta}\right]_{C_{\mathcal{X}}^{\alpha}(\mathbb{R}^{n})} \leq c_{k,\alpha}\delta^{-(k+\alpha)}
\end{equation}
(see \cite[Lemma 2.4]{25JWY}). We observe that $v_\delta$ is a solution to the following Cauchy problem:
\begin{equation*}
\begin{cases}
\mathcal{H} v_{\delta}=Q_{\delta}, & \text {in }(0, T] \times \mathbb{R}^{n}, \\
v_{\delta}(0,x)=0, & \text {in } \mathbb{R}^{n},
\end{cases}
\end{equation*}
where $\mathcal{H}=\partial_t-\Delta_{\mathcal{X}}$, which is a $2$-homogeneous left-invariant H\"{o}rmander's operator on $\mathbb{R}\times\mathbb{G}$, and
\begin{align*}
Q_{\delta}(t,x):= &-b(t, x) \cdot D_{\mathcal{X}} z(t,x)\varphi_{\delta}(x)-c(t,x)z(t,x)\varphi_{\delta}(x)+f(t,x)\varphi_{\delta}(x) \\
& +\Delta_{\mathcal{X}}g(x)\varphi_{\delta}(x)-2D_{\mathcal{X}}z(t,x)\cdot D_{\mathcal{X}}\varphi_{\delta}(x)+2D_{\mathcal{X}}g(x)\cdot D_{\mathcal{X}}\varphi_{\delta}(x) \\
& -(z(t, x)-g(x)) \Delta_{\mathcal{X}} \varphi_{\delta}(x)
\end{align*}
for any $(t, x) \in (0, T] \times \mathbb{R}^{n}$. Using Lemma \ref{Lem_HE existence and uniqueness}, we obtain that for any $(t,x)\in[0,T]\times\mathbb{R}^{n}$,
\begin{align}\label{v_delta 2}
v_{\delta}(t,x) & =\int_{0}^{t}\int_{\mathbb{R}^{n}}\Gamma_0\left(t-s,y^{-1}\circ x\right)Q_{\delta}(s,y)dyds\notag \\
& =\int_{0}^{T}\int_{\mathbb{R}^{n}}\Gamma_0\left(t-s,y^{-1}\circ x\right)Q_{\delta}(s,y)dyds,
\end{align}
where $\Gamma_0$ is the fundamental solution for the operator $\mathcal{H}$ and $\Gamma_0(t,x)=0$ for $(t,x)\in(-\infty,0]\times\mathbb{R}^{n}$.

Obviously $v_{\delta}\in C_{{\mathcal{X}}}^{1,2} \left((0,T]\times\mathbb{R}^{n}\right)$ and has derivatives of the following form (see \cite[Lemma 3.3]{03BLU}):
\begin{align*}
X_j X_i v_{\delta}(t, x) = & P V \int_{[0,T]\times\mathbb{R}^{n}}X_j X_i\Gamma_0\left(t-s,y^{-1}\circ x\right)Q_{\delta}(s,y)dsdy \\
= & \int_{0}^{T}\int_{\mathbb{R}^{n}}X_j X_i\Gamma_0\left(t-s,y^{-1}\circ x\right)Q_{\delta}(s,y)dyds
\end{align*}
for any $i,j\in\{1,\ldots,n_1\}$, and
\begin{align*}
\partial_t v_{\delta}(t, x) = & P V \int_{[0,T]\times\mathbb{R}^{n}}\partial_t\Gamma_0\left(t-s,y^{-1}\circ x\right)Q_{\delta}(s,y)dsdy \\
= & \int_{0}^{T}\int_{\mathbb{R}^{n}}\partial_t\Gamma_0\left(t-s,y^{-1}\circ x\right)Q_{\delta}(s,y)dyds.
\end{align*}

Further, we apply Lemma \ref{Lem_X_i X_J T f} to obtain that for any integer $k\geq2$, multi-index $J=(j_1,j_2,\ldots,j_{k}),j_l\in\{1,\ldots,n_1\}$ and $(t,x)\in(0,T]\times\mathbb{R}^{n}$,
\begin{align*}
X_J v_{\delta}(t, x) = \sum_{p_2,\ldots,p_{k-1}=1}^{n_1} \int_{0}^{T}\int_{\mathbb{R}^{n}}\mathcal{K}^{P_{-1,-k}}\left(t-s,y^{-1}\circ x\right)X_{P_{-1,-k}} Q_{\delta}(s,y)dyds,
\end{align*}
where $P_{-1,-k}=(p_2,\ldots,p_{k-1})$, i.e. $P=(p_1,\ldots,p_k)$ with its $1$-th and $k$-th components removed, and $\mathcal{K}^{P_{-1,-k}}$ are $-(Q+2)$-homogeneous kernels having the form
\begin{align*}
& \mathcal{K}^{P_{-1,-k}}\left(t-s,y^{-1}\circ x\right) \\
= & \begin{cases}
X_{j_1}\sum_{i=1}^{M}\mathcal{R}_{i}^{J_{-1,-k},P_{-1,-k}}(y^{-1}\circ x) \\
\qquad\quad\times X_{l_1^{J_{-1,-k},P_{-1,-k}}}\cdots X_{l_i^{J_{-1,-k},P_{-1,-k}}} X_{j_k}\Gamma_0\left(t-s,y^{-1}\circ x\right), & \mbox{if } \dim(P_{-1,-k})\geq1, \\
X_{j_1}X_{j_2}\Gamma_0\left(t-s,y^{-1}\circ x\right), & \mbox{if } \dim(P_{-1,-k})=0
\end{cases}
\end{align*}
for a certain family of finite number of $i$-homogeneous polynomials $\mathcal{R}_i^{J_{-1,-k},P_{-1,-k}},i\in\{1,\ldots,M\}$ and left-invariant vector fields $\{X_{l_i^{J_{-1,-k},P_{-1,-k}}}\}_{i=1}^{M},l_i^{J_{-1,-k},P_{-1,-k}}\in\{1,\ldots,n_1\}$ with $M=M(J_{-1,-k},P_{-1,-k})\in\mathbb{Z}_+$. Similarly, $J_{-1,-k}=(j_2,\ldots,j_{k-1})$ denotes $J$ with its $1$-th and $k$-th components removed. Moreover, since
\begin{align*}
& X_{j_1}\sum_{i=1}^{M}\mathcal{R}_{i}^{J_{-1,-k},P_{-1,-k}}(y^{-1}\circ x) X_{l_1^{J_{-1,-k},P_{-1,-k}}}\cdots X_{l_i^{J_{-1,-k},P_{-1,-k}}} X_{j_k}\Gamma_0\left(t-s,y^{-1}\circ x\right) \\
= & \sum_{i=1}^{M}X_{j_1}\mathcal{R}_{i}^{J_{-1,-k},P_{-1,-k}}(y^{-1}\circ x) X_{l_1^{J_{-1,-k},P_{-1,-k}}}\cdots X_{l_i^{J_{-1,-k},P_{-1,-k}}} X_{j_k}\Gamma_0\left(t-s,y^{-1}\circ x\right) \\
& +\sum_{i=1}^{M}\mathcal{R}_{i}^{J_{-1,-k},P_{-1,-k}}(y^{-1}\circ x) X_{j_1}X_{l_1^{J_{-1,-k},P_{-1,-k}}}\cdots X_{l_i^{J_{-1,-k},P_{-1,-k}}} X_{j_k}\Gamma_0\left(t-s,y^{-1}\circ x\right),
\end{align*}
then from Lemma \ref{Lem_mathcal K_I,q,l_singular integral theory} with $|I|+2q-l=2$, we know that $\mathcal{K}^{P_{-1,-k}}$ satisfy the conditions in Lemma \ref{Lem_HCSIO} and Lemma \ref{Lem_HCFIO}.

For any couple of functions $\Phi,\Psi\in C_{\mathcal{X}}^{\frac{\alpha}{2},\alpha}\left([0,T]\times B_\delta\left(x_0\right)\right)$, $k\in\mathbb{N}$, $\alpha\in(0,1]$, one has
\begin{equation}\label{Phi Psi C_mathcal X^alpha/2,k+alpha}
\|\Phi\Psi\|_{C_{\mathcal{X}}^{\frac{\alpha}{2},k+\alpha}\left([0,T]\times B_\delta\left(x_0\right)\right)} \leq c_{k}\|\Phi\|_{C_{\mathcal{X}}^{\frac{\alpha}{2},k+\alpha}\left([0,T]\times B_\delta\left(x_0\right)\right)} \|\Psi\|_{C_{\mathcal{X}}^{\frac{\alpha}{2},k+\alpha}\left([0,T]\times B_\delta\left(x_0\right)\right)}
\end{equation}
for some constant $c_{k}>0$ depending only on $k$ (see \cite[Proposition 4.2(i)]{07BB}).

Fix any $\delta>0$. Applying Lemma \ref{Lem_HCSIO} and Lemma \ref{Lem_HCFIO} and then using \eqref{Phi Psi C_mathcal X^alpha/2,k+alpha} and \eqref{cut-off properties}, we obtain that for any $k\geq 2$,
\begin{align}\label{v_delta_SE}
& \left\|v_\delta\right\|_{C_{\mathcal{X}}^{\frac{\alpha}{2},k+\alpha}\left([0,T]\times B_\delta\left(x_0\right)\right)} +\left\|v_\delta\right\|_{C_{\mathcal{X}}^{1+\frac{\alpha}{2},\alpha}\left([0,T]\times B_\delta\left(x_0\right)\right)} \\
\leq & C\left\|Q_\delta\right\|_{C_{\mathcal{X}}^{\frac{\alpha}{2},k-2+\alpha}\left([0,T]\times B_\delta\left(x_0\right)\right)}\notag \\
\leq & C\bigg(\left(\delta^{-(k-2+\alpha)}\|b\|_{C_{\mathcal{X}}^{\frac{\alpha}{2},k-2+\alpha}\left([0,T]\times B_\delta\left(x_0\right);\mathbb{R}^{n_1}\right)}+2\delta^{-(k-1+\alpha)}\right)\|D_{\mathcal{X}}z\|_ {C_{\mathcal{X}}^{\frac{\alpha}{2},k-2+\alpha}\left([0,T]\times B_\delta\left(x_0\right)\right)}\notag \\
& +\left(\delta^{-(k-2+\alpha)}\|c\|_{C_{\mathcal{X}}^{\frac{\alpha}{2},k-2+\alpha}\left([0,T]\times B_\delta\left(x_0\right)\right)}+\delta^{-(k+\alpha)}\right)\|z\|_{C_{\mathcal{X}}^{\frac{\alpha}{2},k-2+\alpha}\left([0,T]\times B_\delta\left(x_0\right)\right)}\notag \\
& +\delta^{-(k-2+\alpha)}\|f\|_{C_{\mathcal{X}}^{\frac{\alpha}{2},k-2+\alpha}\left([0,T]\times B_\delta\left(x_0\right)\right)}+\delta^{-(k+\alpha)}\|g\|_{C_{\mathcal{X}}^{k-2+\alpha}\left(B_\delta\left(x_0\right)\right)}\notag \\
& +2\delta^{-(k-1+\alpha)}\|g\|_{C_{\mathcal{X}}^{k-1+\alpha}\left(B_\delta\left(x_0\right)\right)} +\delta^{-(k-2+\alpha)}\|g\|_{C_{\mathcal{X}}^{k+\alpha}\left(B_\delta\left(x_0\right)\right)}\bigg),\notag
\end{align}
where $C>0$ depends on $\mathbb{G}$, $T$, $\alpha$ and $k$ only.

It can be found that for any function $\phi\in C_{\mathcal{X}}^{\frac{\alpha}{2},k+\alpha}\left([0,T]\times\mathbb{R}^{n}\right)$, $k\in\mathbb{N}$ and any $\delta'>0$, 
\begin{align}\label{phi C^k+alpha [0,T]timesR^3}
\|\phi\|_{C_{\mathcal{X}}^{\frac{\alpha}{2},k+\alpha}\left([0,T]\times\mathbb{R}^{n}\right)}\leq & \sup_{x_0\in\mathbb{R}^{n}} \sum_{|I|\leq k} [X_I\phi]_{C_{\mathcal{X}}^{\frac{\alpha}{2},\alpha}\left([0,T]\times B_{\delta'}(x_0)\right)} \\
& +\left(1+2{\delta'}^{-\alpha}\right)\sup_{x_0\in\mathbb{R}^{n}} \|\phi\|_{C_{\mathcal{X}}^{0,k}\left([0,T]\times B_{\delta'}(x_0)\right)},\notag
\end{align}
by means of the fact that for any multi-index $I=(i_1,i_2,\ldots,i_k), i_j\in\{1,\ldots,n_1\}$ with the length $|I|=k$,
\begin{align*}
[X_I\phi]_{C_{\mathcal{X}}^{\frac{\alpha}{2},\alpha}\left([0,T]\times\mathbb{R}^{n}\right)} \leq & \sup_{\substack{(t,x) \neq (s,y) \\ d_{cc}(x,y)<\delta'}}\frac{|X_I\phi(t,x)-X_I\phi(s,y)|}{\left(|t-s|+d_{cc}(x,y)^2\right)^{\frac{\alpha}{2}}} +\sup_{\substack{(t,x) \neq (s,y) \\ d_{cc}(x,y)\geq\delta'}}\frac{|X_I\phi(t,x)-X_I\phi(s,y)|}{\left(|t-s|+d_{cc}(x,y)^2\right)^{\frac{\alpha}{2}}} \\
\leq & \sup_{x_0\in\mathbb{R}^{n}} [X_I\phi]_{C_{\mathcal{X}}^{\frac{\alpha}{2},\alpha}\left([0,T]\times B_{\delta'}(x_0)\right)} +2{\delta'}^{-\alpha}\|X_I\phi\|_{L^\infty\left([0,T]\times\mathbb{R}^{n}\right)},
\end{align*}
and
\begin{equation*}
\|\phi\|_{C_{\mathcal{X}}^{0,k}\left([0,T]\times\mathbb{R}^{n}\right)}=\sup_{x_0\in\mathbb{R}^{n}} \|\phi\|_{C_{\mathcal{X}}^{0,k}\left([0,T]\times B_{\delta'}(x_0)\right)}.
\end{equation*}
Recalling that $v_\delta=z-g$ in $[0,T]\times B_{\frac{\delta}{2}}(x_0)$, from \eqref{phi C^k+alpha [0,T]timesR^3}, we finally obtain that for any $k\geq 2$,
\begin{align}\label{z SE}
& \left\|z\right\|_{C_{\mathcal{X}}^{\frac{\alpha}{2},k+\alpha}\left([0,T]\times\mathbb{R}^{n}\right)} +\left\|z\right\|_{C_{\mathcal{X}}^{1+\frac{\alpha}{2},\alpha}\left([0,T]\times\mathbb{R}^{n}\right)} \\
\leq & \left\|g\right\|_{C_{\mathcal{X}}^{k+\alpha}\left(\mathbb{R}^{n}\right)}+\left\|z-g\right\|_ {C_{\mathcal{X}}^{\frac{\alpha}{2},k+\alpha}\left([0,T]\times \mathbb{R}^{n}\right)} +\left\|z-g\right\|_ {C_{\mathcal{X}}^{1+\frac{\alpha}{2},\alpha}\left([0,T]\times \mathbb{R}^{n}\right)} \notag\\
\leq & \left\|g\right\|_{C_{\mathcal{X}}^{k+\alpha}\left(\mathbb{R}^{n}\right)} +\sup_{x_0\in\mathbb{R}^{n}} \left(\sum_{|I|\leq k} [X_Iv_\delta]_{C_{\mathcal{X}}^{\frac{\alpha}{2},\alpha}\left([0,T]\times B_{\frac{\delta}{2}}(x_0)\right)}+[v_\delta]_{C_{\mathcal{X}}^{1+\frac{\alpha}{2},\alpha}\left([0,T]\times B_{\frac{\delta}{2}}(x_0)\right)}\right. \notag \\
& \left.+\left(1+2^{1+\alpha}\delta^{-\alpha}\right)\bigg(\|v_\delta\|_{C_{\mathcal{X}}^{0,k}\left([0,T]\times B_{\frac{\delta}{2}}(x_0)\right)}+\|v_\delta\|_{C_{\mathcal{X}}^{1,0}\left([0,T]\times B_{\frac{\delta}{2}}(x_0)\right)}\bigg)\right)\notag \\
\leq & C\left(\left\|g\right\|_{C_{\mathcal{X}}^{k+\alpha}\left(\mathbb{R}^{n}\right)} +\left\|f\right\|_{C_{\mathcal{X}}^{\frac{\alpha}{2},k-2+\alpha}\left([0,T]\times \mathbb{R}^{n}\right)} +\left\|z\right\|_{C_{\mathcal{X}}^{\frac{\alpha}{2},k-1+\alpha}\left([0,T]\times \mathbb{R}^{n}\right)}\right)\notag \\
\leq & C\left(\left\|g\right\|_{C_{\mathcal{X}}^{k+\alpha}\left(\mathbb{R}^{n}\right)} +\left\|f\right\|_{C_{\mathcal{X}}^{\frac{\alpha}{2},k-2+\alpha}\left([0,T]\times \mathbb{R}^{n}\right)} +\left\|z\right\|_{C_{\mathcal{X}}^{\frac{\alpha}{2},1+\alpha}\left([0,T]\times \mathbb{R}^{n}\right)}\right),\notag
\end{align}
where $C>0$ depends on $\mathbb{G}$, $\left\|b\right\|_{C_{\mathcal{X}}^{\frac{\alpha}{2},k-2+\alpha}\left([0,T]\times \mathbb{R}^{n};\mathbb{R}^{n_1}\right)}$, $\left\|c\right\|_{C_{\mathcal{X}}^{\frac{\alpha}{2},k-2+\alpha}\left([0,T]\times \mathbb{R}^{n}\right)}$, $T$, $k$ and $\alpha$ only. 

To prove \eqref{LHJB in R^n SE}, we need the estimate of $\left\|z\right\|_{C^{\frac{\alpha}{2},1+\alpha}_{\mathcal{X}}\left([0,T]\times\mathbb{R}^{n}\right)}$. Note that for any $i\in\{1,\ldots,n\}$, $X_i$ is the $\alpha_i$-homogeneous and left-invariant vector field on $\mathbb{G}$ which coincides with $\partial_{x_i}$ at the origin, and the convolution on $\mathbb{G}$ satisfies the following property (see \cite[Proposition 3.47]{22BB}):
\begin{equation}\label{convolution on homogeneous groups_property}
\int_{\mathbb{R}^{n}}X_if(y)g(y^{-1}\circ x)dy=\int_{\mathbb{R}^{n}}f(y)X_i^R g(\cdot)(y^{-1}\circ x)dy
\end{equation}
for any couple of functions $f,g:\mathbb{R}^{n}\to\mathbb{R}$ for which the integral makes sense. Here, $X_i^R$ is the $\alpha_i$-homogeneous and right-invariant vector field on $\mathbb{G}$ satisfying $X_{i}^R(0)=\partial_{x_{i}}|_{0}=X_{i}(0)$. From \cite[Remark 3.32]{22BB}, since each element $X_{{n_1}+1},\ldots,X_{n}$ is a linear combination (with constant coefficients) of commutators of $X_1,\ldots,X_{n_1}$ with finite steps, it is known that
\begin{equation}\label{X_i^R=sum X_J_j,l}
X_i^R=\sum_{j=i}^{n}\tilde{r}_{ij}(x)X_j=\sum_{j=i}^{n}\sum_{J_{j,l}\in\{1,\ldots,n_1\}^{\alpha_j}}c_{j,l}\tilde{r}_{ij}(x)X_{J_{j,l}},
\end{equation}
with $\tilde{r}_{ij}(x)$ being $(\alpha_j-\alpha_i)$-homogeneous polynomials.

According to \eqref{v_delta 2} and using \eqref{convolution on homogeneous groups_property}, we have that for any $j\in\{0,1\}$, multi-index $I=(i_1,\ldots,i_{j}),i_l\in\{1,\ldots,n_1\}$ and $(t,x)\in[0,T]\times\mathbb{R}^{n}$,
\begin{align*}
& X_I v_{\delta}(t,x) \\
= &\int_{0}^{T}\int_{\mathbb{R}^{n}}X_I\Gamma_0\left(t-s,y^{-1}\circ x\right)R_{\delta}(s,y)dyds \\
& -\int_{0}^{T}\int_{\mathbb{R}^{n}}X_I\Gamma_0\left(t-s,y^{-1}\circ x\right)\bigg(b\cdot D_{\mathcal{X}}z\varphi_{\delta}+2D_{\mathcal{X}}z\cdot D_{\mathcal{X}}\varphi_{\delta}-\Delta_{\mathcal{X}}g\varphi_{\delta}\bigg)(s,y)dyds \\
= &\int_{0}^{T}\int_{\mathbb{R}^{n}}X_I\Gamma_0\left(t-s,y^{-1}\circ x\right)R_{\delta}(s,y)dyds \\
& +\int_{0}^{T}\int_{\mathbb{R}^{n}}X_I\Gamma_0\left(t-s,y^{-1}\circ x\right)\bigg(\operatorname{div}_{\mathcal{X}}(b\varphi_{\delta})z+2\Delta_{\mathcal{X}}\varphi_{\delta}z -D_{\mathcal{X}}g\cdot D_{\mathcal{X}}\varphi_{\delta}\bigg)(s,y)dyds \\
& -\int_{0}^{T}\int_{\mathbb{R}^{n}}\sum_{i=1}^{n_1}X_i^R X_I\Gamma_0\left(t-s,\cdot)(y^{-1}\circ x\right)\left(b_i\varphi_{\delta}z+2X_i\varphi_{\delta}z-X_ig\varphi_{\delta}\right)(s,y)dyds,
\end{align*}
where
\begin{align*}
R_{\delta}(t,x):= &-c(t,x)z(t,x)\varphi_{\delta}(x)+f(t,x)\varphi_{\delta}(x)+2D_{\mathcal{X}}g(x)\cdot D_{\mathcal{X}}\varphi_{\delta}(x) \\
& -(z(t, x)-g(x)) \Delta_{\mathcal{X}} \varphi_{\delta}(x),\,(t,x)\in(0, T]\times\mathbb{R}^{n}.
\end{align*}
It can be known from \eqref{X_i^R=sum X_J_j,l} and Lemma \ref{Lem_mathcal K_I,q,l_singular integral theory} with $|I|+2q-l=2$ that $X_i^R X_I\Gamma_0$ satisfy the conditions in Lemma \ref{Lem_HCSIO} and Lemma \ref{Lem_HCFIO}. Then, similar to the estimate in \eqref{v_delta_SE} and \eqref{z SE}, we can obtain
\begin{equation}\label{z_C_mathcal X^alpha/2,1+alpha}
\left\|z\right\|_{C_{\mathcal{X}}^{\frac{\alpha}{2},1+\alpha}\left([0,T]\times\mathbb{R}^{n}\right)} \leq C\left(\left\|g\right\|_{C_{\mathcal{X}}^{1+\alpha}\left(\mathbb{R}^{n}\right)} +\left\|f\right\|_{C_{\mathcal{X}}^{\frac{\alpha}{2},\alpha}\left([0,T]\times \mathbb{R}^{n}\right)} +\left\|z\right\|_{C_{\mathcal{X}}^{\frac{\alpha}{2},\alpha}\left([0,T]\times \mathbb{R}^{n}\right)}\right),
\end{equation}
where the constant $C>0$ depends on $\mathbb{G}$, $\left\|b\right\|_ {C_{\mathcal{X}}^{\frac{\alpha}{2},1+\alpha}\left([0,T]\times \mathbb{R}^{n};\mathbb{R}^{n_1}\right)}$, $\left\|c\right\|_{C_{\mathcal{X}}^{\frac{\alpha}{2},\alpha}\left([0,T]\times \mathbb{R}^{n}\right)}$, $T$ and $\alpha$ only. 

Due to \eqref{LHJB in R^n_HSE_L^infty} and \eqref{LHJB in R^n_prior C^1 estimate}, we have that for any $(t,x),(s,y)\in[0,T]\times\mathbb{R}^n$,
\begin{align*}
& |z(t,x)-z(s,y)| \\
\leq & |z(t,x)-z(s,x)|+|z(s,x)-z(s,y)| \\
\leq & C\left(\left\|g\right\|_{C_{\mathcal{X}}^{1}\left(\mathbb{R}^{n}\right)} +\|f\|_{L^{\infty}\left((0,T)\times\mathbb{R}^{n}\right)}\right)|t-s|^{\frac{1}{2}}+\sup_{t\in[0, T]}[z(t,\cdot)]_{C_{\mathcal{X}}^{\alpha}\left(\mathbb{R}^{n}\right)}d_{cc}(x,y)^{\alpha} \\
\leq & C\left(\left\|g\right\|_{C_{\mathcal{X}}^{1}\left(\mathbb{R}^{n}\right)} +\|f\|_{L^{\infty}\left((0,T)\times\mathbb{R}^{n}\right)}\right) \left(|t-s|^{\frac{1}{2}}+d_{cc}(x,y)^{\alpha}\right),
\end{align*}
where we note that for any $t\in[0,T]$ and $\delta'>0$,
\begin{align*}
[z(t,\cdot)]_{C_{\mathcal{X}}^{\alpha}\left(\mathbb{R}^{n}\right)} \leq & \sup_{x_0\in\mathbb{R}^{n}}[z(t,\cdot)]_{C_{\mathcal{X}}^{\alpha}\left(B_{\delta'}(x_0)\right)} +2{\delta'}^{-\alpha}\|z(t,\cdot)\|_{L^\infty\left(\mathbb{R}^{n}\right)} \\
\leq & {\delta'}^{1-\alpha}\sup_{x_0\in\mathbb{R}^{n}}\|D_{\mathcal{X}}z(t,\cdot)\|_{L^{\infty}\left(B_{5\delta'}(x_0)\right)} +2{\delta'}^{-\alpha}\|z(t,\cdot)\|_{L^\infty\left(\mathbb{R}^{n}\right)}
\end{align*}
by \cite[Proposition 4.2(ii)]{07BB}. Hence,
\begin{equation*}
\left\|z\right\|_{C_{\mathcal{X}}^{\frac{\alpha}{2},\alpha}\left([0,T]\times \mathbb{R}^{n}\right)} \leq C\left(\left\|g\right\|_{C_{\mathcal{X}}^{1}\left(\mathbb{R}^{n}\right)} +\|f\|_{L^{\infty}\left((0,T)\times\mathbb{R}^{n}\right)}\right).
\end{equation*}
Combining the above inequality with \eqref{z_C_mathcal X^alpha/2,1+alpha} yields that
\begin{equation*}
\left\|z\right\|_{C_{\mathcal{X}}^{\frac{\alpha}{2},1+\alpha}\left([0,T]\times\mathbb{R}^{n}\right)} \leq C\left(\left\|g\right\|_{C_{\mathcal{X}}^{1+\alpha}\left(\mathbb{R}^{n}\right)} +\left\|f\right\|_{C_{\mathcal{X}}^{\frac{\alpha}{2},\alpha}\left([0,T]\times \mathbb{R}^{n}\right)}\right),
\end{equation*}
where the constant $C>0$ depends on $\mathbb{G}$, $\left\|b\right\|_ {C_{\mathcal{X}}^{\frac{\alpha}{2},1+\alpha}\left([0,T]\times \mathbb{R}^{n};\mathbb{R}^{n_1}\right)}$, $\left\|c\right\|_{C_{\mathcal{X}}^{\frac{\alpha}{2},\alpha}\left([0,T]\times \mathbb{R}^{n}\right)}$, $T$ and $\alpha$ only. Putting the above inequality into \eqref{z SE}, we finally get that \eqref{LHJB in R^n SE} holds.
\end{proof}
From now on, we shall restrict the subsequent analysis to $\mathbb{T}_{\mathbb{G}}$.
The following proposition gives the existence and uniqueness of the solution to the linear degenerate Cauchy problem on $[0,T]\times\mathbb{T}_{\mathbb{G}}$.
\begin{Prop}\label{Prop_LHJB well-posedness}
Let $\alpha\in(0,1)$. For any $i,j\in\{1,2,\ldots,n_1\}$, assume $a_{i,j}(t,x)$, $b_i(t,x)$, $c(t,x)$, $f(t,x)$ are continuous on $[0,T]\times\mathbb{T}_{\mathbb{G}}$, satisfying $a_{i,j}\in C_{{\mathcal{X}}}^{\frac{\alpha}{2},\alpha} \left([0,T]\times\mathbb{T}_{\mathbb{G}}\right)$, $b_i$, $c$, $f\in B\left([0,T];C_{\mathcal{X}}^{\alpha}\left(\mathbb{T}_{\mathbb{G}}\right)\right)$, and $g(x)\in C\left(\mathbb{T}_{\mathbb{G}}\right)$. Then there exists a unique solution $z(t,x)\in C_{{\mathcal{X}}}^{1,2} \left((0,T]\times\mathbb{T}_{\mathbb{G}}\right)\cap C\left([0,T]\times\mathbb{T}_{\mathbb{G}}\right)$ to the equation
\begin{equation}\label{LHJB 2}
\begin{cases}
H z(t,x)=f(t,x),& \text {in }(0, T] \times \mathbb{T}_{\mathbb{G}}, \\
z(0,x)=g(x),& \text {in }\mathbb{T}_{\mathbb{G}},
\end{cases}
\end{equation}
where the operator
\begin{equation*}
H:=\partial_t-\sum_{i,j=1}^{n_1}a_{i,j}(t,x)X_i X_j-\sum_{i=1}^{n_1}b_i(t,x)X_i-c(t,x).
\end{equation*}
Moreover, $z$ has the form
\begin{equation*}
z(t,x)=\int_{\mathbb{T}_{\mathbb{G}}}\sum_{k\in\mathbb{Z}^{n}}\Gamma(t,k\circ x;0,y)g(y)d y + \int_{0}^{t}\int_{\mathbb{T}_{\mathbb{G}}}\sum_{k\in\mathbb{Z}^{n}}\Gamma(t,k\circ x;s,y)f(s,y)d y d s
\end{equation*}
for any $(t,x)\in[0,T]\times\mathbb{T}_{\mathbb{G}}$. Here, $\Gamma:(\mathbb{R}\times\mathbb{G})\times(\mathbb{R}\times\mathbb{G})\to\mathbb{R}$ is the fundamental solution for $H$ on $\mathbb{R}\times\mathbb{G}$ and $\Gamma\geq 0$.
\end{Prop}
\begin{proof}
Since the functions on $\mathbb{T}_{\mathbb{G}}$ are $1_{\mathbb{G}}$-periodic on $\mathbb{R}^n$, and by Remark \ref{Rem_Holder norm R^n=T_mathbb G}, we may assume without loss of generality that the bounded functions $a_{i,j}(t,x)$, $b_i(t,x)$, $c(t,x)$, $f(t,x)$ and $g(x)$ are continuous on $[0,T]\times\mathbb{R}^{n}$ and $1_{\mathbb{G}}$-periodic on $\mathbb{R}^{n}$ for all $t\in[0,T]$, satisfying $a_{i,j}\in C_{{\mathcal{X}}}^{\frac{\alpha}{2},\alpha} \left([0,T]\times\mathbb{R}^{n}\right)$, $b_i$, $c$, $f\in B\left([0,T];C_{\mathcal{X}}^{\alpha}\left(\mathbb{R}^{n}\right)\right)$ and $g(x)\in C\left(\mathbb{R}^{n}\right)$.

For any fixed point $(t_0,x_0) \in \mathbb{R}\times\mathbb{R}^n$, denote the parabolic operators
\begin{equation*}\label{mathcal H_(t_0,x_0)}
\mathcal{H}_{(t_0,x_0)}:=\partial_t-\sum_{i,j=1}^{n_1}a_{i,j}(t_0,x_0)X_i X_j.
\end{equation*}
Let $\Gamma_{(t_0,x_0)}(t,x;s,y)=\Gamma_{(t_0,x_0)}(t-s,y^{-1}\circ x)$ be the fundamental solution for $\mathcal{H}_{(t_0,x_0)}$. We can see \cite{02BLU} for more details about $\Gamma_{(t_0,x_0)}$.
It follows from \cite[P. 74]{10BBLU} that the fundamental solution $\Gamma$ for $H$ can be written as
\begin{equation*}
\Gamma(t,x;s,y)=\Gamma_{(s,y)}(t-s,y^{-1}\circ x)+\int_{s}^{t}\int_{\mathbb{R}^{n}}\Gamma_{(\tau,z)}(t-\tau,z^{-1}\circ x)\Phi(\tau,z;s,y)d z d \tau
\end{equation*}
for any $(t,x),(s,y)\in \mathbb{R}\times\mathbb{R}^n$, $t>s$. Here, $\Phi$ is a certain determined kernel function, having the form
\begin{equation*}
\Phi(t,x;s,y)=\sum_{j=1}^{\infty}Z_j(t,x;s,y),\,(t,x),(s,y)\in \mathbb{R}\times\mathbb{R}^n,\,t>s,
\end{equation*}
where
\begin{equation*}\label{Z_1 Def}
Z_1:=-H\Gamma_{(s,y)}(t-s,y^{-1}\circ x),\, (t,x) \neq (s,y) \in \mathbb{R}\times\mathbb{R}^n,
\end{equation*}
and for every $j \in \mathbb{Z}_+$,
\begin{equation*}\label{Z_j+1}
Z_{j+1}(t,x;s,y):=\int_{s}^{t}\int_{\mathbb{R}^n}Z_1(t,x;\tau,z)Z_{j}(\tau,z;s,y)dz d\tau,\, (t,x),(s,y)\in \mathbb{R}\times\mathbb{R}^n,\,t>s.
\end{equation*}

Now let us claim that for any $k\in\mathbb{Z}^{n}$, $(t,x),(s,y)\in \mathbb{R}\times\mathbb{R}^n$ and $t>s$,
\begin{equation}\label{Gamma_periodic property}
\Gamma\left(t,k\circ x;s,y\right)=\Gamma\left(t,x;s,k^{-1}\circ y\right).
\end{equation}
In fact, for any $k\in\mathbb{Z}^{n}$ and $(t_0,x_0)\in\mathbb{R}\times\mathbb{R}^n$,
\begin{equation*}
\mathcal{H}_{(t_0,x_0)}=\partial_t-\sum_{i,j=1}^{n_1}a_{i,j}(t_0,x_0)X_i X_j=\partial_t-\sum_{i,j=1}^{n_1}a_{i,j}(t_0,k\circ x_0)X_i X_j=\mathcal{H}_{(t_0,k\circ x_0)},
\end{equation*}
then we get $\Gamma_{(t_0,x_0)}=\Gamma_{(t_0,k\circ x_0)}$. Since $H=H-\mathcal{H}_{(s,y)}+\mathcal{H}_{(s,y)}$, it can be found that
\begin{align*}
& Z_1(t,k\circ x;s,y) \\
= & \left(\mathcal{H}_{(s,y)}-H\right)\Gamma_{(s,y)}(t-s,y^{-1}\circ (k\circ x)) \\
= & \sum_{i,j=1}^{n_1}\left(a_{i,j}(t,k\circ x)-a_{i,j}(s,y)\right)X_i X_j\Gamma_{(s,y)}(t-s,y^{-1}\circ (k\circ x)) \\
& +\sum_{i=1}^{n_1}b_i(t,k\circ x)X_i\Gamma_{(s,y)}(t-s,y^{-1}\circ (k\circ x))+c(t,k\circ x)\Gamma_{(s,y)}(t-s,y^{-1}\circ (k\circ x)) \\
= & \sum_{i,j=1}^{n_1}\left(a_{i,j}(t,x)-a_{i,j}(s,k^{-1}\circ y)\right)X_i X_j\Gamma_{(s,k^{-1}\circ y)}(t-s,\left(k^{-1}\circ y\right)^{-1}\circ x) \\
& +\sum_{i=1}^{n_1}b_i(t,x)X_i\Gamma_{(s,k^{-1}\circ y)}(t-s,\left(k^{-1}\circ y\right)^{-1}\circ x)+c(t,x)\Gamma_{(s,k^{-1}\circ y)}(t-s,\left(k^{-1}\circ y\right)^{-1}\circ x) \\
= & Z_1(t,x;s,k^{-1}\circ y)
\end{align*}
for any $(t,x) \neq (s,y) \in \mathbb{R}\times\mathbb{R}^n$, where we note that $k^{-1}\in\mathbb{Z}^{n}$ and $k=\left(k^{-1}\right)^{-1}$. Next we assume that $Z_{j}(t,k\circ x;s,y)=Z_{j}(t,x;s,k^{-1}\circ y)$ for $j \in \mathbb{Z}_+$, then
\begin{align*}
Z_{j+1}(t,k\circ x;s,y)= & \int_{s}^{t}\int_{\mathbb{R}^n}Z_1(t,k\circ x;\tau,z)Z_{j}(\tau,z;s,y)dzd\tau \\
= & \int_{s}^{t}\int_{\mathbb{R}^n}Z_1(t,x;\tau,k^{-1}\circ z)Z_{j}(\tau,z;s,y)dzd\tau \\
= & \int_{s}^{t}\int_{\mathbb{R}^n}Z_1(t,x;\tau,\xi)Z_{j}(\tau,k\circ \xi;s,y)d\xi d\tau \\
= & \int_{s}^{t}\int_{\mathbb{R}^n}Z_1(t,x;\tau,\xi)Z_{j}(\tau,\xi;s,k^{-1}\circ y)d\xi d\tau \\
= & Z_{j+1}(t,x;s,k^{-1}\circ y).
\end{align*}
Hence, by induction, for every $j \in \mathbb{Z}_+$, we have $Z_{j}(t,k\circ x;s,y)=Z_{j}(t,x;s,k^{-1}\circ y)$. Therefore, for any $k\in\mathbb{Z}^{n}$, $(t,x),(s,y)\in \mathbb{R}\times\mathbb{R}^n$ and $t>s$,
\begin{equation*}
\Phi(t,k\circ x;s,y) =\sum_{j=1}^{\infty}Z_j(t,k\circ x;s,y) =\sum_{j=1}^{\infty}Z_j(t,x;s,k^{-1}\circ y) =\Phi(t,x;s,k^{-1}\circ y).
\end{equation*}
In a similar way, it can be obtained that
\begin{align*}
& \Gamma(t,k\circ x;s,y) \\
= & \Gamma_{(s,y)}(t-s,y^{-1}\circ (k\circ x))+\int_{s}^{t}\int_{\mathbb{R}^{n}}\Gamma_{(\tau,z)}(t-\tau,z^{-1}\circ (k\circ x))\Phi(\tau,z;s,y)dz d\tau \\
= & \Gamma_{(s,k^{-1}\circ y)}(t-s,\left(k^{-1}\circ y\right)^{-1}\circ x)+\int_{s}^{t}\int_{\mathbb{R}^{n}}\Gamma_{(\tau,\xi)}(t-\tau,\xi^{-1}\circ x)\Phi(\tau,\xi;s,k^{-1}\circ y)d\xi d\tau,
\end{align*}
which leads to \eqref{Gamma_periodic property}.

According to \cite[Proposition 4.2(1)]{25JWY}, the function 
\begin{equation*}
z(t,x)=\int_{\mathbb{R}^{n}}\Gamma(t,x;0,y)g(y)d y + \int_{0}^{t}\int_{\mathbb{R}^{n}}\Gamma(t,x;s,y)f(s,y)d y d s,
\end{equation*}
which belongs to $C_{{\mathcal{X}}}^{1,2} \left((0,T]\times\mathbb{R}^{n}\right)\cap C\left([0,T]\times\mathbb{R}^{n}\right)$, is a solution to the Cauchy problem
\begin{equation*}
\begin{cases}
Hz(t,x)=f(t,x),& \text {in }(0, T] \times \mathbb{R}^{n}, \\
z(0,x)=g(x),& \text {in }\mathbb{R}^{n}.
\end{cases}
\end{equation*}
From \eqref{Gamma_periodic property} and the $1_{\mathbb{G}}$-periodicity of $f$ and $g$, we find that $z(t,\cdot)$ is $1_{\mathbb{G}}$-periodic on $\mathbb{R}^{n}$ for all $t\in[0,T]$. Thus $z(t,x)\in C_{{\mathcal{X}}}^{1,2} \left((0,T]\times\mathbb{T}_{\mathbb{G}}\right)\cap C\left([0,T]\times\mathbb{T}_{\mathbb{G}}\right)$ is a solution to the Cauchy problem \eqref{LHJB 2} and
\begin{align*}
& z(t,x) \\
= & \int_{\mathbb{T}_{\mathbb{G}}}\sum_{k\in\mathbb{Z}^{n}}\Gamma(t,x;0,k^{-1}\circ y)g(k^{-1}\circ y)d y + \int_{0}^{t}\int_{\mathbb{T}_{\mathbb{G}}}\sum_{k\in\mathbb{Z}^{n}}\Gamma(t,x;s,k^{-1}\circ y)f(s,k^{-1}\circ y)d y d s \\
= & \int_{\mathbb{T}_{\mathbb{G}}}\sum_{k\in\mathbb{Z}^{n}}\Gamma(t,k\circ x;0,y)g(y)d y + \int_{0}^{t}\int_{\mathbb{T}_{\mathbb{G}}}\sum_{k\in\mathbb{Z}^{n}}\Gamma(t,k\circ x;s,y)f(s,y)d y d s.
\end{align*}

The uniqueness of the solution follows directly from the weak maximum principle in \cite[Theorem 13.1]{10BBLU}, which also applies to the bounded domain $U=(0,T)\times\mathbb{T}_{\mathbb{G}}$. Therefore, there exists a unique solution $z(t,x)\in C_{{\mathcal{X}}}^{1,2} \left((0,T]\times\mathbb{T}_{\mathbb{G}}\right)\cap C\left([0,T]\times\mathbb{T}_{\mathbb{G}}\right)$ to the Cauchy problem \eqref{LHJB 2}.
\end{proof}
We are now in a position to prove Theorem \ref{Thm_LHJB wellposed regularity}.
\begin{proof}[\bf{Proof of Theorem \ref{Thm_LHJB wellposed regularity}}]
It follows from Proposition \ref{Prop_LHJB well-posedness} that there exists a unique solution $w(t,x)\in C_{{\mathcal{X}}}^{1,2} \left((0,T]\times\mathbb{T}_{\mathbb{G}}\right)\cap C\left([0,T]\times\mathbb{T}_{\mathbb{G}}\right)$ to the following Cauchy problem
\begin{equation}\label{LHJB forward}
\begin{cases}
\partial_t w-\Delta_{\mathcal{X}} w+b(t,x) \cdot D_{\mathcal{X}} w=f(t,x),& \text {in }(0, T] \times \mathbb{T}_{\mathbb{G}}, \\
w(0,x)=z_T(x),& \text {in }\mathbb{T}_{\mathbb{G}}.
\end{cases}
\end{equation}
Set $z(t,x)=w(T-t,x)$, then $z(t,x)\in C_{{\mathcal{X}}}^{1,2} \left([0,T)\times\mathbb{T}_{\mathbb{G}}\right)\cap C\left([0,T]\times\mathbb{T}_{\mathbb{G}}\right)$ satisfies the equation \eqref{LHJB}. Suppose $z'(t,x)\in C_{{\mathcal{X}}}^{1,2} \left([0,T)\times\mathbb{T}_{\mathbb{G}}\right)\cap C\left([0,T]\times\mathbb{T}_{\mathbb{G}}\right)$ is another solution to the equation \eqref{LHJB}. Then $z'(T-t,x)=w(t,x)$ by the uniqueness of the solution to the equation \eqref{LHJB forward}. Hence we have $z(t,x)=z'(t,x)$. Consequently, we obtain the existence and uniqueness of the solution to the equation \eqref{LHJB}.

It remains for us to prove the estimates \eqref{LHJB_USE}, \eqref{LHJB_HSE 1} and \eqref{LHJB_HSE 2}. The desired regularity estimates can be obtained by Proposition \ref{Prop_LHJB in R^n USE&HSE} and the fact that $z(t,x)=w(T-t,x)$ for any $(t,x)\in[0,T]\times\mathbb{T}_{\mathbb{G}}$. Hence it is sufficient to prove that $w \in B\left([0,T];C_{\mathcal{X}}^{k+\alpha}\left(\mathbb{T}_{\mathbb{G}}\right)\right)$.

In fact, by Proposition \ref{Prop_Mollifiers_Holder} and Remark \ref{Rem_Mollifiers_Holder}, there exist the mollified versions $b^{\varepsilon}, f^{\varepsilon}$ and $z_T^{\varepsilon}$ of $b, f$ and $z_T$ respectively, satisfying
\begin{gather*}
\sup_{t\in[0,T]}\left\|b^{\varepsilon}(t,\cdot)\right\| _{C_{\mathcal{X}}^{k-1+\alpha}\left(\mathbb{T}_{\mathbb{G}};\mathbb{R}^{n_1}\right)} \leq C\sup_{t\in[0,T]}\left\|b(t,\cdot)\right\| _{C_{\mathcal{X}}^{k-1+\alpha}\left(\mathbb{T}_{\mathbb{G}};\mathbb{R}^{n_1}\right)}, \\
\sup_{t\in[0,T]}\left\|f^{\varepsilon}(t,\cdot)\right\| _{C_{\mathcal{X}}^{k-1+\alpha}\left(\mathbb{T}_{\mathbb{G}}\right)} \leq C\sup_{t\in[0,T]}\left\|f(t,\cdot)\right\| _{C_{\mathcal{X}}^{k-1+\alpha}\left(\mathbb{T}_{\mathbb{G}}\right)}, \\
\left\|z_T^{\varepsilon}\right\|_{C_{\mathcal{X}}^{k+\alpha}\left(\mathbb{T}_{\mathbb{G}}\right)} \leq C\left\|z_T\right\|_{C_{\mathcal{X}}^{k+\alpha}\left(\mathbb{T}_{\mathbb{G}}\right)},\,\varepsilon\in(0,1],
\end{gather*}
where the constants $C>0$ are independent of $\varepsilon$. In addition, when $\varepsilon \to 0$, we have $\left\|b^{\varepsilon}-b\right\|_{L^{\infty}\left([0,T]\times\mathbb{T}_{\mathbb{G}}\right)}\to 0$, $\left\|f^{\varepsilon}-f\right\|_{L^{\infty}\left([0,T]\times\mathbb{T}_{\mathbb{G}}\right)}\to 0$ and $\left\|z_T^{\varepsilon}-z_T\right\|_{L^{\infty}\left(\mathbb{T}_{\mathbb{G}}\right)}\to 0$.

For every $\varepsilon\in(0,1]$, define the operator
$$
H_{\varepsilon}:=\partial_t-\Delta_{\mathcal{X}}+b^{\varepsilon}\cdot D_{\mathcal{X}}.
$$
From Proposition \ref{Prop_LHJB well-posedness} and Corollary \ref{Cor_LHJB existence and uniqueness}, we obtain that there exists a unique solution $w^{\varepsilon}\in C_{{\mathcal{X}}}^{1,2} \left([0,T]\times\mathbb{T}_{\mathbb{G}}\right)$ to the equation
\begin{equation*}
\begin{cases}
H_{\varepsilon} w^{\varepsilon}(t,x)=f^{\varepsilon}(t,x),& \text {in }(0, T] \times \mathbb{T}_{\mathbb{G}}, \\
w^{\varepsilon}(0,x)=z_T^{\varepsilon}(x),& \text {in }\mathbb{T}_{\mathbb{G}}.
\end{cases}
\end{equation*}
It is easy to find that $w^{\varepsilon}\in B\left([0,T];C_{\mathcal{X}}^{1+\alpha}\left(\mathbb{T}_{\mathbb{G}}\right)\right)$. By using \eqref{LHJB in R^n SE}, we know that $w^{\varepsilon}\in C_{\mathcal{X}}^{\frac{\alpha}{2},k+\alpha}\left([0,T]\times\mathbb{T}_{\mathbb{G}}\right)$, thus $w^{\varepsilon}\in B\left([0,T];C_{\mathcal{X}}^{k+\alpha}\left(\mathbb{T}_{\mathbb{G}}\right)\right)$. According to Proposition \ref{Prop_LHJB in R^n USE&HSE}, we have $w^{\varepsilon}$ satisfies for any $k\in\mathbb{Z}_+$,
\begin{equation}\label{w_varepsilon USE}
\begin{aligned}
\sup_{t \in [0,T]} \left\|w^{\varepsilon}(t, \cdot)\right\|_{C_{\mathcal{X}}^{k+\alpha}\left(\mathbb{T}_{\mathbb{G}}\right)} \leq & C\left(\sup _{t \in (0,T)} \left\|f^{\varepsilon}(t,\cdot)\right\|_{C_{\mathcal{X}}^{k-1+\alpha}\left(\mathbb{T}_{\mathbb{G}}\right)} +\left\|z_T^{\varepsilon}\right\|_{C_{\mathcal{X}}^{k+\alpha}\left(\mathbb{T}_{\mathbb{G}}\right)}\right) \\
\leq & C\left(\sup _{t \in [0,T]} \left\|f(t,\cdot)\right\|_{C_{\mathcal{X}}^{k-1+\alpha}\left(\mathbb{T}_{\mathbb{G}}\right)} +\left\|z_T\right\|_{C_{\mathcal{X}}^{k+\alpha}\left(\mathbb{T}_{\mathbb{G}}\right)}\right),
\end{aligned}
\end{equation}
where $C>0$ is dependent of $\mathbb{G}$, $\alpha$, $k$, $T$ and $\sup_{t\in[0,T]}\|b(t,\cdot)\|_{C_{\mathcal{X}}^{k-1+\alpha}\left(\mathbb{R}^{n};\mathbb{R}^{n_1}\right)}$ only. Hence, %thanks to \eqref{d_cc VS Euclidean}, 
by using the Arzel\`{a}-Ascoli theorem, we can find a sequence $\left\{\varepsilon_j\right\}_{j=1}^{+\infty}$ with $\varepsilon_j\to0$ as $j\to+\infty$, and a function $v(t,\cdot) \in C_{\mathcal{X}}^{k}\left(\mathbb{T}_{\mathbb{G}}\right)$ such that $X_I w^{\varepsilon_j}(t,x) \to X_I v(t,x)$ uniformly with respect to $x\in\mathbb{T}_{\mathbb{G}}$ for any $t\in[0,T]$ and multi-index $I$ with the length $|I|\leq k$. Moreover, following \eqref{w_varepsilon USE} and the definition of H\"{o}lder norms, one can find that $v \in B\left([0,T];C_{\mathcal{X}}^{k+\alpha}\left(\mathbb{T}_{\mathbb{G}}\right)\right)$.

Due to Proposition \ref{Prop_LHJB well-posedness}, we have that for any $(t,x)\in[0,T]\times\mathbb{T}_{\mathbb{G}}$,
\begin{equation*}
w^{\varepsilon_j}(t,x)= \int_{\mathbb{T}_{\mathbb{G}}}\sum_{k\in\mathbb{Z}^{n}}\Gamma(t,k\circ x;0,y)z_T^{\varepsilon_j}(y)d y +\int_{0}^{t}\int_{\mathbb{T}_{\mathbb{G}}}\sum_{k\in\mathbb{Z}^{n}}\Gamma(t,k\circ x;s,y)f^{\varepsilon_j}(s,y)d y d s.
\end{equation*}
Letting $j\to +\infty$ and using the dominated convergence theorem, we obtain that for any $(t,x)\in[0,T]\times\mathbb{T}_{\mathbb{G}}$,
\begin{align}\label{lim w_varepsilon=w}
v(t,x)= & \int_{\mathbb{T}_{\mathbb{G}}}\sum_{k\in\mathbb{Z}^{n}}\Gamma(t,k\circ x;0,y)z_T(y)d y +\int_{0}^{t}\int_{\mathbb{T}_{\mathbb{G}}}\sum_{k\in\mathbb{Z}^{n}}\Gamma(t,k\circ x;s,y)f(s,y)d y d s \\
= & w(t,x)\notag.
\end{align}
Therefore, we obtain that for any $k\in\mathbb{Z}_+$, $w=v \in B\left([0,T];C_{\mathcal{X}}^{k+\alpha}\left(\mathbb{T}_{\mathbb{G}}\right)\right)$ when $b\in B\left([0,T];C_{\mathcal{X}}^{k-1+\alpha}\left(\mathbb{T}_{\mathbb{G}};\mathbb{R}^{n_1}\right)\right)$, $f\in B\left([0,T];C_{\mathcal{X}}^{k-1+\alpha}\left(\mathbb{T}_{\mathbb{G}}\right)\right)$ and $z_T(x)\in C_{\mathcal{X}}^{k+\alpha}\left(\mathbb{T}_{\mathbb{G}}\right)$.

So far we have completed the proof.
\end{proof}
In the end we come to the proof of Theorem \ref{Thm_LHJB Lipschitz}.
\begin{proof}[\bf{Proof of Theorem \ref{Thm_LHJB Lipschitz}}]
From \eqref{d_cc VS Euclidean} we have
$$
d_{cc}(x,y) \leq C |x-y|^{\frac{1}{r}},\,x,y\in K,
$$
for any compact set $K \subset \mathbb{R}^n$, which implies that $z_T$ is continuous on $\mathbb{R}^n$. From Proposition \ref{Prop_LHJB well-posedness}, we employ a time transformation to obtain that there exists a unique solution $z \in C_{\mathcal{X}}^{1,2} \left([0, T) \times \mathbb{T}_{\mathbb{G}} \right) \cap C\left([0, T] \times \mathbb{T}_{\mathbb{G}}\right)$ to the equation \eqref{LHJB}.

In order to prove \eqref{LHJB_Holder&Lipschitz}, by Proposition \ref{Prop_Mollifiers_Lip.}, there exist the mollified versions $z_T^{\varepsilon}$ of $z_T$, satisfying
\begin{equation*}
\left[z_T^{\varepsilon}\right] _{C_{\mathcal{X}}^{0+1}\left(\mathbb{T}_{\mathbb{G}}\right)} \leq C\left[z_T\right] _{C_{\mathcal{X}}^{0+1}\left(\mathbb{T}_{\mathbb{G}}\right)},\,\varepsilon>0
\end{equation*}
and $\left\|z_T^{\varepsilon}-z_T\right\|_{L^{\infty}\left(\mathbb{T}_{\mathbb{G}}\right)}\to 0$ as $\varepsilon \to 0$. According to Proposition \ref{Prop_LHJB well-posedness} and Corollary \ref{Cor_LHJB existence and uniqueness}, we obtain that for every $\varepsilon>0$, there exists a unique solution $z^{\varepsilon} \in C_{{\mathcal{X}}}^{1,2} \left([0,T]\times\mathbb{T}_{\mathbb{G}}\right)$ to the equation \eqref{LHJB} with the terminal condition $z(T, x)=z_T^{\varepsilon}(x)$.

The Lie derivative of the function $z_T^{\varepsilon}(x)$ can be expressed as (see \cite[P. 4]{22BB})
\begin{equation*}
X_{i}z_T^{\varepsilon}(x)=\frac{d}{d\tau}\bigg|_{\tau=0}z_T^{\varepsilon}(\exp(\tau X_{i})(x))=\lim_{\tau\to 0}\frac{z_T^{\varepsilon}(\exp(\tau X_{i})(x))-z_T^{\varepsilon}(x)}{\tau},\,i\in\{1,\ldots,n_1\}.
\end{equation*}
Here, $\exp(\tau X_{i})(x)$ denotes the exponential map of the vector field $X_{i}$, which is the solution to the ordinary differential equation $\gamma'(\tau)=X_{i}(\gamma(\tau))$ with initial condition $\gamma(0)=x\in\mathbb{R}^n$. Since $\gamma$ is $\mathcal{X}$-subunit (see Definition \ref{Def_CC distance}), we have
$$
\left|z_T^{\varepsilon}(\exp(\tau X_{i})(x))-z_T^{\varepsilon}(x)\right| \leq \left[z_T^{\varepsilon}\right]_{C_{\mathcal{X}}^{0+1}\left(\mathbb{T}_{\mathbb{G}}\right)} d_{cc}(\gamma(\tau), \gamma(0)) \leq \left[z_T^{\varepsilon}\right]_{C_{\mathcal{X}}^{0+1}\left(\mathbb{T}_{\mathbb{G}}\right)}|\tau|.
$$
Thus
\begin{equation}\label{D_X z_T^varepsilon leq Lip. z_T}
\left\|D_{\mathcal{X}} z_T^{\varepsilon}\right\|_{L^{\infty}\left(\mathbb{T}_{\mathbb{G}}\right)} \leq C\left[z_T^{\varepsilon}\right]_{C_{\mathcal{X}}^{0+1}\left(\mathbb{T}_{\mathbb{G}}\right)} \leq C\left[z_T\right]_{C_{\mathcal{X}}^{0+1}\left(\mathbb{T}_{\mathbb{G}}\right)},
\end{equation}
where the constant $C>0$ is independent of $\varepsilon$.

Since $z^{\varepsilon} \in B\left([0,T];C_{\mathcal{X}}^{1+\alpha}\left(\mathbb{T}_{\mathbb{G}}\right)\right)$, according to Proposition \ref{Prop_LHJB in R^n UE&HE} and \eqref{D_X z_T^varepsilon leq Lip. z_T}, we can obtain
\begin{align}\label{z^n Holder&Lipschitz}
& \sup _{\substack{t \neq t'\\t,t' \in [0,T]}} \frac{\left\|z^{\varepsilon}\left(t^{\prime}, \cdot\right)-z^{\varepsilon}(t, \cdot)\right\|_{L^{\infty}\left(\mathbb{T}_{\mathbb{G}}\right)}}{\left|t^{\prime}-t\right|^{\frac{1}{2}}} +\sup_{t \in[0,T]}\|z^{\varepsilon}(t, \cdot)\|_{C_{\mathcal{X}}^{1}\left(\mathbb{T}_{\mathbb{G}}\right)} \\
\leq & C\left(\left\|z_T^{\varepsilon}\right\|_{C_{\mathcal{X}}^{1}\left(\mathbb{T}_{\mathbb{G}}\right)} +\left\|z_T^{\varepsilon}\right\|_{C_{\mathcal{X}}^{0+1}\left(\mathbb{T}_{\mathbb{G}}\right)} +\|f\|_{L^{\infty}\left((0,T)\times\mathbb{T}_{\mathbb{G}}\right)}\right)\notag \\
\leq & C\left(\left\|z_T\right\|_{C_{\mathcal{X}}^{0+1}\left(\mathbb{T}_{\mathbb{G}}\right)} +\|f\|_{L^{\infty}\left((0,T)\times\mathbb{T}_{\mathbb{G}}\right)}\right),\notag
\end{align}
where the constant $C>0$ depends on $\mathbb{G}$, $T$ and $\|b\|_{L^{\infty}\left((0,T)\times\mathbb{T}_{\mathbb{G}}\right)}$ only.

Using the Arzel\`{a}-Ascoli theorem, we can find a sequence $\left\{\varepsilon_j\right\}_{j=1}^{+\infty}$ with $\varepsilon_j\to0$ as $j\to+\infty$, and a function $v\in C\left([0,T]\times\mathbb{T}_{\mathbb{G}}\right)$ such that $z^{\varepsilon_j}(t,x) \to  v(t,x)$ uniformly on $[0,T]\times\mathbb{T}_{\mathbb{G}}$. Following the same method as in \eqref{lim w_varepsilon=w}, we can obtain $v(t,x)=z(t,x)$ on $[0,T] \times \mathbb{T}_{\mathbb{G}}$.

Moreover, we obtain from \cite[Proposition 4.2(ii)]{07BB} and \eqref{z^n Holder&Lipschitz} that, for any $t,t' \in [0,T]$ and $x,y \in \mathbb{R}^n$,
\begin{align*}
\left|z^{\varepsilon_j}(t,x)-z^{\varepsilon_j}(t,y)\right| \leq & \left\|D_{\mathcal{X}}z^{\varepsilon_j}\right\|_{L^{\infty}\left(\mathbb{T}_{\mathbb{G}}\right)} d_{cc}(x,y) \\
\leq & C\left(\left\|z_T\right\|_{C_{\mathcal{X}}^{0+1}\left(\mathbb{T}_{\mathbb{G}}\right)} +\|f\|_{L^{\infty}\left((0,T)\times\mathbb{T}_{\mathbb{G}}\right)}\right) d_{cc}(x,y),
\end{align*}
and
\begin{align*}
\left|z^{\varepsilon_j}(t,x)-z^{\varepsilon_j}(t',x)\right| \leq C\left(\left\|z_T\right\|_{C_{\mathcal{X}}^{0+1}\left(\mathbb{T}_{\mathbb{G}}\right)} +\|f\|_{L^{\infty}\left((0,T)\times\mathbb{T}_{\mathbb{G}}\right)}\right)|t-t'|^{\frac{1}{2}},
\end{align*}
where the constant $C>0$ depends on $\mathbb{G}$, $T$ and $\|b\|_{L^{\infty}\left((0,T)\times\mathbb{T}_{\mathbb{G}}\right)}$ only. Letting $j \to +\infty$ in the above inequalities, we finally get that \eqref{LHJB_Holder&Lipschitz} holds.
\end{proof}

\section{Results for the degenerate FPK equation}\label{Sec_5}

In this section, we investigate the degenerate FPK equation \eqref{general FPK} and prove Theorem \ref{Thm_general FPK wellposed regularity}.

We start by recalling the definition of the weak solution to the FPK equation \eqref{general FPK} in Definition \ref{Def_general FPK weak sol.}. Noting that the definition is well-posed. In fact, by Theorem \ref{Thm_LHJB wellposed regularity}, equation \eqref{dual FPK} has a unique solution $z \in C_{\mathcal{X}}^{1,2}\left([0,t)\times\mathbb{T}_{\mathbb{G}}\right)\cap C\left([0,t]\times\mathbb{T}_{\mathbb{G}}\right)$, satisfying $z\in B\left([0,t];C_{\mathcal{X}}^{k+\alpha}(\mathbb{T}_{\mathbb{G}})\right)$, thus $\left\langle\rho_0,z(0,\cdot)\right\rangle$ and $\left\langle \upsilon(s),z(s,\cdot)\right\rangle$ are well defined.

We now prove the existence, uniqueness and regularity of the weak solution to the FPK equation \eqref{general FPK}.
\begin{proof}[\bf{Proof of Theorem \ref{Thm_general FPK wellposed regularity}}]
\emph{Step 1: Existence.} We begin by assuming that
\begin{gather*}
b \in B\left([0,T];C_{\mathcal{X}}^{1+\alpha}\left(\mathbb{T}_{\mathbb{G}}\right)\right),\\
\upsilon \in C\left([0,T]\times\mathbb{T}_{\mathbb{G}}\right)\cap B\left([0,T];C_{\mathcal{X}}^{\alpha}\left(\mathbb{T}_{\mathbb{G}}\right)\right),\, \rho_0 \in C(\mathbb{T}_{\mathbb{G}}),
\end{gather*}
and proving \eqref{general FPK sol. regularity}.

In this case, splitting the divergence terms in \eqref{general FPK} and using Proposition \ref{Prop_LHJB well-posedness}, we obtain that there exists a unique solution $\rho\in C_{{\mathcal{X}}}^{1,2} \left([0,T)\times\mathbb{T}_{\mathbb{G}}\right)\cap C\left([0,T]\times\mathbb{T}_{\mathbb{G}}\right)$ to equation \eqref{general FPK}. Let $z$ be the unique solution to equation \eqref{dual FPK} with $f=0$ and $\xi \in C_{\mathcal{X}}^{k+\alpha}(\mathbb{T}_{\mathbb{G}})$. By multiplying the equation of $\rho$ for $z$ and integrating by parts in $[0,t] \times \mathbb{T}_{\mathbb{G}}$, we get that for any $t \in [0,T]$,
\begin{equation}\label{weak formula}
\left\langle \rho(t),\xi \right\rangle = \left\langle \rho_0,z(0,\cdot) \right\rangle + \int_{0}^{t}\left\langle \upsilon(s),z(s,\cdot) \right\rangle ds.
\end{equation}
It follows from \eqref{LHJB_USE} that
$$
\sup_{s \in [0,T]} \left\|z(s,\cdot)\right\|_{C_{\mathcal{X}}^{k+\alpha}(\mathbb{T}_{\mathbb{G}})} \leq C\left\|\xi\right\|_{C_{\mathcal{X}}^{k+\alpha}(\mathbb{T}_{\mathbb{G}})},
$$
where $C>0$ depends on $\mathbb{G}$, $\alpha$, $k$, $T$ and $\sup_{t\in(0,T)}\|b(t,\cdot)\|_{C_{\mathcal{X}}^{k-1+\alpha}\left(\mathbb{T}_{\mathbb{G}};\mathbb{R}^{n_1}\right)}$ only. Hence, the right hand side of \eqref{weak formula} satisfies
\begin{align*}
& \left\langle \rho_0,z(0,\cdot) \right\rangle + \int_{0}^{t}\left\langle \upsilon(s),z(s,\cdot) \right\rangle ds \\
\leq & C\left\|\xi\right\|_{C_{\mathcal{X}}^{k+\alpha}(\mathbb{T}_{\mathbb{G}})} \left(\left\|\rho_0\right\|_{C_{\mathcal{X}}^{-(k+\alpha)}\left(\mathbb{T}_{\mathbb{G}}\right)} +\int_{0}^{t}\left\|\upsilon(s)\right\|_{C_{\mathcal{X}}^{-(k+\alpha)}\left(\mathbb{T}_{\mathbb{G}}\right)}ds\right).
\end{align*}
Taking the supremum for $\xi \in C_{\mathcal{X}}^{k+\alpha}(\mathbb{T}_{\mathbb{G}})$ with $\left\|\xi\right\|_{C_{\mathcal{X}}^{k+\alpha}(\mathbb{T}_{\mathbb{G}})} \leq 1$ for \eqref{weak formula}, we obtain
\begin{equation}\label{general FPK sol. regularity_regular case}
\sup_{t \in [0,T]}\left\|\rho(t)\right\|_{C_{\mathcal{X}}^{-(k+\alpha)}(\mathbb{T}_{\mathbb{G}})} \leq C\left(\left\|\rho_0\right\|_{C_{\mathcal{X}}^{-(k+\alpha)}(\mathbb{T}_{\mathbb{G}})} + \left\|\upsilon\right\|_{L^1([0,T];C_{\mathcal{X}}^{-(k+\alpha)}(\mathbb{T}_{\mathbb{G}}))}\right),
\end{equation}
where the constant $C>0$ depends on $\mathbb{G}$, $\alpha$, $k$, $T$ and $\sup_{t\in(0,T)}\|b(t,\cdot)\|_{C_{\mathcal{X}}^{k-1+\alpha}\left(\mathbb{T}_{\mathbb{G}};\mathbb{R}^{n_1}\right)}$ only.

In the general case, by Proposition \ref{Prop_Mollifiers_Holder}, Remark \ref{Rem_Mollifiers_Holder} and Proposition \ref{Prop_Mollifiers_dual space}, we can consider the mollified versions $b^{\varepsilon}$, $\rho_0^\varepsilon$, $\upsilon^{\varepsilon}$ converging to $b$, $\rho_0$, $\upsilon$ respectively in $C \left([0,T] \times \mathbb{T}_{\mathbb{G}}\right)$, $C_{\mathcal{X}}^{-(k+\alpha)}(\mathbb{T}_{\mathbb{G}})$ and $L^1([0,T];C_{\mathcal{X}}^{-(k+\alpha)}(\mathbb{T}_{\mathbb{G}}))$, satisfying
\begin{gather*}
\|b^{\varepsilon}\|_{C_{\mathcal{X}}^{\frac{\alpha}{2},k-1+\alpha}\left([0,T]\times\mathbb{T}_{\mathbb{G}}\right)} \leq C\|b\|_{C_{\mathcal{X}}^{\frac{\alpha}{2},k-1+\alpha}\left([0,T]\times\mathbb{T}_{\mathbb{G}}\right)}, \\
\left\|\rho_0^\varepsilon\right\|_{C_{\mathcal{X}}^{-k}(\mathbb{T}_{\mathbb{G}})} \leq C\left\|\rho_0\right\|_{C_{\mathcal{X}}^{-k}\left([0,1)^n\right)}, \\ \left\|\upsilon^{\varepsilon}\right\|_{L^1\left([0,T];C_{\mathcal{X}}^{-k}(\mathbb{T}_{\mathbb{G}})\right)} \leq C\left\|\upsilon\right\|_{L^1\left([0,T];C_{\mathcal{X}}^{-k}\left([0,1)^n\right)\right)},\,\varepsilon\in(0,1],
\end{gather*}
where the constant $C>0$ is independent of $\varepsilon$. Moreover, we can use the Arzel\`{a}-Ascoli theorem to find that $b^{\varepsilon} \to b$ up to a subsequence uniformly in $C_{\mathcal{X}}^{0,k-1}\left([0,T]\times\mathbb{T}_{\mathbb{G}}\right)$.

We denote $\rho^{\varepsilon}\in C_{{\mathcal{X}}}^{1,2} \left([0,T]\times\mathbb{T}_{\mathbb{G}}\right)$ as the corresponding unique solution to equation \eqref{general FPK}. Set $\rho^{\varepsilon_1,\varepsilon_2}:=\rho^{\varepsilon_1}-\rho^{\varepsilon_2}$ for any $\varepsilon_1,\varepsilon_2\in(0,1]$. Then $\rho^{\varepsilon_1,\varepsilon_2}$ satisfies equation \eqref{general FPK} with $b=b^{\varepsilon_1}$, $\upsilon=\upsilon^{\varepsilon_1}-\upsilon^{\varepsilon_2}+\operatorname{div}_{\mathcal{X}}\left(\rho^{\varepsilon_2}(b^{\varepsilon_1}-b^{\varepsilon_2})\right)$, $\rho_0=\rho_0^{\varepsilon_1}-\rho_0^{\varepsilon_2}$. From \eqref{general FPK sol. regularity_regular case}, we have
\begin{align}\label{Cauchy seq. regularity}
& \sup_{t \in [0,T]}\left\|\rho^{\varepsilon_1,\varepsilon_2}(t)\right\|_{C_{\mathcal{X}}^{-(k+\alpha)}(\mathbb{T}_{\mathbb{G}})} \notag\\
\leq & C\left(\left\|\rho_0^{\varepsilon_1}-\rho_0^{\varepsilon_2}\right\|_{C_{\mathcal{X}}^{-(k+\alpha)}(\mathbb{T}_{\mathbb{G}})} + \left\|\upsilon^{\varepsilon_1}-\upsilon^{\varepsilon_2}\right\|_{L^1\left([0,T];C_{\mathcal{X}}^{-(k+\alpha)}(\mathbb{T}_{\mathbb{G}})\right)}\right.\\
& \left. + \left\|\operatorname{div}_{\mathcal{X}}(\rho^{\varepsilon_2}(b^{\varepsilon_1}-b^{\varepsilon_2}))\right\|_{L^1\left([0,T];C_{\mathcal{X}}^{-(k+\alpha)}(\mathbb{T}_{\mathbb{G}})\right)}\right). \notag
\end{align}

Below we claim the following fact:
\begin{equation*}
\int_{0}^{T}\int_{\mathbb{T}_{\mathbb{G}}} \rho^{\varepsilon_2}f dxds \leq  C\sup_{t \in(0, T)}\|f(t, \cdot)\|_{C_{\mathcal{X}}^{k-1}\left(\mathbb{T}_{\mathbb{G}}\right)}\left(\left\|\rho_0\right\|_{C_{\mathcal{X}}^{-k}\left([0,1)^n\right)} + \left\|\upsilon\right\|_{L^1\left([0,T];C_{\mathcal{X}}^{-k}\left([0,1)^n\right)\right)}\right)
\end{equation*}
for any $f\in B\left([0,T];C_{\mathcal{X}}^{k-1+\alpha}\left(\mathbb{T}_{\mathbb{G}}\right)\right)$, where the constant $C$ is independent of $\varepsilon_2$.

In fact, consider the unique solution $z^{\varepsilon_2}$ to equation \eqref{dual FPK} with $t=T$, $b=b^{\varepsilon_2}$, $\xi=0$ and any $f\in B\left([0,T];C_{\mathcal{X}}^{k-1+\alpha}\left(\mathbb{T}_{\mathbb{G}}\right)\right)$, by Theorem \ref{Thm_LHJB wellposed regularity} and \eqref{LHJB in R^n_prior C^k estimate} we obtain
\begin{equation}\label{phi_1}
\sup_{t\in[0,T]}\left\|z^{\varepsilon_2}(t,\cdot)\right\|_{C_{\mathcal{X}}^{k}\left(\mathbb{T}_{\mathbb{G}}\right)} \leq C\sup_{t \in(0, T)}\|f(t, \cdot)\|_{C_{\mathcal{X}}^{k-1}\left(\mathbb{T}_{\mathbb{G}}\right)},
\end{equation}
where $C>0$ depends on $\mathbb{G}$, $k$, $T$ and $\sup_{t\in(0,T)}\|b^{\varepsilon_2}(t,\cdot)\|_{C_{\mathcal{X}}^{k-1}\left(\mathbb{R}^{n};\mathbb{R}^{n_1}\right)}$ only. Multiplying the equation of $\rho^{\varepsilon_2}$ for $z^{\varepsilon_2}$ and integrating by parts in $[0,t] \times \mathbb{T}_{\mathbb{G}}$, one has
$$
\int_{0}^{T} \int_{\mathbb{T}_{\mathbb{G}}} \rho^{\varepsilon_2}f dxds=\left\langle\rho_{0}^{\varepsilon_2},z^{\varepsilon_2}(0, \cdot)\right\rangle+\int_{0}^{T}\left\langle \upsilon^{\varepsilon_2}(s),z^{\varepsilon_2}(s,\cdot)\right\rangle ds.
$$
Combining the above with \eqref{phi_1}, we obtain
\begin{equation*}
\begin{aligned}
\int_{0}^{T}\int_{\mathbb{T}_{\mathbb{G}}} \rho^{\varepsilon_2}f dxds \leq & C\sup_{t\in[0,T]}\left\|z^{\varepsilon_2}(t,\cdot)\right\|_{C_{\mathcal{X}}^{k}(\mathbb{T}_{\mathbb{G}})} \left(\left\|\rho_0^{\varepsilon_2}\right\|_{C_{\mathcal{X}}^{-k}(\mathbb{T}_{\mathbb{G}})} + \left\|\upsilon^{\varepsilon_2}\right\|_{L^1\left([0,T];C_{\mathcal{X}}^{-k}(\mathbb{T}_{\mathbb{G}})\right)}\right) \\
\leq & C\sup_{t \in(0, T)}\|f(t, \cdot)\|_{C_{\mathcal{X}}^{k-1}\left(\mathbb{T}_{\mathbb{G}}\right)}\left(\left\|\rho_0\right\|_{C_{\mathcal{X}}^{-k}\left([0,1)^n\right)} + \left\|\upsilon\right\|_{L^1\left([0,T];C_{\mathcal{X}}^{-k}\left([0,1)^n\right)\right)}\right),
\end{aligned}
\end{equation*}
where the constant $C>0$ is independent of $\varepsilon_2$.

Hence, by integration by parts and Fatou's lemma, we have
\begin{align}\label{div_mathcal X(rho(b1-b2))}
& \left\|\operatorname{div}_{\mathcal{X}}(\rho^{\varepsilon_2}(b^{\varepsilon_1}-b^{\varepsilon_2}))\right\|_{L^1\left([0,T];C_{\mathcal{X}}^{-(k+\alpha)}(\mathbb{T}_{\mathbb{G}})\right)} \\
= & \int_{0}^{T}\sup_{\left\|w\right\|_{C_{\mathcal{X}}^{k+\alpha}(\mathbb{T}_{\mathbb{G}})\leq 1}} \left(-\int_{\mathbb{T}_{\mathbb{G}}}\rho^{\varepsilon_2}\left(b^{\varepsilon_1}-b^{\varepsilon_2}\right)\cdot D_{\mathcal{X}}w dx\right)dt \notag\\
\leq & \sup_{\left\|w\right\|_{C_{\mathcal{X}}^{k+\alpha}(\mathbb{T}_{\mathbb{G}})\leq 1}} \left|\int_{0}^{T}\int_{\mathbb{T}_{\mathbb{G}}}\rho^{\varepsilon_2}\left(b^{\varepsilon_2}-b^{\varepsilon_1}\right)\cdot D_{\mathcal{X}}w dx dt\right| \notag\\
\leq & C\left\|b^{\varepsilon_1}-b^{\varepsilon_2}\right\|_{C_{\mathcal{X}}^{0,k-1}\left([0,T]\times\mathbb{T}_{\mathbb{G}}\right)}, \notag
\end{align}
where the constant $C>0$ is independent of $\varepsilon_1$ and $\varepsilon_2$.

Therefore, the right hand side of \eqref{Cauchy seq. regularity} tends to $0$ when $\varepsilon_1,\varepsilon_2 \to 0$, that is, for any $\left\{\varepsilon_j\right\}_{j=1}^{+\infty}$ with $\varepsilon_j\to0$ as $j\to+\infty$, $\left\{\rho^{\varepsilon_j}\right\}_{j=1}^{+\infty}$ is a Cauchy sequence. According to the completeness and Cauchy criterion for uniform convergence, there exists a $\rho \in C\left([0,T];C_\mathcal{X}^{-(k+\alpha)}(\mathbb{T}_{\mathbb{G}})\right)$ such that $\rho^{\varepsilon_j} \to \rho$ in $C\left([0,T];C_\mathcal{X}^{-(k+\alpha)}(\mathbb{T}_{\mathbb{G}})\right)$. Moreover, since it follows from \eqref{general FPK sol. regularity_regular case} that
\begin{equation*}%\label{approx. sol. regularity}
\sup_{t \in [0,T]}\left\|\rho^{\varepsilon_j}(t)\right\|_{C_{\mathcal{X}}^{-(k+\alpha)}(\mathbb{T}_{\mathbb{G}})}\\
\leq C\left(\left\|\rho_0^{\varepsilon_j}\right\|_{C_{\mathcal{X}}^{-(k+\alpha)}(\mathbb{T}_{\mathbb{G}})}+ \left\|\upsilon^{\varepsilon_j}\right\|_{L^1\left([0,T];C_{\mathcal{X}}^{-(k+\alpha)}(\mathbb{T}_{\mathbb{G}})\right)}\right),
\end{equation*}
where $C>0$ is independent of $\varepsilon_j$ due to the fact that $\sup_{t\in[0,T]}\|b^{\varepsilon_j}(t,\cdot)\|_{C_{\mathcal{X}}^{k-1+\alpha}\left(\mathbb{T}_{\mathbb{G}};\mathbb{R}^{n_1}\right)}$ is uniformly bounded. Letting $j\to+\infty$ in the above inequality, we obtain that $\rho$ satisfies \eqref{general FPK sol. regularity}.

We next prove that $\rho$ is a solution to equation \eqref{general FPK} in the sense of Definition \ref{Def_general FPK weak sol.}. Let $z$ and $z^{\varepsilon_j}$ be the solutions to equation \eqref{dual FPK} associated with $b$ and $b^{\varepsilon_j}$ respectively. The weak formulation of $\rho^{\varepsilon_j}$ is as follows:
\begin{equation}\label{rho_varepsilon_j WF}
\left\langle\rho^{\varepsilon_j}(t),\xi\right\rangle+\int_{0}^{t} \left\langle \rho^{\varepsilon_j}(s),f(s,\cdot)\right\rangle ds=\left\langle\rho_0^{\varepsilon_j},z^{\varepsilon_j}(0,\cdot)\right\rangle+\int_{0}^{t}\left\langle \upsilon^{\varepsilon_j}(s),z^{\varepsilon_j}(s,\cdot)\right\rangle ds.
\end{equation}
%We only need to show that $z^{\varepsilon_j}$ converges to $z$, since the desired conclusion can be obtained by directly taking the limit for both sides of the equality as above. Actually, 
For any $j \in \mathbb{Z}_+$, the function $\bar{z}_j:=z^{\varepsilon_j}-z$ satisfies
\begin{equation*}
\begin{cases}
-\partial_t\bar{z}_j-\Delta_\mathcal{X}\bar{z}_j+b^{\varepsilon_j}\cdot D_{\mathcal{X}}\bar{z}_j=-(b^{\varepsilon_j}-b)\cdot D_{\mathcal{X}}z, & \text{ in } [0,t) \times \mathbb{T}_{\mathbb{G}},\\
\bar{z}_j(t)=0, & \text{ in } \mathbb{T}_{\mathbb{G}}.
\end{cases}
\end{equation*}
From Theorem \ref{Thm_LHJB wellposed regularity} and \eqref{LHJB in R^n_prior C^k estimate}, we have
\begin{align*}
\sup_{t\in[0,T]}\left\|\bar{z}_j(t,\cdot)\right\|_{C_{\mathcal{X}}^{k}\left(\mathbb{T}_{\mathbb{G}}\right)} \leq & C\sup_{t \in(0, T)}\|D_{\mathcal{X}}z(t, \cdot)\|_{C_{\mathcal{X}}^{k-1}\left(\mathbb{T}_{\mathbb{G}}\right)}\left\|b^{\varepsilon_j}-b\right\|_{C_{\mathcal{X}}^{0,k-1}\left([0,T]\times\mathbb{T}_{\mathbb{G}}\right)} \to 0
\end{align*}
as $j\to+\infty$. Then we can obtain
\begin{align*}
& \left|\left\langle\rho_0^{\varepsilon_j},z^{\varepsilon_j}(0,\cdot)\right\rangle -\left\langle\rho_0,z(0,\cdot)\right\rangle\right| \\
\leq & \left|\left\langle\rho_0^{\varepsilon_j},\bar{z}_j(0,\cdot)\right\rangle\right| +\left|\left\langle\rho_0^{\varepsilon_j}-\rho_0,z(0,\cdot)\right\rangle\right| \\
\leq & \left\|\rho_0\right\|_{C_{\mathcal{X}}^{-k}\left([0,1)^n\right)} \left\|\bar{z}_j(0,\cdot)\right\|_{C_{\mathcal{X}}^{k}\left(\mathbb{T}_{\mathbb{G}}\right)} +\left\|\rho_0^{\varepsilon_j}-\rho_0\right\|_{C_{\mathcal{X}}^{-(k+\alpha)}\left(\mathbb{T}_{\mathbb{G}}\right)} \left\|z(0,\cdot)\right\|_{C_{\mathcal{X}}^{k+\alpha}\left(\mathbb{T}_{\mathbb{G}}\right)} \to 0
\end{align*}
as $j\to+\infty$. Similarly, there is
$$
\int_{0}^{t}\left\langle \upsilon^{\varepsilon_j}(s),z^{\varepsilon_j}(s,\cdot)\right\rangle ds \to \int_{0}^{t}\left\langle \upsilon(s),z(s,\cdot)\right\rangle ds\,\text{ as }j\to+\infty.
$$
Therefore, letting $j\to+\infty$ in \eqref{rho_varepsilon_j WF}, we get that $\rho$ is a weak solution to equation \eqref{general FPK}.

\emph{Step 2: Uniqueness.} Let $\rho_1,\rho_2$ be two weak solutions to equation \eqref{general FPK}. Then $\bar{\rho}:=\rho_1-\rho_2$ is a weak solution of
\begin{equation*}
\begin{cases}
\partial_t\bar{\rho}-\Delta_\mathcal{X}\bar{\rho}-\operatorname{div}_{\mathcal{X}}(\bar{\rho} b)=0, & \text{ in } [0,T] \times \mathbb{T}_{\mathbb{G}}, \\
\bar{\rho}(0)=0, & \text{ in } \mathbb{T}_{\mathbb{G}}.
\end{cases}
\end{equation*}
The weak formulation implies that for any $f \in C\left([0,t]\times\mathbb{T}_{\mathbb{G}}\right)\cap B\left([0,t];C_\mathcal{X}^{k+\alpha}(\mathbb{T}_{\mathbb{G}})\right)$ and $\xi \in C_{\mathcal{X}}^{k+\alpha}(\mathbb{T}_{\mathbb{G}})$,
$$
\left\langle\bar{\rho}(t),\xi\right\rangle+\int_{0}^{t} \left\langle \bar{\rho}(s),f(s,\cdot) \right\rangle ds=0,
$$
which leads to
$$
\sup_{t \in [0,T]}\left\|\bar{\rho}(t)\right\|_{C_{\mathcal{X}}^{-(k+\alpha)}(\mathbb{T}_{\mathbb{G}})}=0.
$$
Thus, the uniqueness is established.

\emph{Step 3: Stability.} For any $i,j \in \mathbb{Z}_+$, set $\bar{\rho}^{i,j}:=\rho^{i}-\rho^{\varepsilon_j}$, where $\left\{\rho^{\varepsilon_j}\right\}_{j=1}^{+\infty}$ is the sequence of functions defined in Step 1. Then $\bar{\rho}^{i,j}$ satisfies equation \eqref{general FPK} with $b$, $\rho_0$ and $\upsilon$ replaced by $b^{i}$, $\rho_{0}^{i}-\rho_0^{\varepsilon_j}$ and $\upsilon^{i}-\upsilon^{\varepsilon_j}+\operatorname{div}_{\mathcal{X}}(\rho^{\varepsilon_j}(b^{i}-b^{\varepsilon_j}))$. Using \eqref{general FPK sol. regularity} and \eqref{div_mathcal X(rho(b1-b2))}, we have
\begin{align*}
& \sup_{t \in [0,T]}\left\|\bar{\rho}^{i,j}(t)\right\|_{C_{\mathcal{X}}^{-(k+\alpha)}(\mathbb{T}_{\mathbb{G}})} \\
\leq & C\left(\left\|\rho_{0}^{i}-\rho_0^{\varepsilon_j}\right\|_{C_{\mathcal{X}}^{-(k+\alpha)}(\mathbb{T}_{\mathbb{G}})} + \left\|\upsilon^{i}-\upsilon^{\varepsilon_j}\right\|_{L^1\left([0,T];C_{\mathcal{X}}^{-(k+\alpha)}\left(\mathbb{T}_{\mathbb{G}}\right)\right)} +\left\|b^{i}-b^{\varepsilon_j}\right\|_{C_{\mathcal{X}}^{0,k-1}\left([0,T]\times\mathbb{T}_{\mathbb{G}}\right)}\right) \\
\leq & C\left(\left\|\rho_{0}^{i}-\rho_0\right\|_{C_{\mathcal{X}}^{-(k+\alpha)}(\mathbb{T}_{\mathbb{G}})} +\left\|\rho_{0}-\rho_0^{\varepsilon_j}\right\|_{C_{\mathcal{X}}^{-(k+\alpha)}(\mathbb{T}_{\mathbb{G}})} +\left\|\upsilon^{i}-\upsilon\right\|_{L^1\left([0,T];C_{\mathcal{X}}^{-(k+\alpha)}(\mathbb{T}_{\mathbb{G}})\right)}\right.\\
& \left. +\left\|\upsilon-\upsilon^{\varepsilon_j}\right\|_{L^1\left([0,T];C_{\mathcal{X}}^{-(k+\alpha)}\left(\mathbb{T}_{\mathbb{G}}\right)\right)} +\left\|b^{i}-b\right\|_{C_{\mathcal{X}}^{0,k-1}\left([0,T]\times\mathbb{T}_{\mathbb{G}}\right)} +\left\|b-b^{\varepsilon_j}\right\|_{C_{\mathcal{X}}^{0,k-1}\left([0,T]\times\mathbb{T}_{\mathbb{G}}\right)}\right),
\end{align*}
where the constants $C>0$ depends on $\mathbb{G}$, $\alpha$, $k$, $T$ and $\sup_{t\in(0,T)}\|b^{i}(t,\cdot)\|_{C_{\mathcal{X}}^{k-1+\alpha}\left(\mathbb{T}_{\mathbb{G}};\mathbb{R}^{n_1}\right)}$ only, and we note that
$$
\sup_{t\in(0,T)}\|b^{i}(t,\cdot)\|_{C_{\mathcal{X}}^{k-1+\alpha}\left(\mathbb{T}_{\mathbb{G}};\mathbb{R}^{n_1}\right)} \leq \left\|b^{i}-b\right\|_{C_{\mathcal{X}}^{\frac{\alpha}{2},k-1+\alpha}\left([0,T]\times\mathbb{T}_{\mathbb{G}}\right)} +\sup_{t\in(0,T)}\|b(t,\cdot)\|_{C_{\mathcal{X}}^{k-1+\alpha}\left(\mathbb{T}_{\mathbb{G}};\mathbb{R}^{n_1}\right)}.
$$
Then we can get that
\begin{align*}
& \sup_{t \in [0,T]}\left\|(\rho^{i}-\rho)(t)\right\|_{C_{\mathcal{X}}^{-(k+\alpha)}(\mathbb{T}_{\mathbb{G}})} \\
\leq & \sup_{t \in [0,T]}\left\|\bar{\rho}^{i,j}(t)\right\|_{C_{\mathcal{X}}^{-(k+\alpha)}(\mathbb{T}_{\mathbb{G}})} +\sup_{t \in [0,T]}\left\|(\rho^{\varepsilon_j}-\rho)(t)\right\|_{C_{\mathcal{X}}^{-(k+\alpha)}(\mathbb{T}_{\mathbb{G}})}\to 0 \text{ as }i,j\to+\infty,
\end{align*}
thus $\rho^{i} \to \rho$ in $C\left([0,T];C_\mathcal{X}^{-(k+\alpha)}(\mathbb{T}_{\mathbb{G}})\right)$ as $i \to +\infty$.

This concludes the proof of the theorem.
\end{proof}

\section*{Acknowledgments}
Yiming Jiang is supported by National Natural Science Foundation of China (Grant No. 12471141). Yawei Wei is supported by National Natural Science Foundation of China (Grant No. 12271269) and Fundamental Research Funds for the Central Universities. All authors contributed equally to this work, and the author list is ordered alphabetically by surname. The authors are grateful to the referees for their careful reading and thoughtful comments.

%\appendix
%\numberwithin{equation}{section}


\vspace{0.4cm}
\begin{thebibliography}{00}

%\bibitem{02Al} G.K. Alexopoulos, Sub-laplacians with drift on lie groups of polynomial volume growth, Mem. Amer. Math. Soc., 2002, 155(739), 101.
\bibitem{14BGL} D. Bakry, I. Gentil, M. Ledoux, Analysis and Geometry of Markov Diffusion Operators, Grundlehren der mathematischen Wissenschaften, vol. 348, Springer International Publishing, 2014.
%\bibitem{19BCP} V. Bally, L. Caramellino, P. Pigato, Tube estimates for diffusions under a local strong H\"{o}rmander condition, Ann. Inst. H. Poincar\'{e} Probab. Statist., 2019, 55(4)2320-2369.
%\bibitem{24BM} M. Bansil, A.R. M\'{e}sz\'{a}ros, Hidden monotonicity and canonical transformations for mean field games and master equations, 2024, arXiv preprint arXiv: 2403.05426.
%\bibitem{23BMM} M. Bansil, A.R. M\'{e}sz\'{a}ros, C. Mou, Global Well-Posedness of Displacement Monotone Degenerate Mean Field Games Master Equations, 2023, arXiv preprint arXiv: 2308.16167.
%\bibitem{19BCCD} E. Bayraktar, A. Cecchin, A. Cohen, F. Delarue, Finite state mean field games with Wright-Fisher common noise, Journal de Math\'{e}matiques Pures et Appliqu\'{e}es, 2019.
%\bibitem{15BFY} A. Bensoussan, J. Frehse, S. C. P. Yam, The master equation in mean field theory, J. Math. Pures et Appliqu\'{e}es., 2015, 103(6): 1441-1474.
%\bibitem{17BFY} A. Bensoussan, J. Frehse, S. C. P. Yam, On the interpretation of the master equation, Stoc. Proc. App., 2017, 127(7): 2093-2137.
%\bibitem{21Be_C} C. Bertucci, Monotone solutions for mean field games master equations: continuous state space and common noise, 2021, arXiv preprint arXiv: 2107.09531.
%\bibitem{21Be_F} C. Bertucci, Monotone solutions for mean field games master equations: finite state space and optimal stopping, Journal de l'\'{E}cole polytechnique-Math\'{e}matiques, 2021, 8, 1099-1132.
\bibitem{02BLU} A. Bonfiglioli, E. Lanconelli, F. Uguzzoni, Uniform Gaussian estimates of the fundamental solutions for heat operators on Carnot groups, Adv. Differential Equations, 2002, 7(10):1153-1192.
\bibitem{03BLU} A. Bonfiglioli, E. Lanconelli, F. Uguzzoni, Fundamental solutions for non-divergence form operators on stratified groups, Trans. Amer. Math. Soc., 2003, 356, no. 7, 2709-2737.
\bibitem{07BLU} A. Bonfiglioli, E. Lanconelli, F. Uguzzoni, Stratified Lie groups and potential theory for their sub-Laplacians, Springer, Berlin, Heidelberg, 2007.
\bibitem{14Br} M. Bramanti, An invitation to hypoelliptic operators and H\"{o}rmander's vector fields, Springer, 2014.
%\bibitem{05BB} M. Bramanti, L. Brandolini, Estimates of BMO type for singular integrals on spaces of homogeneous type and applications to hypoelliptic PDEs, Rev. Mat. Iberoamericana, 2005, 21(2)511-556.
\bibitem{07BB} M. Bramanti, L. Brandolini, Schauder estimates for parabolic nondivergence operators of H\"{o}rmander type, J. Differ. Equ., 2007, 234, 177-245.
\bibitem{22BB} M. Bramanti, L. Brandolini, H\"{o}rmander Operators, World Scientific, 2022.
\bibitem{10BBLU} M. Bramanti, L. Brandolini, E. Lanconeli, F. Uguzzoni, Non-divergence equations stnuctured on H\"{o}rmander vector fields: heat kernels and Harnack inequalities, Mem. Amer. Math. Soc., 2010, 204(961).
%\bibitem{11BZ} M. Bramanti, M. Zhu, $L^p$ and Schauder estimates for nonvariational operators structured on H\"{o}rmander vector fields with drift, 2011, arXiv preprint arXiv: 1103.5116.
%\bibitem{17BLPR} R. Buckdahn, J. Li, S. Peng, C. Rainer, Mean-field stochastic differential equations and associated PDEs, Ann. Probab., 2017, 45(2): 824-878.
\bibitem{13CCM} L. Capogna, G. Citti, M. Manfredini, Uniform Gaussian bounds for subelliptic heat kernels and an application to the total variation flow of graphs over Carnot groups, Anal. Geom. Metr. Spaces, 2013, 1, 255-275.
\bibitem{03CH} L. Capogna, Q. Han, Pointwise Schauder estimates for second order linear equations in Carnot groups, in: Harmonic Analysis at Mount Holyoke, South Hadley, MA, 2001, 45-69, in: Contemp. Math., Amer. Math. Soc., Providence, RI, 2003, 320: 45-69.
\bibitem{19CDLL} P. Cardaliaguet, F. Delarue, J. Lasry, P. Lions, The Master Equation and the Convergence Problem in Mean Field Games, Ann. Math. Stud., 201, Princeton University Press., 2019.
%\bibitem{15CGPT} P. Cardaliaguet, P. J. Graber, A. Porretta, D. Tonon, Second order mean field games with degenerate diffusion and local coupling, Nonlinear Differ. Equ. Appl. NoDEA, 2015, 22: 12871317.
%\bibitem{22CSS} P. Cardaliaguet, B. Seeger, P. Souganidis, Mean field games with common noise and degenerate idiosyncratic noise, 2022, arXiv preprint arXiv: 2207.10209.
%\bibitem{21CS} P. Cardaliaguet, P.E. Souganidis, Weak solutions of the master equation for Mean Field Games with no idiosyncratic noise, 2021, arXiv preprint arXiv: 2109.14911.
%\bibitem{14CD} R. Carmona, F. Delarue, The Master Equation for large population equilibriums, In: D. Crisan, B. Hambly, T. Zariphopoulou, eds, Stochastic Analysis and Applications, France: Springer, 2014.
\bibitem{18CD} R. Carmona, F. Delarue, Probabilistic Theory of Mean Field Games with Applications I: Mean Field FBSDEs, Control, and Games, 2018.
%\bibitem{14CCD} J. Chassagneux, D. Crisan, F. Delarue, A Probabilistic Approach to Classical Solutions of the Master Equation for Large Population Equilibria, Memoirs of the American Mathematical Society, 2014.
\bibitem{22CCL} H. Chen, H. G. Chen, J. N. Li, Upper bound estimates of eigenvalues for H\"{o}rmander operators on non-equiregular sub-Riemannian manifolds, J. Math. Pures Appl., 2022, 164, 180-212.
%\bibitem{18DF} F. Dragoni, E. Feleqi, Ergodic Mean Field Games with H\"{o}rmander diffusions, Calc. Var. Partial Differential Equations, 2018, 57: 1-22.
\bibitem{10Ev} L. C. Evans, Partial Differential Equations, 2nd ed., Graduate studies in mathematics, no. 19, American Mathematical Society, Providence, R.I, 2010.
%\bibitem{20FGT} E. Feleqi, D. Gomes, T. Tada, Hypoelliptic mean field games--a case study, Minimax Theory Appl., 2020, 5(2): 305-326.
%\bibitem{21FGT} R. Ferreira, D. Gomes, T. Tada, Existence of weak solutions to time-dependent mean-field games, Nonlinear Anal., 2021, 212: 112470.
\bibitem{75Fo} G. B. Folland, Subelliptic estimates and function spaces on nilpotent Lie groups. Ark. Mat. 13 (1975) 161-207.
\bibitem{22GMMZ} W. Gangbo, A. R. M\'{e}sz\'{a}ros, C. Mou, and J. Zhang. Mean field games master equations with nonseparable Hamiltonians and displacement monotonicity. Ann. Probab., 2022, 50(6):2178-2217.
%\bibitem{15GS} W. Gangbo, A. \'{S}wi\c{e}ch, Existence of a solution to an equation arising from the theory of mean field games, J. Differ. Equ., 2015, 259(11): 6573-6643.
%\bibitem{77Ga} B. Gaveau, Principe de moindre action, propagation de la chaleur et estim\'{e}es sous-elliptiques sur certains groupes nilpotents, Acta Math., 1977, 139, 95-153.
\bibitem{01GT} D. Gilbarg, N. S. Trudinger, Elliptic Partial Differential Equations of Second Order, 2nd ed., rev. 3rd printing. ed., Springer, Berlin, 2001.
%\bibitem{09GL} C. E. Guti\'{e}rrez, E. Lanconelli, Schauder estimates for sub-elliptic equations, J. Evol. Equ. 9 (2009) 707-726.
\bibitem{67Ho} L. H\"{o}rmander, Hypoelliptic second order differential equations, Acta Math., 1967, 119, 147-171.
%\bibitem{06HCM} M. Huang, P. E. Caines, R. P. Malham\'{e}, Large population stochastic dynamic games: closed-loop McKean-Vlasov systems and the Nash certainty equivalence principle, Comm. Inf. Syst., 2006, 6: 221-251.
%\bibitem{76Hu} A. Hulanicki, The distribution of energy in the Brownian motion in the Gaussian field and analytic hypoellipticity of certain subelliptic operators on the Heisenberg group, Studia Math., 1976, 56, 165-173.
%\bibitem{23JR} E. R. Jakobsen, A. Rutkowski, The master equation for mean field game systems with fractional and nonlocal diffusions, 2023, arXiv preprint arXiv: 2305.18867.
\bibitem{25JWY} Y. Jiang, Y. Wei, Y. Yang, Schauder estimates for Cauchy problems on Carnot groups with rough coefficients, J. Differ. Equ., 2025, 440(2): 113448.
%\bibitem{01LP} E. Lanconelli, A. Pascucci, On the fundamental solution for hypoelliptic second order partial differential equations with non-negative characteristic form, Ricerche Mat., 2001, 48, no. 1, 81-106.
%\bibitem{06LL_I} J.-M. Lasry, P.-L. Lions, Jeux \`{a} champ moyen. I. Le cas stationnaire, C. R. Math. Acad. Sci. Paris., 2006, 343(9): 619-625.
%\bibitem{06LL_II} J.-M. Lasry, P.-L. Lions, Jeux \`{a} champ moyen. II. Horizon fini et contr\^{o} le optimal, C. R. Math. Acad. Sci. Paris., 2006, 343:679-684.
%\bibitem{07LL} J.-M. Lasry, P.-L. Lions, Mean field games, Jpn. J. Math., 2007, 2(1):229-260.
%\bibitem{08LLG} J.-M. Lasry, P.-L. Lions, O. Gu\`{e}ant, Application of mean field games to growth theory, 2008, hal-00348376.
%\bibitem{09Li} H. Q. Li, Fonctions maximales centr¨¦es de Hardy-Littlewood sur les groupes de Heisenberg, Studia Math., 2009, 191, 89-100.
%\bibitem{Li} P.-L. Lions, Cours au Coll\`{e}ge de France, Available at: www.college-de-france.fr.
%\bibitem{22LS} H. Q. Li, P. Sjgren, Estimates for Operators Related to the Sub-Laplacian with Drift in Heisenberg Groups, Journal of Fourier Analysis and Applications., 2022, 28(1), 1-29.
%\bibitem{03Lu} F. Lust-Piquard, A simple-minded computation of heat kernels on Heisenberg groups, Colloq. Math., 2003, 97, 233-249.
\bibitem{24MMM} P. Mannucci, C. Marchi, C. Mendico, Semi-linear parabolic equations on homogenous Lie groups arising from mean field games, Math. Ann., 2024, 390, 3077-3108.
\bibitem{21MMT} P. Mannucci, C. Marchi, N. Tchou, First order periodic Mean Field Games in Heisenberg group, 2021, arXiv preprint arXiv: 2010.09279.
%\bibitem{21Mi} N. Mimikos-Stamatopoulos, Weak and renormalized solutions to a hypoelliptic Mean Field Games system, 2021, arXiv preprint arXiv: 2105.05777.
%\bibitem{06Mo} R. Montgomery, A Tour of Subriemannian Geometries, Their Geodesics and Applications, 2006.
\bibitem{04RR} M. Renardy, R. C. Rogers, An Introduction to Partial Differential Equations, 2nd ed., Texts in Applied Mathematics, no. 13, Springer, New York, 2004.
\bibitem{21Ri} M. Ricciardi, The Master Equation in a bounded domain with Neumann conditions, Communications in Partial Differential Equations, 2021, 47, 912-947.
%\bibitem{18Ry} L. Ryzhik, Lecture notes, 2018.
\bibitem{07St} B. Stroffolini, Homogenization of Hamilton-Jacobi equations in Carnot Groups, ESAIM: Control, Optimisation and Calculus of Variations, 2007, 13(1): 107-119.
\bibitem{08St} D. W. Stroock, Partial Differential Equations for Probabilists. Cambridge: Cambridge University Press, 2008.
\bibitem{92Xu} C.-J. Xu, Regularity for quasilinear second-order subelliptic equations, Comm. Pure Appl. Math., 1992, 45(1): 77-96.

\end{thebibliography}
\end{document}